\documentclass[11pt,epsfig,amsfonts]{amsart}

\newtheorem{df}{Definition}[section]
\newtheorem{lm}{Lemma}[section]
\newtheorem{thm}{Theorem}[section]
\newtheorem{pro}{Proposition}[section]
\newtheorem{rmk}{Remark}[section]
\newtheorem{cor}{Corollary}[section]
\numberwithin{equation}{section}

\usepackage{xcolor}
\usepackage[colorlinks,
            linkcolor=red,
            anchorcolor=blue,
            citecolor=green
            ]{hyperref}

\usepackage{epsfig}


\keywords{McKean-Vlasov equation, pullback measure attractor, stochastic delay lattice system, the upper semi-continuity, tail-ends estimate.
\\ This work was supported by NSFC (12090010, 12090013, 12471170, 12071317, 12471154, 12071384). All correspondences should be addressed to Jun Shen.}

\author{ Lin Shi}
\address[Lin Shi]
{ School of Mathematical Sciences \\
University of Electronic Science and Technology of China\\
Chengdu, Sichuan 611731, P. R. China}
\email[L.~Shi]{shilinlavender@163.com}

\author{ Jun Shen}
\address[Jun Shen]
{School of Mathematics \\
Sichuan University\\
Chengdu, Sichuan 610064, P. R. China}
\email[J.~Shen]{junshen85@163.com}

\author{Kening Lu}
\address[Kening Lu]
{School of Mathematics \\
Sichuan University\\
Chengdu, Sichuan 610064, P. R. China}
\email[K.~Lu]{keninglu@scu.edu.cn}

\begin{document}

\begin{abstract}
We study the long-term behavior of the distribution of the solution process to the non-autonomous McKean-Vlasov stochastic delay lattice system defined on the integer set
$\mathbb{Z}$. Specifically, we first establish the well-posedness of solutions for this non-autonomous, distribution-dependent stochastic delay lattice system. Then, we prove the existence and uniqueness of pullback measure attractors for the non-autonomous dynamical system generated by the solution operators, defined in the space of probability measures. Furthermore, as an application of the pullback measure attractor, we prove
the ergodicity and exponentially mixing of invariant measures for the system under appropriate conditions.
Finally, we establish the upper semi-continuity of these attractors as the distribution-dependent stochastic delay lattice system converges to a distribution-independent system.
\end{abstract}

\title[pullback measure attractors of McKean-Vlasov stochastic delay lattice systems]{
pullback measure attractors and limiting behaviors of McKean-Vlasov stochastic delay lattice systems}
\maketitle

\section{introduction}
In this paper, we are concerned with the pullback measure attractors of the following non-autonomous, distribution-dependent stochastic delay lattice system defined on the integer set $\mathbb{Z}$:
\begin{equation} \label{ob-1}
\begin{split}
&\, du_i(t)-\nu(u_{i-1}(t)-2u_i(t)+u_{i+1}(t))dt+\lambda u_i(t)dt
\\
&\,= (f_i(t,u_i(t),u_i(t-r),\mathcal{L}_{u_i(t)})+g_i(t))dt
 +\sigma_i(t, u_i(t),u_i(t-r),\mathcal{L}_{u_i(t)}) dW_i(t),\quad t>\tau
\end{split}
\end{equation}
with initial data
\begin{eqnarray}\label{ob-2}
(u_i)_{\tau}(s)=\zeta_i(s), \quad \forall s \in [-r,0],
\end{eqnarray}
where $u=(u_i)_{i \in \mathbb{Z}}$ is an unknown sequence,
$\zeta=(\zeta_i)_{i \in \mathbb{Z}}$ is a given sequence, and $\nu, \lambda$, $r$ are positive constants. The functions
$f_i$ and $\sigma_i$ are nonlinear mappings for every $i \in \mathbb{Z}$, and
$g(t)=(g_i(t))_{i \in \mathbb{Z}}$ is a given time dependent sequence. The term
$u(t-r)=(u_i(t-r))_{i\in \mathbb{Z}}$ represents the dependence of the solution on the past history of the state. Additionally,
$\mathcal{L}_{u_i}$ denote the distribution of the random variable $u_i$, and $(W_i(t))_{i \in \mathbb{Z}  }$
is a sequence of independent standard two-sided
real-valued Wiener processes on a complete filtered probability space
$(\Omega, \mathcal{F}, \{\mathcal{F}_t\}_{t\in \mathbb{R}}, \mathbb{P})$.

Lattice systems have been extensively studied due to their wide range of applications in fields such as physics, biology, and engineering, including phenomena like pattern formation, propagation of nerve pulses and electric circuits. Over decades of research, a vast number of publications have emerged concerning the solutions and long-term dynamics of lattice systems. For example, studies on deterministic lattice systems can be found in
\cite{Bates99, Bates01, Bates03, Beyn03, Chow98, Gu16, Zhao14, wang06, Zhousheng15}, while stochastic lattice systems are investigated in
\cite{Bates06, Bates14, Caraballo12, Han11, wang19, wang10, wangzhao16}.

Time delays are commonly encountered in practical systems where the current state depends on past states, such as in memory processes and animal growth in biology. Consequently, lattice systems with delays have received significant attention. In recent years, the solutions and long-term dynamics of delay lattice systems have been extensively studied.
For deterministic cases, notable works include  \cite{Caraballo14, Chen10,
Han16, wangye15}, while for stochastic cases, studies such as \cite{wang16, Zhou22, Zhou13, Yan10, Zhang17} have been conducted.

Beyond these contributions, the statistical dynamics of stochastic lattice systems, including delay-dependent cases, have become a focus of research in recent years. For instance, the existence, ergodicity, and mixing of invariant measures have been explored in \cite{Li21, Li22, Chenzhang25, Chenzhang24+, Chenzhang23, wangyu24} and topics like the large deviation principle are addressed in  \cite{wang24, Chenzhang24} , among others.

Unlike previous studies, this work focuses on the non-autonomous distribution-dependent (also known as McKean-Vlasov) stochastic delay lattice system. McKean-Vlasov stochastic differential equations (MVSDEs), first introduced in \cite{McKean1966, Vlasov1968}, often arise from interacting particle systems
\cite{BH2022, DV1995, FG2015, S1991}.
A distinctive feature of these equations is their dependence not only on the states of the solutions but also on their distributions. Consequently, the Markov operators associated with MVSDEs do not form semigroups (see, e.g.,\cite{FWang2018}), making conventional methods for studying distribution-independent stochastic differential equations inapplicable to these systems.

A significant body of research has been devoted to the properties of solutions to MVSDEs. For example, the existence of solutions has been studied in
\cite{AD1995,  FHSY2022, GHL2022, HDS2021,  RZ2021, FWang2018}, while the Bismut formula is discussed in  \cite{Banos2018, RW2019}.
 Averaging principles are addressed in \cite{LWX2022, RSX2021}, and large deviation results are presented in \cite{CW2024, HLL2021, LSZZ2023}.
The existence and ergodicity of invariant measures, as well as periodic measures for finite-dimensional MVSDEs, have been explored in \cite{Bao2022,Hua2024, Hua2024b, Lia2021, FWang2023, FWang2023b, Zha2023} and \cite{Ren2021}, respectively. Moreover, the existence and uniqueness of solutions for MVSDEs with delays in both finite- and infinite-dimensional state spaces have been investigated in works such as \cite{Hein21, Huang2019}.

In this article, we establish a theory of pullback measure attractors for the non-autonomous McKean-Vlasov stochastic delay lattice system \eqref{ob-1}. The concept of measure attractors was first introduced in \cite{Schmalfuss91} for the stochastic Navier-Stokes equations, and since then, the existence of measure attractors for stochastic differential equations has been further studied in \cite{LW2024, MC1998, Morimoto1992,
Schmalfuss1997, SWeng2024}.

It is worth noting that none of the stochastic equations studied in the aforementioned works are distribution-dependent. Given the significance of distribution-dependent equations, we recently investigated invariant measures, periodic measures, and pullback measure attractors for McKean-Vlasov stochastic reaction-diffusion equations on unbounded domains in \cite{SSLW24}.
Despite these advancements, research on this topic remains limited, and many fundamental theories still need to be developed.

To that end, the first goal of this paper is the existence of pullback measure attractors for the McKean-Vlasov stochastic lattice system with delays.
To establish  the existence of pullack measure attractors for equation \eqref{ob-1},
we first define a semigroup $\{P^*_{\tau, t}\}_{t \ge \tau}$ on
the appropriate probability measure space (see \eqref{df-P} for details),
based on the solution segments  of equation \eqref{ob-1}, which is ensured by the weak uniqueness of solutions.
We then demonstrate that $\{P^*_{\tau, t}\}_{t \ge \tau}$ forms a continuous non-autonomous dynamical system on this space. The key challenge in this step is proving the continuity of $\{P^*_{\tau, t}\}_{t \ge \tau}$.
To address this, we apply Vitali theorem to establish the continuity of $\{P^*_{\tau, t}\}_{t \ge \tau}$
 on any closed subset of a suitable subspace, rather than the entire space (see Lemma \ref{cor}).
This necessitates an $L^4$ estimate for the solution segments in $C([-r, 0], \ell^2)$, which is crucial for constructing a $\mathcal{D}$-pullback absorbing set.
 Next, we prove that $\{P^*_{\tau, t}\}_{t \ge \tau}$ is $\mathcal{D}$-pullback asymptotically compact on such subspaces.
A major obstacle here is establishing the tightness of the probability distributions for the solution segments of equation \eqref{ob-1}. Since the stochastic lattice system \eqref{ob-1} with delays
is defined on the full space $\ell^2$, and Sobolev embeddings are not compact on unbounded domains,
the pullback asymptotic compactness of $\{P^*_{\tau, t}\}_{t \ge \tau}$ on the target space does not follow from the uniform estimates
of solutions and the Sobolev embeddings.
To overcome this, we use uniform tail-end estimates for the solution segments and apply the Arzelà-Ascoli theorem to address the non-compactness of Sobolev embeddings in $\ell^2$ (see Lemma \ref{tight}). Moreover, because the lattice system includes delays, the uniform estimates for the solution segments are more intricate compared to cases without delays.

Additionally, as an application of the pullback measure attractor, we prove the ergodicity and exponential mixing of the invariant measures for equation \eqref{ob-1}, under suitable conditions.

The second objective of this paper is to investigate the limiting behavior of pullback measure attractors for
the distribution-dependent stochastic lattice equation \eqref{ob-1} with delays.  Specifically, we consider the following distribution-dependent stochastic delay lattice system \eqref{ob-1} parameterized by $\epsilon \in (0, 1)$:
\begin{equation} \label{ob-1'}
\begin{split}
&\, du^\epsilon_i(t)-\nu(u^\epsilon_{i-1}(t)-2u^\epsilon_i(t)+u^\epsilon_{i+1}(t))dt+\lambda u^\epsilon_i(t)dt
\\
&\,= (f^\epsilon_i(t,u^\epsilon_i(t),u^\epsilon_i(t-r),\mathcal{L}_{u^\epsilon_i(t)})+g_i(t))dt
 +\sigma^\epsilon_i(t, u^\epsilon_i(t),u^\epsilon_i(t-r), \mathcal{L}_{u^\epsilon_i(t)}) dW_i(t), \quad t>\tau.
\end{split}
\end{equation}
We prove that, under suitable conditions, as $\epsilon \to 0$,
the pullback measure attractor of \eqref{ob-1'} is upper semi-continuous with respect to the pullback measure attractor of the corresponding distribution-independent stochastic lattice equation with delays:
\begin{equation} \label{ob-1''}
\begin{split}
&\, du_i(t)-\nu(u_{i-1}(t)-2u_i(t)+u_{i+1}(t))dt+\lambda u_i(t)dt
\\
&=\, (f_i(t,u_i(t),u_i(t-r))+g_i(t))dt
 +\sigma(t, u_i(t),u_i(t-r)) dW_i(t), \quad  t>\tau.
\end{split}
\end{equation}

This paper is organized as follows. In Section 2, we recall some basic concepts and introduce the necessary assumptions. Section 3 is devoted to establishing the well-posedness of solutions for the distribution-dependent stochastic delay lattice system \eqref{ob-1}-\eqref{ob-2}. In Section 4, we derive various uniform estimates of the solutions, including the uniform tail-end estimates.  In Section 5, the existence and uniqueness of pullback measure attractors for equation \eqref{ob-1} are established. In Section 6, we give the conditions of the ergodicity and exponentially mixing of invariant measures for equation \eqref{ob-1}.
Finally, in Section 7, we demonstrate the upper semi-continuity of pullback measure attractors for equation \eqref{ob-1'} as $\epsilon\rightarrow 0^+$.

Throughout this paper, we use $\| \cdot \|$ and $(\cdot, \cdot)$  to denote the norm
and the inner product of $\ell^{\,2}$, respectively, and $\| \cdot\|_{p}$ to denote
the norm  in  $\ell^{\,p}$.

\section{Preliminaries}
In this section, we begin by recalling some fundamental concepts related to pullback measure attractors, followed by introducing the basic assumptions for the lattice system \eqref{ob-1}.

Let $X$ be a separable Banach space with norm  $\|\cdot\|_X$
and Borel $\sigma$-algebra  $\mathcal B(X)$. Denote by $C_b(X)$ the space of bounded continuous functions $\xi: X\rightarrow \mathbb{R}$.
Let $L_b(X)$ be the space of bounded Lipschitz continuous functions $\xi: X\rightarrow \mathbb{R}$, equipped with norm $\|\cdot\|_{L_b(X)}$,
$$
\|\xi \|_{L_b(X)}
=\sup_{x\in X} |\xi (x)|
+    \mathop {\sup }\limits_{x_1 ,x_2  \in X,x_1\neq x_2} \frac{{\left| {\xi \left( {x_1 } \right)}-
{\xi \left( {x_2 } \right)} \right|}}{{\|x_1-x_2\|_X}} .
$$

Let $\mathcal P(X)$
be the space of probability measures on $(X,\mathcal B(X))$.
If $\mu, \mu_n \in \mathcal P(X)$ and for every $\xi \in C_b(X)$, $\int_X \xi (x)  \mu_n (dx) \to \int_X \xi (x) \mu (dx), $
then we say that $\{\mu_n\}_{n=1}^\infty$ converges to $\mu$ weakly.
The weak topology of $\mathcal P(X)$ is metrizable with a metric given by
 $$
  d_{\mathcal{P} (X)}
  (\mu_1, \mu_2)
  =\sup_{\xi \in L_b(X), \ \|\xi \|_{L_b} \le 1}
  \left |\int_X \xi d\mu_1 -\int_X \xi d\mu_2 \right |,\quad
 \ \forall \mu_1,\mu_2\in \mathcal P(X).
$$
It is well known that $(\mathcal P(X),  d_{\mathcal{P} (X)})$ is a polish space.
Moreover, a sequence $\{\mu_n\}_{n=1}^{\infty} \subseteq \mathcal{P}(X)$ converges to $\mu$
in $(\mathcal{P}(X),d_{\mathcal{P}(X)})$ if and only if $\{\mu_n\}_{n=1}^{\infty}$ converges to $\mu$ weakly.
For every $p\ge 1$, let $(\mathcal P_p(X), \mathbb{W}_p^X)$  denote  the
Polish space defined by
$$
\mathcal P_p \left( X \right) = \left\{ {\mu  \in \mathcal
P\left( X \right): {\mu(\|\cdot\|_X^p):=}
\int_X {\|x\|_X^p \mu \left( {dx} \right) <    \infty } } \right\}
$$
and
$$
\mathbb{W}_p^{X} ( \mu , \upsilon ) =
\inf\limits_{ \pi \in \Pi ( \mu, \upsilon ) }
\Big (
\int_{ X \times X}
    \| x-y  \|_X^p \pi (dx, dy)
\Big )^{ \frac{1}{p} },\quad
\forall \mu, \upsilon \in \mathcal{P}_p ( X ),
$$
where $ \Pi ( \mu, \upsilon ) $ is the  set of all
couplings of $\mu$ and $\upsilon$.
The metric $\mathbb{W} _p^{X}$
is called the
Wasserstein distance.

Given  $r>0$, denote by
$$
 B_ {\mathcal P_p(X)} (r)  = \left\{ {\mu  \in \mathcal P_p \left( X \right)
 	:
 	\left(\int_X \|x\|_X^p \mu \left( {dx} \right)\right)^{\frac{1}{p}} \leq r } \right\}.
$$
 A  subset $S\subseteq {\mathcal P}_p \left( X \right)$ is
  bounded  if there is $r>0$ such that
  $S\subseteq B_
{{\mathcal P}_p \left( X \right)
}
(r)$. If   $S$ is bounded in   $ {\mathcal P}_p ( X )$,
then we set
$$
\|S\|_{{\mathcal P}_p ( X )}
=
\sup_{
\mu \in S
} \left(\int_X \|x\|_X^p \mu
 ( {dx}  )\right)^{\frac{1}{p}}.
$$

Note that $(\mathcal P_p(X), \mathbb{W}_p^X)$ is a Polish space, whereas $(\mathcal{P}_p (X), d_{\mathcal{P}(X)})$
is not complete. However,  for every $r>0$, the set $B_ {\mathcal P_p(X)} (r)$ is a closed subset of  $\mathcal{P}(X)$ with respect to the
metric $d_{\mathcal{P}(X)}$. Consequently, the space $(B_ {\mathcal P_p(X)} (r), \ d_{\mathcal{P}(X)})$ is complete for every $r>0$.

Next, we recall the definition of non-autonomous dynamical systems in the metric space
$(\mathcal {P}_p (X), d_{\mathcal{P}(X)})$  along with some related concepts.
\begin{df}
A family  $\Phi=\{\Phi(t,\tau): t\in \mathbb{R}^+,\,\tau\in \mathbb{R}\}$ of mappings from
$(\mathcal {P}_p (X), d_{\mathcal{P}(X)})$
 to $(\mathcal {P}_p (X), d_{\mathcal{P}(X)})$
 is called  a  non-autonomous dynamical system  on
 $(\mathcal {P}_p (X), d_{\mathcal{P}(X)})$
 if  $\Phi$ satisfies: for all $\tau\in \mathbb{R}$ and $t, s\in \mathbb{R}^+$:

\begin{description}
\item[(i)]  $\Phi(0,\tau)=I$,  where $I $ is the identity operator on
  $(\mathcal {P}_p (X), d_{\mathcal{P}(X)})$;

\item[(ii)] $\Phi(t+s,\tau)=\Phi(t,s+\tau)\circ \Phi(s,\tau)$.
\end{description}
\end{df}

If
$\Phi$ is
 a  non-autonomous dynamical system  on
 $(\mathcal {P}_p (X), d_{\mathcal{P}(X)})$,
 and,  for every bounded subset $S \subset\mathcal {P}_p (X)$,
$\Phi$
is continuous from
$ (S,  d_{\mathcal{P}(X)})$
to
$(\mathcal {P}_p (X), d_{\mathcal{P}(X)})$,
then we say that $\Phi$ is continuous on
bounded subsets of
 $\mathcal {P}_p (X)$.

Let $\mathcal{D}$ be a collection of  families of nonempty bounded subsets  $D(\tau) \subseteq \mathcal{P}_p(X)$,
parametrized by $\tau \in \mathbb{R}$. 
Such a collection $\mathcal{D}$ is called inclusion-closed if $D=\{D(\tau): \tau \in \mathbb{R}\} \in \mathcal{D}$ implies that every
family $\tilde{D}=\{\tilde{D}(\tau): \emptyset\neq \tilde{D}(\tau)\subseteq D(\tau), \; \forall \tau \in \mathbb{R}\}$ also belongs to $\mathcal{D}$.
\begin{df}
A family $K=\{K(\tau): \tau \in \mathbb{R}\} \in \mathcal{D}$ is called a $\mathcal{D}$-pullback absorbing set for
$\Phi$ if for each $\tau \in \mathbb{R}$ and $D \in \mathcal{D}$, there exists $T=T(\tau,D)>0$ such that
\begin{eqnarray*}
\Phi(t,\tau-t)D(\tau-t)\subseteq K(\tau), \quad  \forall t \ge T.
\end{eqnarray*}
\end{df}

\begin{df}
A non-autonomous dynamical system $\Phi$ defined on $(\mathcal {P}_p (X), d_{\mathcal{P}(X)})$ is said to
$\mathcal{D}$-pullback asymtotically compact if for each $\tau \in \mathbb{R}$,
$\{\Phi(t_n, \tau-t_n)\mu_n\}_{n=1}^{\infty}$ has a convergent subsequence in $\mathcal{P}_p(X)$ whenever
$t_n \rightarrow \infty$ and $\mu_n \in D(\tau-t_n)$ with $D \in \mathcal{D}$.
\end{df}

\begin{df}
A family $\mathcal{A}=\{\mathcal{A}(\tau): \tau \in \mathbb{R}\} \in \mathcal{D}$ is called a $\mathcal{D}$-pullback measure attractor for
$\Phi$ if the following conditions are satisfied.
\begin{description}
\item[(i)] $\mathcal{A}(\tau)$ is compact in $\mathcal{P}_p(X)$ for each $\tau \in \mathbb{R}$;
\\
\item[(ii)] $\mathcal{A}(\tau)$ is invariant, namely, $\Phi(t,\tau)\mathcal{A}(\tau)=\mathcal{A}(\tau+t)$ for all $\tau \in \mathbb{R}$ and $t \in \mathbb{R}^+$;
\\
\item[(iii)] $\mathcal{A}(\tau)$ attracts every set in $\mathcal{D}$, namely, for each $D=\{D(\tau): \tau \in \mathbb{R}\} \in \mathcal{D}$,
$$\lim_{t \rightarrow \infty} d_{\mathcal{P}_p(X)}(\Phi(t,\tau-t)D(\tau-t),\mathcal{ A}(\tau))=0,$$
where $d_{\mathcal{P}_p(X)}(\cdot, \cdot)$ denotes the Hausdorff semi-metric between any two subsets of $\mathcal{P}_p(X)$.
\end{description}
\end{df}

 Since $(\mathcal {P}_p (X), d_{\mathcal{P}(X)})$ is a metric space,  by Theorem 2.23 (more specifically Proposition 3.6)
 and Theorem 2.25 in \cite{wang12}, we obtain the following criterion for the existence and uniqueness of pullback measure attractors.

\begin{pro}\label{exatt}
	 Let $\mathcal D$ be an
	  inclusion-closed collection of
 families of  nonempty bounded subsets of
 $ \mathcal {P}_p (X),$
 and
  $\Phi$
 be  a  non-autonomous dynamical system
 on
 $(\mathcal {P}_p (X), d_{\mathcal{P}(X)})$.
 Suppose $\Phi$ is continuous on
bounded subsets of
 $\mathcal {P}_p (X)$.
  If  $\Phi$ has a
 closed  $\mathcal D$-pullback  absorbing
set $K\in \mathcal D$ and   is
$\mathcal D$-pullback asymptotically
 compact in
 $(\mathcal {P}_p (X), d_{\mathcal{P}(X)})$,
 then
$ \Phi$   has a unique  $\mathcal D$-pullback
  measure attractor $\mathcal A \in   \mathcal D$
  in $(\mathcal {P}_p (X), d_{\mathcal{P}(X)})$.
 \end{pro}



For the nonlinear mappings $f_i$, $\sigma_i$ and $g_i$ $(i \in \mathbb{Z})$, appearing in equation \eqref{ob-1}, we assume that

\begin{itemize}
\item[\bf{(H1)}] \label{H1}
$f_i: \mathbb{R} \times \mathbb{R} \times \mathbb{R} \times \mathcal{P}_2(\mathbb{R}) \rightarrow \mathbb{R}$ is continuous,
there exist a $\psi(\cdot)=(\psi_i(\cdot))_{i\in \mathbb{Z}} \in L_{loc}^{\infty}(\mathbb{R}, \ell^2)$ such that
for any $t, u, v_1, v_2 \in \mathbb{R}$ and $\nu_1, \nu_2 \in \mathcal{P}_2(\mathbb{R})$,
\begin{eqnarray} \label{fi-Lip}
\begin{split}
|f_i(t,u,v_1,\nu_1)-f_i(t,u,v_2,\nu_2)|  \le &\; \psi_i(t)(|v_1-v_2|+\mathbb{W}_2^{\mathbb{R}}(\nu_1,\nu_2))
\end{split}
\end{eqnarray}
and
\begin{eqnarray}\label{fi-Lip-1}
\varphi(\cdot)=(\varphi_i(\cdot))_{i\in \mathbb{Z}} :=(f_i(\cdot,0,0,\delta_0))_{i\in \mathbb{Z}}\in L_{loc}^\infty(\mathbb{R}, \ell^2).
\end{eqnarray}
For all $t, u, v \in \mathbb{R}, \mu \in \mathcal{P}_2(\mathbb{R})$, 
\begin{eqnarray}\label{fi-1}
f_i(t,u,v,\mu)u \!\!\!&\leq&\!\!\! -\alpha|u|^p+\eta_i(t)(1+|u|^2+|v|^2+\mathbb{W}_2^{\mathbb{R}}(\mu,\delta_0)^2)
\end{eqnarray}
and
\begin{eqnarray}\label{fi-3}
\frac{\partial}{\partial u} f_i(t,u,v,\mu) \!\!\!&\leq&\!\!\! \Theta_i(t),
\end{eqnarray}
where $p\geq 2$,  $\alpha$ is a positive constant, $\eta(\cdot)=(\eta_i(\cdot))_{i\in \mathbb{Z}}\in
L^{\infty}(\mathbb{R},\ell^2)\cap L_{loc}(\mathbb{R},\ell^1)$,
$\Theta(\cdot)=(\Theta_i(\cdot))
 _{i\in \mathbb{Z}}\in L_{loc}^\infty(\mathbb{R}, \ell^{2})$,
and $\delta_0$ is the Dirac measure at point $0 \in \mathbb{R}$;

\item[\bf{(H2)}] \label{H2}
$\sigma_i: \mathbb{R} \times \mathbb{R} \times \mathbb{R} \times \mathcal{P}_2(\mathbb{R}) \rightarrow \mathbb{R}$ is continuous,
there exists a $\chi(\cdot)=(\chi_i(\cdot))_{i\in \mathbb{Z}} \in L^{\infty}(\mathbb{R}, \ell^2)$ such that for any $t, u_1,u_2, v_1, v_2 \in \mathbb{R}$ and $\nu_1, \nu_2 \in \mathcal{P}_2(\mathbb{R})$,
\begin{eqnarray} \label{segamai-Lip}
\begin{split}
|\sigma_i(t, u_1,v_1,\nu_1)-\sigma_i(t, u_2,v_2,\nu_2)|  \le \chi_i(t)(|u_1-u_2|+|v_1-v_2|+\mathbb{W}_2^\mathbb{R}(\nu_1,\nu_2)),
\end{split}
\end{eqnarray}
and
\begin{eqnarray}\label{segamai-Lip-1}
\kappa(\cdot)=(\kappa_i(\cdot))_{i\in \mathbb{Z}}:=(\sigma_i(t,0,0,\delta_0))_{i\in \mathbb{Z}}\in L^\infty(\mathbb{R}, \ell^2);
\end{eqnarray}

\item[\bf{(H3)}] \label{H3}
$g(\cdot)=(g_i(\cdot))_{i\in \mathbb{Z}} \in  L^2_{loc}(\mathbb{R}, \ell^2)$.
\end{itemize}
\vskip0.1in

By a simple observation,  we find that \eqref{fi-Lip}, \eqref{fi-Lip-1} and \eqref{fi-3} imply that for $t,u,v \in \mathbb{R}$ and
$\mu \in \mathcal{P}_2(\mathbb{R})$, there exists a $\gamma(\cdot)=(\gamma_i(\cdot))_{i\in \mathbb{Z}}\in L_{loc}^{\infty}(\mathbb{R}, \ell^2)$ such that for each $i \in \mathbb{Z}$,
\begin{eqnarray} \label{fi-2}
|f_i(t,u,v,\mu)|\!\!\!&\leq&\!\!\! \gamma_i(t)(1+|u|+|v|+\mathbb{W}_2^\mathbb{R}(\mu,\delta_0)).
\end{eqnarray}
Moreover, \eqref{segamai-Lip} and \eqref{segamai-Lip-1} suggest that
\begin{eqnarray} \label{segamai-1}
\begin{split}
|\sigma_i(t, u,v,\mu)|  \le \chi_i(t)(|u|+|v|+\mathbb{W}_2^\mathbb{R}(\mu,\delta_0))+\kappa_i(t).
\end{split}
\end{eqnarray}

Next, we introduce the space
\begin{eqnarray*}
\mathcal{L}^2:=\{(\mu_i)_{i \in \mathbb{Z}}: \mu_i \in \mathcal{P}_2(\mathbb{R}) \ \mbox{for each}\ i \in \mathbb{Z}\ \mbox{and}\
\sum_{i \in \mathbb{Z}}\mathbb{W}_2^{\mathbb{R}}(\mu_i, \delta_0)^2<\infty\},
\end{eqnarray*}
which is equipped with the metric
\begin{eqnarray*}
\rho(\mu,\varpi):= \left\{  \sum_{i \in \mathbb{Z}}\mathbb{W}_2^{\mathbb{R}}(\mu_i,\varpi_i)^2  \right\}^{\frac{1}{2} }
\end{eqnarray*}
for $\mu=(\mu_i)_{i \in \mathbb{Z}}$, $\varpi=(\varpi_i)_{i \in \mathbb{Z}} \in \mathcal{L}^2$.
We have the following.
\begin{lm}
$(\mathcal{L}^2, \rho)$ is a complete metric space.
\end{lm}
\begin{proof}
We first claim that $\rho$ is a metric on $\mathcal{L}^2$. Clearly,  $\rho(\mu,\varpi) \ge 0$ for any
$\mu=(\mu_i)_{i \in \mathbb{Z}}$, $\varpi=(\varpi_i)_{i \in \mathbb{Z}} \in \mathcal{L}^2$,
and $\rho(\mu,\varpi)=0$ if and only if $\mathbb{W}_2^{\mathbb{R}}(\mu_i,\varpi_i)=0$ for all $i \in \mathbb{Z}$.
Since $(\mathcal{P}_2(\mathbb{R}), \mathbb{W}_2^{\mathbb{R}})$
is a Polish space, $\mathbb{W}_2^{\mathbb{R}}(\mu_i, \varpi_i)=0$ for all $i \in \mathbb{Z}$ if and only if $\mu_i=\varpi_i$ for all $i \in \mathbb{Z}$.
Therefore, $\rho(\mu,\varpi)=0$ if and only if $\mu=\varpi$.

The symmetry  is also evident. To prove the triangle inequality, for all $\mu=(\mu_i)_{i \in \mathbb{Z}}$,
$\varpi=(\varpi_i)_{i \in \mathbb{Z}}$, $\upsilon=(\upsilon_i)_{i \in \mathbb{Z}}$, utilizing the Minkowski inequality, we obtain that
\begin{eqnarray*}
\rho(\mu,\varpi) \!\!\!& \le &\!\!\! \left\{\sum_{i \in \mathbb{Z}}[\mathbb{W}_2^{\mathbb{R}}(\mu_i, \varpi_i)+\mathbb{W}_2^{\mathbb{R}}(\varpi_i,\upsilon_i)]^2 \right\}^{\frac{1}{2}}
\\
\!\!\!& \le &\!\!\!
\left\{\sum_{i \in \mathbb{Z}}\mathbb{W}_2^{\mathbb{R}}(\mu_i, \varpi_i)^2  \right\}^{\frac{1}{2}}
+\left\{\sum_{i \in \mathbb{Z}}\mathbb{W}_2^{\mathbb{R}}(\varpi_i,\upsilon_i)^2 \right\}^{\frac{1}{2}}
\\
\!\!\!& = &\!\!\! \rho(\mu,\varpi)+\rho(\varpi,\upsilon).
\end{eqnarray*}
Therefore, our assertion holds. It remains to prove the completeness.

Let $\{\mu_n\}$ be a Cauchy sequence of $(\mathcal{L}^2, \rho)$ with $\mu_n=(\mu_{n,i})_{i \in \mathbb{Z}}$.
Then for any $\epsilon>0$, there exists a constant $N>0$
such that for any $m, n \ge N$,
\begin{eqnarray*}
\rho(\mu_n,\mu_m)=\left\{\sum_{i \in \mathbb{Z}} \mathbb{W}_2^{\mathbb{R}}(\mu_{n,i}, \mu_{m,i})^2 \right\}^{\frac{1}{2}}<\epsilon,
\end{eqnarray*}
which implies that for all $i \in \mathbb{Z}$,
\begin{eqnarray*}
\mathbb{W}_2^{\mathbb{R}}(\mu_{n,i}, \mu_{m,i})<\epsilon.
\end{eqnarray*}
Thus for each fixed $i \in \mathbb{Z}$, $\{\mu_{n,i}\}_{n=1}^{\infty}$ is a Cauchy sequence in $\mathcal{P}_2(\mathbb{R})$.
Since $(\mathcal{P}_2(\mathbb{R}), \mathbb{W}_2^{\mathbb{R}})$ is a complete metric space, there exists a $\mu_i \in \mathcal{P}_2(\mathbb{R})$
such that
\begin{eqnarray*}
\lim_{n\rightarrow \infty} \mathbb{W}_2^{\mathbb{R}}(\mu_{n,i}, \mu_i)=0
\end{eqnarray*}
for each $i \in \mathbb{Z}$. Set $\mu=(\mu_i)_{i \in \mathbb{Z}}$. Notice that for any $m,n \ge N$ and $k \in \mathbb{N}$,
\begin{eqnarray*}
\sum_{|i| \le k} \mathbb{W}_2^{\mathbb{R}}(\mu_{n,i}, \mu_{m,i})^2 < \epsilon^2.
\end{eqnarray*}
It then follows that for any $n\ge N$ and $k \in \mathbb{N}$,
\begin{equation}\label{lm-1}
\begin{split}
\sum_{|i| \le k} \mathbb{W}_2^{\mathbb{R}}(\mu_{n,i}, \mu_{i})^2 \le&\, 2 \sum_{|i| \le k} \mathbb{W}_2^{\mathbb{R}}(\mu_{n,i}, \mu_{m,i})^2
+ 2\sum_{|i| \le k} \mathbb{W}_2(\mu_{m,i}, \mu_{i})^2
\\
\le&\, 2\epsilon^2+2\sum_{|i| \le k} \mathbb{W}_2^{\mathbb{R}}(\mu_{m,i}, \mu_{i})^2,
\end{split}
\end{equation}
where $m \ge N$. Letting $m\rightarrow \infty$ in \eqref{lm-1}, we have for $n \ge N$ and $k \in \mathbb{N}$ that
\begin{eqnarray*}
\sum_{|i| \le k} \mathbb{W}_2^{\mathbb{R}}(\mu_{n,i}, \mu_{i})^2 \!\!\!&\le&\!\!\! 2 \epsilon^2.
\end{eqnarray*}
Then, letting $k\rightarrow \infty$, we obtain
$\rho(\mu_n, \mu) \le \sqrt{2} \epsilon$ for all $n \ge N$, which implies that
$\mu_n$ converges to $\mu$ under the metric $\rho$ as $n \rightarrow \infty$, and
\begin{eqnarray*}
\rho(\mu,\hat{\delta}_0)\!\!\!& \le &\!\!\! \rho(\mu,\mu_N)+\rho(\mu_N,\hat{\delta}_0)< \infty,
\end{eqnarray*}
where $\hat{\delta}_0=(\delta_0)_{i \in \mathbb{Z}}$. Thus, we conclude that $\mu \in \mathcal{L}^2$. This completes the proof of the lemma.
\end{proof}

For $u=(u_i)_{i \in \mathbb{Z}}\in L^2(\Omega,\ell^2)$, we denote  $\mathcal{L}_{u}:=(\mathcal{L}_{u_i})_{i\in \mathbb{Z}}$.
Clearly, $\mathcal{L}_{u} \in \mathcal{L}^2$.
Moreover, for any $u=(u_i)_{i \in \mathbb{Z}}$, $v=(v_i)_{i \in \mathbb{Z}} \in L^2(\Omega,\ell^2)$,
the following property holds
\begin{eqnarray} \label{dis-property}
\rho(\mathcal{L}_{u}, \mathcal{L}_{v})^2 \!\!\!& = &\!\!\!
\sum_{i \in \mathbb{Z}} \mathbb{W}_2^\mathbb{R}(\mathcal{L}_{u_i}, \mathcal{L}_{v_i})^2 \le
\sum_{i \in \mathbb{Z}}\mathbb{E}\|u_i-v_i\|^2= \mathbb{E}\|u-v\|^2.
\end{eqnarray}

To express the lattice system \eqref{ob-1}-\eqref{ob-2} as a differential equation in $\ell^2$, we define the following linear operators for every
$u=(u_i)_{i \in \mathbb{Z}} \in \ell^2$:
\begin{eqnarray*}
(Bu)_i:=u_{i+1}-u_i,\ \ \ \ \ \ \ \ (B^*u)_i:=u_{i-1}-u_i
\end{eqnarray*}
and
\begin{eqnarray*}
(Au)_i:=-u_{i-1}+2u_i-u_{i+1}, \quad \forall i \in \mathbb{Z}.
\end{eqnarray*}
Note that $A=BB^*=B^*B$ and
\begin{eqnarray*}
(B^*u,v)=(u,Bv)
\end{eqnarray*}
for all $u,v \in \ell^2$. Moreover, it can be seen  that $A$ is a bounded linear operator from $\ell^2$ to $\ell^2$ with
\begin{eqnarray}\label{bounded-A}
\|Au\|^2 \!\!\! &\le &\!\!\! 18\|u\|^2,\quad \forall u \in \ell^2.
\end{eqnarray}

For every $t \in \mathbb{R}$, $u=(u_i)_{i \in \mathbb{Z}}$, $v=(v_i)_{i \in \mathbb{Z}} \in \ell^2$
and $\mu=(\mu_i)_{i \in \mathbb{Z}}\in \mathcal{L}^2$, we denote
\begin{eqnarray*}
f(t,u,v,\mu)=(f_i(t,u_i, v_i, \mu_i))_{i \in \mathbb{Z}},
\end{eqnarray*}
and define the operator
$\tilde{\sigma}: \mathbb{R} \times \ell^2 \times \ell^2 \times \mathcal{L}^2 \rightarrow L(\ell^2, \ell^2)$ by
\begin{eqnarray*}
\tilde{\sigma}(t, u,v,\mu)w=(\sigma_i(t, u_i, v_i, \mu_i)w_i)_{i \in \mathbb{Z}}
\end{eqnarray*}
for every $u=(u_i)_{i \in \mathbb{Z}}$, $v=(v_i)_{i \in \mathbb{Z}}$, $w=(w_i)_{i \in \mathbb{Z}} \in \ell^2$
and $\mu=(\mu_i)_{i \in \mathbb{Z}}\in \mathcal{L}^2$.
Let $e_k \in \ell^2$ represent the element having $1$ at position $k$ and $0$ in all the other components. Then, the set $\{e_k\}_{k \in \mathbb{Z}}$
forms a complete orthogonal basis of $\ell^2$.

We denote by $L_2(\ell^2,\ell^2)$ the space of Hilbert Schmidt operators
from $\ell^2$ to $\ell^2$ which is equipped with norm
$$\|C\|_{L_2(\ell^2,\ell^2)}:=\Big(\sum_{k\in\mathbb{Z}}\|Ce_k\|^2\Big)^{\frac{1}{2}},\quad  \forall C\in L_2(\ell^2,\ell^2).
$$

 For any $t \in \mathbb{R}$,
$u_1, u_2, v_1, v_2 \in \ell^2$ and $\nu_1, \nu_2 \in \mathcal{L}^2$, by \eqref{fi-Lip}, \eqref{fi-3}
and \eqref{segamai-Lip}, we have that
\begin{equation}\label{f-Lip}
\begin{split}
 &\; \|f(t, u_1,v_1,\nu_1)-f(t, u_2,v_2,\nu_2)\|^2
\\
\le &\;
2\|\Theta(t)\|^2 \|u_1-u_2\|^2
+4 \|\psi(t)\|^2(\|v_1-v_2\|^2+\rho(\nu_1,\nu_2)^2)
\end{split}
\end{equation}
and
\begin{equation}\label{segama-Lip}
\begin{split}
&\, \|\tilde{\sigma}(t,u_1,v_1,\nu_1)-\tilde{\sigma}(t,u_2,v_2,\nu_2)\|_{L_2(\ell^2, \ell^2)}^2
\\
= &\, \sum_{k\in \mathbb{Z}} \|[\tilde{\sigma}(t,u_1,v_1,\nu_1)-\tilde{\sigma}(t,u_2,v_2,\nu_2)]e_k\|^2
\\
\le &\, 3 \|\chi(t)\|^2(\|u_1-u_2\|^2+\|v_1-v_2\|^2+\rho(\nu_1,\nu_2)^2).
\end{split}
\end{equation}
For all $t\in \mathbb{R},
u,v\in\ell^2$ and $\mu \in \mathcal{L}^2$, using \eqref{fi-2} and \eqref{segamai-1}, we have that
\begin{eqnarray}\label{bound-1}
\|f(t, u,v,\mu)\|^2 \!\!\!& \le &\!\!\! 4\|\gamma(t)\|^2(1+\|u\|^2+\|v\|^2+\rho(\mu,\hat{\delta}_0)^2)
\end{eqnarray}
and
\begin{equation}\label{bound-2}
\begin{split}
\|\tilde{\sigma}(t,u,v,\mu)\|_{L_2(\ell^2,\ell^2)}^2
= &\, \sum_{k \in \mathbb{Z}}\|\sigma_k(t, u_k,v_k,\mu_k)\|^2
\\
\le &\, 6\|\chi(t)\|^2(\|u\|^2+\|v\|^2+\rho(\mu,\hat{\delta}_{0})^2)+2\|\kappa(t)\|^2.
\end{split}
\end{equation}

With the aforementioned preparations, we can reformulate \eqref{ob-1}-\eqref{ob-2} as follows:
\begin{equation}\label{ob-3}
\begin{split}
& du(t)+\nu A u(t) dt+\lambda u(t) dt
\\
&= (f(t, u(t),u(t-r), \mathcal{L}_{u(t)})+g(t)) dt+\tilde{\sigma}(t, u(t),u(t-r), \mathcal{L}_{u(t)})dW(t),\quad t >\tau
\end{split}
\end{equation}
with initial condition
\begin{eqnarray} \label{ob-4}
u_\tau(s)=\zeta(s),\quad  \forall s\in [-r,0],
\end{eqnarray}
where $W$ is a cylindrical $Q$-Wiener process with $Q=I$ in $\ell^2$ on a probability space $(\Omega, \mathcal{F}, \mathbb{P})$.
Let $\mathcal{N}$ denote the collection of null sets in the completion of $\mathcal{F}$. We choose the normal filtration
\begin{equation*}
\mathcal{F}_{t}:=\sigma\{W(u): \tau \le u \leq t\}\vee \mathcal{N}, \quad  \forall t \ge \tau.
\end{equation*}

We denote by $\mathcal{C}_r:=C([-r,0], \ell^2)$ the Banach space of all continuous functions
$\varphi$ mapping from $[-r,0]$ to $\ell^2$, equipped with norm
$\|\varphi\|_{\mathcal{C}_r}:=\sup_{-r \le \theta \le 0}\|\varphi(\theta)\|$.
Notably,  $u_t$ represents  an element of $\mathcal{C}_r$ such that $u_t(s)=u(t+s)$ for $t+s>\tau$,
and $u_t(s)=\zeta(t+s-\tau)$ for $\tau-r \le t+s \le \tau$. By the way, for simplicity, in the following, we use the notation
$\mathbb{W}_p(\cdot, \cdot)$ instead of $\mathbb{W}_p^{\mathcal{C}_r}(\cdot, \cdot)$, which refers to the Wasserstein distance on $\mathcal{P}_p(\mathcal{C}_r)$.

The strong solution of the system \eqref{ob-3}-\eqref{ob-4} is understood in the following sense.
\begin{df}\label{df-solu}
Suppose that $\zeta \in L^2(\Omega, \mathcal{C}_r)$ is $\mathcal{F}_\tau$-measurable.
A continuous $\ell^2$-valued stochastic process $u(t)$ with $t \in [\tau-r,\infty)$ is called a strong solution of
\eqref{ob-3} with initial data \eqref{ob-4} if $(u_t)_{t \ge \tau}$ is $\mathcal{F}_t$-adapted, $u_\tau=\zeta$, $u \in L^2(\Omega, C([\tau-r,T],\ell^2))$ for any $T>\tau-r$ and for each $\tau \le t \le T$, $u(t)$ satisfies {\small
\begin{equation} \label{inte-ob-eq}
\begin{split}
&\, u(t)+\nu \int_\tau^t A u(s) ds+\lambda \int_\tau^t u(s) ds
\\
=&\, \zeta(0)+ \int_\tau^t [f(s, u(s), u(s-r), \mathcal{L}_{u(s)})+g(s) ]ds + \int_\tau^t \tilde{\sigma}(s, u(s), u(s-r), \mathcal{L}_{u(s)})dW(s), \quad
\mathbb{P}-a.s.
\end{split}
\end{equation}}
\end{df}
\begin{rmk}
Here, we briefly mention that there exists
$\mathcal{L}_{\zeta(\cdot)} \in C([-r,0],\mathcal{C}_r)$ as long as one assumes that $\zeta \in L^2(\Omega, \mathcal{C}_r)$.
In fact, for any $\zeta \in L^2(\Omega,\mathcal{C}_r)$, we have $\zeta(t) \in L^2(\Omega, \ell^2)$ for all $t \in [-r,0]$. Thus,
$\mathcal{L}_{\zeta(t)} \in \mathcal{L}^2$ for every $t \in [-r,0]$.
Moreover, for any fixed $s \in [-r,0]$,
we have that $\lim_{t\rightarrow s} \|\zeta(t)-\zeta(s)\|^2=0$ for each $\omega \in \Omega$ (Here as $s=-r$ or $0$, $t\rightarrow s$ indicates only a unilateral limit). By \eqref{dis-property}, we observe that
\begin{eqnarray*}
\rho(\mathcal{L}_{\zeta(t)},\mathcal{L}_{\zeta(s)})^2\!\!\!&\le&\!\!\! \mathbb{E}\|\zeta(t)-\zeta(s)\|^2.
\end{eqnarray*}
In view of $\mathbb{E} \|\zeta(t)\|^2 \le \mathbb{E}(\|\zeta\|_{\mathcal{C}_r}^2)<\infty$ for all $t \in [-r,0]$, by using the dominated convergence theorem, we obtain that
\begin{eqnarray*}
\lim_{t \rightarrow s} \rho(\mathcal{L}_{\zeta(t)},\mathcal{L}_{\zeta(s)})^2=0.
\end{eqnarray*}
Hence $\mathcal{L}_{\zeta(\cdot)} \in C([-r,0],\mathcal{L}^2)$ and our assertion is valid. Similarly, $u \in L^2(\Omega, C([\tau-r,T],\ell^2))$
also implies that $\mathcal{L}_{u(\cdot)} \in C([\tau-r,T],\mathcal{L}^2)$.
 \end{rmk}

For any $\zeta \in L^2(\Omega, \mathcal{C}_r)$,
we denote the law of $\zeta$ on $\mathcal{C}_r$ by $\mathcal{L}_\zeta=\zeta_*(\mathbb{P})=\mathbb{P}\circ \zeta^{-1}$.

The weak solution of system \eqref{ob-3} will be understood in the following sense.
\begin{df} \label{df1}
For any $\mu_\tau \in \mathcal{P}_2(\mathcal{C}_r)$, $(\{\tilde{u}(t)\}_{t \ge \tau-r}, \{\tilde{W}(t)\}_{t \ge \tau})$
is called a weak solution to \eqref{ob-3} starting at $\mu_\tau$, if $\{\tilde{W}(t)\}_{t \ge \tau}$
is a cylindrical $Q$-Wiener process with $Q=I$ under a stochastic basis
$(\tilde{\Omega}, \tilde{\mathcal{F}}, \{\tilde{\mathcal{F}_t}\}_{t \ge \tau}, \tilde{\mathbb{P}})$
(i.e. a complete probability space with normal filtration),
and $\{\tilde{u}(t)\}_{t \ge \tau}$ is a $\mathcal{F}_t$-adapted $\ell^2$-valued stochastic process with $\mathcal{L}_{\tilde{u}_\tau}|_{\tilde{\mathbb{P}}}=\mu_\tau$,
$\tilde{u} \in L^2(\Omega, C([\tau-r,T],\ell^2))$ for any $T>\tau-r$ and for each $\tau \le t \le T$,
$\tilde{u}(t)$ satisfies
\begin{equation} \label{inte-ob-eq-weak}
\begin{split}
&\, \tilde{u}(t)+\nu \int_\tau^t A \tilde{u}(s) ds+\lambda \int_\tau^t \tilde{u}(s) ds
\\
=&\, \tilde{u}(\tau)+ \int_\tau^t f(s,\tilde{u}(s), \tilde{u}(s-r), \mathcal{L}_{\tilde{u}(s)})ds + \int_\tau^t \tilde{\sigma}
(s,\tilde{u}(s), \tilde{u}(s-r), \mathcal{L}_{\tilde{u}(s)})d\tilde{W}(s), \tilde{\mathbb{P}}-a.s.
\end{split}
\end{equation}
\end{df}

From now on, we assume that the coefficient $\lambda$ in \eqref{ob-3} is sufficiently
large such that
\begin{eqnarray} \label{lamda}
\lambda> 8\|\eta\|_{L^{\infty}(\mathbb{R},\ell^2)}(5+e^{2\varepsilon r})
+\left[24\|\chi\|^2_{L^{\infty}(\mathbb{R},\ell^2)}(5+e^{2 \varepsilon r})
+16\|\kappa\|^2_{L^{\infty}(\mathbb{R},\ell^2)} \right](3+8c_1^2),
\end{eqnarray}
where $c_1$ follows from the BDG inequality (see Proposition \ref{Pro1}). It follows from \eqref{lamda} that there exists a sufficiently small $\varepsilon \in (0,1)$
such that
\begin{eqnarray}\label{lamda-1}
\qquad \lambda-{4\varepsilon}>8\|\eta\|_{L^{\infty}(\mathbb{R},\ell^2)}(5+e^{2\varepsilon r})
+\left[24\|\chi\|^2_{L^{\infty}(\mathbb{R},\ell^2)}(5+e^{2 \varepsilon r})
+16\|\kappa\|^2_{L^{\infty}(\mathbb{R},\ell^2)}\right](3+8c_1^2).
\end{eqnarray}

Let $D_1 = \{  D(\tau):\tau\in\mathbb{R}, \ D(\tau) \ \text{is a bounded nonempty subset of} \ \mathcal{P}_2 (\mathcal{C}_r)\}$,
and
$ D_2 = \{  D(\tau): \tau\in\mathbb{R}, D(\tau) \ \text{is a bounded nonempty subset of  } \ \mathcal{P}_4 (\mathcal{C}_r)\}$.
Denote by $ \mathcal{D}_0 $ the collection of all such  families  $D_1$  which further satisfy:
$$ \lim\limits_{ \tau \rightarrow - \infty } e^{\varepsilon\tau } \| D ( \tau )  \|_{ \mathcal{P}_2 (\mathcal{C}_r) }^2=0,$$
where $\varepsilon$ is the same number as in \eqref{lamda-1}. Similarly, denote by
$\mathcal{D} $ the collection of all such families $D_2$   which further satisfy:
$$\lim\limits_{ \tau \rightarrow - \infty }e^{ 2\varepsilon  \tau }\| D ( \tau )  \|_{ \mathcal{P}_4 (\mathcal{C}_r) }^4=0.$$
It is evident
$ \mathcal{D} \subseteq   \mathcal{D}_0 $.

We also assume  the following condition in order to derive uniform estimates of solutions:
\begin{equation}\label{g^2}
\int_ {-\infty}^{\tau} e^{\varepsilon (s-\tau)}(\|g(s)\|^2+\|\eta(s)\|_1) ds<\infty
\end{equation}
and
\begin{equation}\label{g^4}
\int_{-\infty}^\tau e^{2 \varepsilon  (s-\tau)} (\|g(s)\|^4+\|\eta(s)\|_1^2) ds<\infty, \ \ \forall \tau \in \mathbb{R}.
\end{equation}

\section{the well-posedness of solutions}

In this section, we will establish the well-posedness of solutions (including both strong and weak solutions) to the non-autonomous distribution-dependent stochastic delay lattice system \eqref{ob-3}-\eqref{ob-4}.

We begin by proving the existence and uniqueness of strong solutions for the stochastic delay lattice system \eqref{ob-3}-\eqref{ob-4} in the following proposition.
\begin{pro} \label{Pro1}
Let $\zeta \in L^2(\Omega, \mathcal{C}_r)$ be an $\mathcal{F}_r$-measurable random variable.
If {\bf (H1)}-{\bf (H3)} hold, then system \eqref{ob-3}-\eqref{ob-4}
has a unique strong solution $u$ in $L^2(\Omega, C([\tau-r,T],\ell^2))$ for any $T>\tau-r$ with the property
$\mathcal{L}_{u_t} \in \mathcal{P}_2(\mathcal{C}_r)$ for all $\tau \le t \le T$. Uniqueness, in this context, means that any other solution $v$ of system \eqref{ob-3}-\eqref{ob-4} is indistinguishable from $u$, i.e.,
\begin{eqnarray*}
\mathbb{P}(\{u(t)= v(t), \forall t \in [\tau-r,T]\})=1.
\end{eqnarray*}
\end{pro}

\begin{proof}
The proof is divided into three steps.

{\bf Step 1.} First, we prove the existence. For all $\tau \le t \le T$, let
$u^{(0)}(t)=\zeta(0)$, $\mathcal{L}_{u^{(0)}(t)}=\mathcal{L}_{\zeta(0)}$ and take $u_\tau^{(0)}=\zeta$.
For any $n \in \mathbb{N}^+$, let $u^{(n)}(t)$ solve the following stochastic integral equation
\begin{equation} \label{pro1-1+}
\begin{split}
&\, u^{(n)}(t)+\nu A \int_\tau^t u^{(n-1)}(s) ds+\lambda \int_\tau^t u^{(n-1)}(s) ds
\\
=&\,\zeta(0)+\int_\tau^t f(s, u^{(n-1)}(s),u^{(n-1)}(s-r), \mathcal{L}_{u^{(n-1)}(s)}) ds+\int_\tau^t g(s) ds
\\
&\, +\int_\tau^t \tilde{\sigma}(s, u^{(n-1)}(s),u^{(n-1)}(s-r), \mathcal{L}_{u^{(n-1)}(s)})dW(s),~\forall  t\in [\tau, T]
\end{split}
\end{equation}
with initial condition $u_\tau^{(n)}(s)=\zeta(s)$ for all $s \in [-r, 0]$.
It is easy to see that
\[u^{(n)} \in L^2(\Omega, C([\tau,T],\ell^2))
\] for each $n \in \mathbb{N}$.
By \eqref{bounded-A}, \eqref{bound-1} and \eqref{bound-2}, we have that for any $n \in \mathbb{N}^+$ and $\tau \le t \le T$,
\begin{eqnarray*}
\!\!\!&&\!\!\! \mathbb{E}(\sup_{\tau \le \varrho \le t} \|u^{(n)}(\varrho)\|^2)
\\
\!\!\!&\le&\!\!\! 6\mathbb{E}(\|\zeta\|_r^2)+6(T-\tau)(18\nu^2+\lambda^2)
\int_\tau^t \mathbb{E}(\sup_{\tau \le \varrho \le s}\|u^{(n-1)}(\varrho)\|^2)ds
\\
\!\!\!&&\!\!\! +24(T-\tau) \|\gamma\|_{L^\infty([\tau,T],\ell^2)}^2 \mathbb{E}\int_\tau^t(1+\|u^{(n-1)}(s)\|^2+\|u^{(n-1)}(s-r)\|^2
+\mathbb{E}\|u^{n-1}(s)\|^2) ds
\\
\!\!\!&&\!\!\! +36 \|\chi\|_{L^\infty(\mathbb{R},\ell^2)}^2 \mathbb{E}\int_\tau^t(\|u^{(n-1)}(s)\|^2+\|u^{(n-1)}(s-r)\|^2
+\mathbb{E}\|u^{n-1}(s)\|^2) ds
\\
\!\!\!&&\!\!\!+12 (T-\tau) \|\kappa\|_{L^\infty(\mathbb{R},\ell^2)}^2+6(T-\tau)\int_\tau^T \|g(s)\|^2 ds
\\
\!\!\!&\le&\!\!\! \left[ 6(18\nu^2+\lambda^2+12\|\gamma\|_{L^\infty([\tau,T],\ell^2)}^2)(T-\tau)
+108\|\chi\|_{L^\infty(\mathbb{R},\ell^2)}^2 \right]
\int_\tau^t \mathbb{E}(\sup_{\tau \le \varrho \le s}\|u^{(n-1)}(\varrho)\|^2)ds
\\
\!\!\!&&\!\!\! +\left[ 6+24r(T-\tau) \|\gamma\|_{L^\infty([\tau,T],\ell^2)}^2+36r\|\chi\|_{L^\infty(\mathbb{R},\ell^2)}^2 \right]
\mathbb{E}(\|\zeta\|^2_{\mathcal{C}_r})
\\
\!\!\!&&\!\!\! +12 \left[ 2\|\gamma\|_{L^\infty([\tau,T],\ell^2)}^2(T-\tau)+\|\kappa\|_{L^\infty(\mathbb{R},\ell^2)}^2
+\frac{1}{2}\int_\tau^T \|g(s)\|^2 ds \right] (T-\tau).
\end{eqnarray*}
This implies that for any $k \in \mathbb{N}^+$ and $\tau \le t \le T$,
\begin{eqnarray*}
\sup_{1 \le n \le k}\{\mathbb{E}(\sup_{\tau \le \varrho \le t} \|u^{(n)}(\varrho)\|^2)\}
\!\!\!&\le&\!\!\! \Delta_2+\Delta_1
\int_\tau^t \sup_{1 \le n \le k} \{\mathbb{E}(\sup_{\tau \le \varrho \le s}\|u^{(n-1)}(\varrho)\|^2)\}ds,
\end{eqnarray*}
where $\Delta_1$ and $\Delta_2$ are positive constants independent of $n$ and $t$. Thus, we have that
\begin{eqnarray*}
\!\!\!&&\!\!\! \sup_{1 \le n \le k}\{\mathbb{E}(\sup_{\tau \le \varrho \le t} \|u^{(n)}(\varrho)\|^2)\}
\\
\!\!\!&\le&\!\!\! \Delta_2+\Delta_1
\int_\tau^t \left(\mathbb{E}(\|\zeta(0)\|^2)+\sup_{1 \le n \le k}\{\mathbb{E}(\sup_{\tau \le \varrho \le s}\|u^{(n)}(\varrho)\|^2)\} \right)ds
\\
\!\!\!&\le&\!\!\!  \Delta_2+\Delta_1 (T-\tau) \mathbb{E}\|\zeta\|_{\mathcal{C}_r}^2+
\Delta_1\int_\tau^t \sup_{1 \le n \le k}\{\mathbb{E}(\sup_{\tau \le \varrho \le s}\|u^{(n)}(\varrho)\|^2)\}ds.
\end{eqnarray*}
Employing the Gronwall inequality, we obtain that for any $k \in \mathbb{N}^+$ and $\tau \le t \le T$,
\begin{eqnarray*}
\sup_{1 \le n \le k}\{\mathbb{E}(\sup_{\tau \le \varrho \le t} \|u^{(n)}(\varrho)\|^2)\} \!\!\!&\le&\!\!\!  \left[\Delta_2+\Delta_1 (T-\tau) \mathbb{E}\|\zeta\|_{\mathcal{C}_r}^2 \right] e^{\Delta_1 (T-\tau)}.
\end{eqnarray*}
By the arbitrariness of $k$, it  follows that for all $n \in \mathbb{N}^+$ and $\tau \le t \le T$,
\begin{eqnarray} \label{pro1-step1-1}
\mathbb{E}(\sup_{\tau \le \varrho \le t} \|u^{(n)}(\varrho)\|^2) \!\!\!&\le&\!\!\! \left[\Delta_2+\Delta_1 (T-\tau) \mathbb{E}\|\zeta\|_{\mathcal{C}_r}^2\right]  e^{\Delta_1 T}.
\end{eqnarray}

To simplify the notation, we will denote
$
\xi^{(n)}(t):=u^{(n+1)}(t)-u^{(n)}(t)
$
for all $n \in \mathbb{N}$ and $t \in [\tau-r,T]$. As $n=0$, $\xi^{(0)}(t):=u^{(1)}(t)-\zeta(0)$ for $t \in [\tau-r,T]$.
Let
\begin{eqnarray*}
M\!\!\!&:=&\!\!\! 5(T-\tau)\Big\{18\nu^2+\lambda^2+\bigl(4\|\gamma\|_{L^\infty([\tau, T], \ell^2)}^2
+6 c_1 \|\chi\|_{L^{\infty}(\mathbb{R},\ell^2)}^2\bigr)\bigl[3(T-\tau)+r \bigr] \Big\}\mathbb{E}\|\zeta\|_{\mathcal{C}_r}^2
\\
\!\!\!&&\!\!\! +5(T-\tau) \Big\{4(T-\tau) \|\gamma\|_{L^\infty([\tau, T], \ell^2)}^2+\int_\tau^T \|g(s)\|^2 ds+2c_1\|\kappa\|_{L^{\infty}(\mathbb{R},\ell^2)}^2 \Big\}
\end{eqnarray*}
and
\begin{eqnarray*}
N\!\!\!&:=&\!\!\! 4(T-\tau)\{(18 \nu^2+\lambda^2)+2[\|\Theta\|_{L^\infty([\tau,T], \ell^2)}^2+4\|\psi\|_{L^\infty([\tau,T], \ell^2)}^2]\}
+36c_1 \|\chi\|_{L^\infty(\mathbb{R},\ell^2)}^2 .
\end{eqnarray*}
We first claim by induction that for all $n \in \mathbb{N}$ and $\tau \le t \le T$,
\begin{eqnarray}\label{pro1-step1-2}
E(\sup_{\tau \le \varrho \le t} \|\xi^{(n)}(\varrho)\|^2) \!\!\!&\le&\!\!\! \frac{M [N(t-\tau)]^n}{n!}.
\end{eqnarray}
Now we prove this claim.
Indeed, as $n=0$, we have that for all $\tau \le t \le T$,
\begin{equation*}
\begin{split}
 u^{(1)}(t) + \nu \int_\tau^t Au^{(0)}(s) ds +\lambda \int_\tau^t u^{(0)}(s) ds
=&\,\zeta(0)+\int_\tau^t f(s,u^{(0)}(s), u^{(0)}(s-r), \mathcal{L}_{u^{(0)}(s)}) ds+\int_\tau^t g(s) ds
\\
& \, +\int_\tau^t \tilde{\sigma}(s,u^{(0)}(s), u^{(0)}(s-r), \mathcal{L}_{u^{(0)}(s)})dW(s).
\end{split}
\end{equation*}
Along with \eqref{bounded-A}, we get that for all $\tau \le t \le T$,
\begin{eqnarray*}
E(\sup_{\tau \le \varrho \le t} \|\xi^{(0)}(\varrho)\|^2)
\!\!\!& \le &\!\!\! 5(T-\tau)[(18\nu^2+\lambda^2)\mathbb{E}(\|\zeta\|_{\mathcal{C}_r}^2)+\int_\tau^T \|g(s)\|^2 ds]
\\
\!\!\!&  &\!\!\! + 5 (T-\tau) \int_\tau^t \mathbb{E}\|f(s, u^{(0)}(s), u^{(0)}(s-r), \mathcal{L}_{u^{(0)}(s)})\|^2 ds
\\
\!\!\!&&\!\!\!+5 \mathbb{E}(\sup_{\tau \le \varrho \le t}\|\int_\tau^\varrho \tilde{\sigma}(s, u^{(0)}(s), u^{(0)}(s-r), \mathcal{L}_{u^{(0)}(s)})dW(s)\|^2).
\end{eqnarray*}
By using \eqref{fi-2}, \eqref{bound-1} and \eqref{dis-property}, for all $\tau \le t \le T$ we observe that
\begin{eqnarray*}
\!\!\!&  &\!\!\! 5 (T-\tau) \int_\tau^t \mathbb{E}\|f(s, u^{(0)}(s), u^{(0)}(s-r), \mathcal{L}_{u^{(0)}(s)})\|^2 ds
\\
\!\!\!& \le &\!\!\! 20 (T-\tau) \|\gamma\|_{L^\infty([\tau, T], \ell^2)}^2
\left\{\bigl[3(T-\tau)+r\bigr]\mathbb{E}\|\zeta\|_{\mathcal{C}_r}^2+T-\tau \right\}.
\end{eqnarray*}
Utilizing the BDG inequality, \eqref{dis-property} and \eqref{bound-2},  for all $\tau \le t \le T$ we have that
\begin{equation*}
\begin{split}
&\, 5 \mathbb{E}(\sup_{\tau \le \varrho \le t}\|\int_\tau^\varrho \tilde{\sigma}(s, u^{(0)}(s), u^{(0)}(s-r), \mathcal{L}_{u^{(0)}(s)})dW(s)\|^2)
\\
\le &\, 5 c_1 \int_\tau^t
\mathbb{E} \|\tilde{\sigma}(s, u^{(0)}(s), u^{(0)}(s-r), \mathcal{L}_{u^{(0)}(s)})\|_{L_2(\ell^2, \ell^2)}^2 ds
\\
\le &\, 30 c_1 \|\chi\|_{L^{\infty}(\mathbb{R},\ell^2)}^2 [ 3(T-\tau)+r] \mathbb{E}\|\zeta\|_{\mathcal{C}_r}^2
+10c_1(T-\tau)\|\kappa\|_{L^{\infty}(\mathbb{R},\ell^2)}^2,
\end{split}
\end{equation*}
where $c_1$ follows from the BDG inequality.
Based on the above analyses, \eqref{pro1-step1-2} holds for $n=0$.
Next, we assume that \eqref{pro1-step1-2} holds for $k=n-1$. To prove the case of $k=n$, we first observe for all $\tau \le t \le T$ that
\begin{equation*}
\begin{split}
&\, \xi^{(n)}(t) + \nu \int_\tau^t A\xi^{(n-1)}(s) ds +\lambda \int_\tau^t \xi^{(n-1)}(s) ds
\\
=&\, \int_\tau^t [f(s, u^{(n)}(s), u^{(n)}(s-r), \mathcal{L}_{u^{(n)}(s)})
-f(s, u^{(n-1)}(s), u^{(n-1)}(s-r), \mathcal{L}_{u^{(n-1)}(s)})] ds
\\
&\,+\int_\tau^t [\tilde{\sigma}(s, u^{(n)}(s), u^{(n)}(s-r), \mathcal{L}_{u^{(n)}(s)})
-\tilde{\sigma}(s, u^{(n-1)}(s), u^{(n-1)}(s-r), \mathcal{L}_{u^{(n-1)}(s)})] dW(s).
\end{split}
\end{equation*}
Then, we have that for all $\tau \le t \le T$,
\begin{equation}\label{pro1-step1-3}\small
\begin{split}
&\,\mathbb{ E}(\sup_{\tau \le \varrho \le t}\|\xi^{(n)}(\varrho)\|^2)
\\
\le &\, 4(T-\tau)(18 \nu^2+\lambda^2)\int_\tau^t \mathbb{ E}(\sup_{\tau \le \varrho \le s} \|\xi^{(n-1)}(\varrho)\|^2) ds
\\
&\, +4 (T-\tau)\int_\tau^t \mathbb{ E} \|f(s, u^{(n)}(s), u^{(n)}(s-r), \mathcal{L}_{u^{(n)}(s)})
-f(s, u^{(n-1)}(s), u^{(n-1)}(s-r), \mathcal{L}_{u^{(n-1)}(s)})\|^2 ds
\\
&\, +4\mathbb{ E}(\sup_{\tau \le \varrho \le t} \|\int_\tau^\varrho [\tilde{\sigma}(s, u^{(n)}(s), u^{(n)}(s-r), \mathcal{L}_{u^{(n)}(s)})
\\
&\,\qquad\qquad\qquad\quad -\tilde{\sigma}(s, u^{(n-1)}(s), u^{(n-1)}(s-r), \mathcal{L}_{u^{(n-1)}(s)})] dW(s)\|^2).
\end{split}
\end{equation}
For the second term in the right-hand side of \eqref{pro1-step1-3},
by using \eqref{dis-property} and \eqref{f-Lip} we find that for all $\tau \le t \le T$,
\begin{equation}\label{pro1-step1-4}\small
\begin{split}
&\, 4(T-\tau) \int_\tau^t \mathbb{E} \|f(s, u^{(n)}(s), u^{(n)}(s-r), \mathcal{L}_{u^{(n)}(s)})
-f(s, u^{(n-1)}(s), u^{(n-1)}(s-r), \mathcal{L}_{u^{(n-1)}(s)})\|^2 ds
\\
&\, +16 (T-\tau) \|\psi\|_{L^\infty([\tau,T], \ell^2)}^2 \int_\tau^t \mathbb{E}(\|\xi^{(n-1)}(s-r)\|^2+\mathbb{E}\|\xi^{(n-1)}(s)\|^2) ds
\\
\le &\, 8 (T-\tau)\left( \|\Theta\|_{L^\infty([\tau,T], \ell^2)}^2+4\|\psi\|_{L^\infty([\tau,T], \ell^2)}^2\right) \int_\tau^t \mathbb{E} (\sup_{\tau \le \varrho \le s}\|\xi^{(n-1)}(\varrho)\|^2) ds.
\end{split}
\end{equation}
For the last term in the right-hand side of \eqref{pro1-step1-3},
by using the BDG inequality, \eqref{segama-Lip} and \eqref{dis-property} we find that for all $\tau \le t \le T$,
\begin{equation} \label{pro1-step1-5}\small
\begin{split}
&\, 4 \mathbb{E}(\sup_{\tau\le \varrho \le t} \|\int_\tau^\varrho[\tilde{\sigma}(s,u^{(n)}(s), u^{(n)}(s-r), \mathcal{L}_{u^{(n)}(s)})
-\tilde{\sigma}(s,u^{(n-1)}(s), u^{(n-1)}(s-r), \mathcal{L}_{u^{(n-1)}(s)})]dW(s)\|^2 )
\\
\le &\, 4 c_1 \int_\tau^t \mathbb{E} \|\tilde{\sigma}(s, u^{(n)}(s), u^{(n)}(s-r), \mathcal{L}_{u^{(n)}(s)})
-\tilde{\sigma}(s, u^{(n-1)}(s), u^{(n-1)}(s-r), \mathcal{L}_{u^{(n-1)}(s)})\|_{L_2(\ell^2, \ell^2)}^2 ds
\\
\le &\, 36c_1 \|\chi\|_{L^\infty(\mathbb{R},\ell^2)}^2 \int_\tau^t \mathbb{E} (\sup_{\tau \le \varrho \le s}\|\xi^{(n-1)}(\varrho)\|^2) ds.
\end{split}
\end{equation}
In view of \eqref{pro1-step1-3}-\eqref{pro1-step1-5}, by induction hypothesis, we conclude that for all $t \in [\tau,T]$,
\begin{eqnarray*}\label{E-un-u}
\mathbb{E}(\sup_{\tau \le \varrho \le t}\|\xi^{(n)}(\varrho)\|^2)
\!\!\!& \le &\!\!\! N \int_\tau^t \frac{M[N (s-\tau)]^{n-1}}{(n-1)!}ds= \frac{M [N(t-\tau)]^n}{n!}.
\end{eqnarray*}
Then assertion \eqref{pro1-step1-2} is true for all $n \in \mathbb{N}$.

Employing the Chebyshev inequality we get that
\begin{eqnarray*}
\mathbb{P}(\sup_{\tau \le \varrho \le T}\|\xi^{(n)}(\varrho)\|>\frac{1}{2^n}) \!\!\!&\le &\!\!\! \frac{M (4NT)^n}{n!}.
\end{eqnarray*}
By the convergence of the positive series $\sum_{n=0}^{\infty}\frac{M (4 NT)^n}{n!}$ and the Borel-Cantelli lemma,
for almost all $\omega \in \Omega$ there exists a positive integer $N_0=N_0(\omega)$ such that for all $n > N_0$
\begin{eqnarray*}
\sup_{\tau \le \varrho \le T} \|\xi^{(n)}(\varrho)\|\le \frac{1}{2^n},
\end{eqnarray*}
which implies that the partial sum sequence
\begin{eqnarray*}
u^{(0)}+\sum_{k=0}^{n-1} (u^{(k+1)}(t)-u^{(k)}(t))=u^{(n)}(t)
\end{eqnarray*}
is convergent uniformly on $t \in [\tau, T]$ for almost all $\omega \in \Omega$. We denote the limit by $u(t)$ for $t \in [\tau-r,T]$ with
$u(s)=\zeta(s)$ for $\tau-r \le s \le \tau$.  Clearly, $u$ is $\mathcal{F}_t$-adapted for $t \ge \tau$. Additionally, by \eqref{pro1-step1-1} and the dominated convergence theorem, we get that
\begin{eqnarray}\label{converg-un}
\mathbb{E}(\sup_{\tau \le \varrho \le T}\|u^{(n)}(\varrho)-u(\varrho)\|^2)\rightarrow 0,~\mbox{as}~n\rightarrow \infty.
\end{eqnarray}

In order to verify that $u(t)$ satisfies equation
\eqref{inte-ob-eq}, according to \eqref{bounded-A} and \eqref{converg-un}, we obtain for each $t \in [\tau,T]$ that
\begin{eqnarray}\label{converg-inter-un-1}
\begin{split}
&\, \mathbb{E} \|\nu \int_\tau^t A (u^{(n-1)}(s)-u(s)) ds+\lambda \int_\tau^t (u^{(n-1)}(s)-u(s)) ds\|^2
\\
\le &\,
2(T-\tau)(18\nu^2+\lambda^2 ) \int_\tau^T \mathbb{E}(\sup_{\tau \le \varrho \le s} \|u^{(n-1)}(\varrho)-u(\varrho)\|^2) ds\rightarrow 0,~\mbox{as}~n\rightarrow \infty.
\end{split}
\end{eqnarray}
By using \eqref{dis-property}, \eqref{f-Lip}, \eqref{segama-Lip} and \eqref{converg-un}, we find that for each $t \in [\tau,T]$,
\begin{equation} \label{converg-inter-un-2}
\begin{split}
&\, \mathbb{E}\|\int_\tau^t [f(s, u^{(n-1)}(s),u^{(n-1)}(s-r), \mathcal{L}_{u^{(n-1)}(s)})-f(s,u(s),u(s-r), \mathcal{L}_{u(s)})] ds\|^2
\\
&\, + \mathbb{E}\|\int_\tau^t [\sigma(s,u^{(n-1)}(s),u^{(n-1)}(s-r), \mathcal{L}_{u^{(n-1)}(s)})-\sigma(s,u(s),u(s-r),
\mathcal{L}_{u(s)})] dW(s)\|^2
\\
\le &\,  2\|\Theta\|_{L^\infty([\tau, T], \ell^2)}^2(T-\tau) \int_\tau^t \mathbb{E}\|u^{(n-1)}(s)-u(s)\|^2ds
\\
&\, +4 \|\psi\|_{L^\infty([\tau, T], \ell^2)}^2(T-\tau) \int_\tau^t \mathbb{E}(\|u^{(n-1)}(s-r)-u(s-r)\|^2
+ \mathbb{E}\|u^{(n-1)}(s)-u(s)\|^2) ds
\\
&+3 \|\chi\|_{L^\infty(\mathbb{R}, \ell^2)}^2 \int_\tau^t \mathbb{E}(\|u^{(n-1)}(s)-u(s)\|^2+\mathbb{E}\|u^{(n-1)}(s-r)-u(s-r)\|^2
\\
& \qquad\qquad\qquad\qquad\quad + \mathbb{E}\|u^{(n-1)}(s)-u(s)\|^2)ds
\\
\le &\,   \left[ 2\|\Theta\|_{L^\infty([\tau, T], \ell^2)}^2(T-\tau)+8 \|\psi\|_{L^\infty([\tau, T], \ell^2)}^2(T-\tau)
+9 \|\chi\|_{L^\infty(\mathbb{R}, \ell^2)}^2\right]
\\
&\, \times \int_\tau^T \mathbb{E}(\sup_{\tau \le \varrho \le s}\|u^{(n-1)}(\varrho)-u(\varrho)\|^2) ds\rightarrow 0,
~\mbox{as}~n\rightarrow \infty.
\end{split}
\end{equation}
Combing \eqref{converg-inter-un-1} with \eqref{converg-inter-un-2}, we find that $u$ satisfies \eqref{inte-ob-eq} by letting $n\rightarrow \infty$ in \eqref{pro1-1+}.
Finally, in virtue of \eqref{converg-un}, there exists $N_1>0$ such that
$$\mathbb{E}(\sup_{\tau \le \varrho \le T}\|u^{(N_1)}(\varrho)-u(\varrho)\|^2)<\frac{1}{2}.$$
Using \eqref{pro1-step1-1} again, we have that
\begin{equation}\label{cha}
\begin{split}
\mathbb{E}  \sup_{\tau \le \varrho \le T} \|u(\varrho)\|^2 \le &
2\mathbb{E}  \sup_{\tau \le \varrho \le T} \|u^{(N_1)}(\varrho)-u(\varrho)\|^2)+
2\mathbb{E}  \sup_{\tau \le \varrho \le T} \|u^{(N_1)}(\varrho)\|^2)
\\
\le & 1+2(\Delta_2+\Delta_1 (T-\tau) \mathbb{E}\|\zeta\|_{\mathcal{C}_r}^2) e^{\Delta_1 T}.
\end{split}
\end{equation}
Hence $u \in L^2(\Omega, C([\tau-r,T], \ell^2))$.
\vskip0.1in
{\bf Step 2.} Next, we prove the uniqueness.
Let $u(t)$ and $v(t)$ be two strong solutions of \eqref{ob-3}-\eqref{ob-4} in $L^2(\Omega, C([\tau-r,T],\ell^2))$.
For any $\tau \le t \le T$, notice that
\begin{eqnarray*}
\!\!\!&&\!\!\! u(t)-v(t)+\nu \int_\tau^t A (u(s)-v(s)) ds+\lambda \int_\tau^t (u(s)-v(s)) ds
\\
\!\!\!& =&\!\!\! \int_\tau^t [f(s, u(s), u(s-r), \mathcal{L}_{u(s)})-f(s, v(s), v(s-r), \mathcal{L}_{v(s)})]ds
\\
\!\!\!& &\!\!\! +\int_\tau^t [\tilde{\sigma}(s, u(s), u(s-r), \mathcal{L}_{u(s)})-\tilde{\sigma}(s, v(s), v(s-r), \mathcal{L}_{v(s)})]dW(s).
\end{eqnarray*}
Utilizing the It\^o formula, and taking the supremum and expectation, we have that for all $\tau \le t \le T$,
\begin{equation} \label{pro1-1}\small
\begin{split}
&\, \mathbb{E}(\sup_{\tau \le \varrho \le t }\|u(\varrho)-v(\varrho)\|^2)
\\
\le &\, 2 \int_\tau^t \mathbb{E} \|u(s)-v(s)\| \cdot \|f(s, u(s), u(s-r), \mathcal{L}_{u(s)})-f(s, v(s), v(s-r), \mathcal{L}_{v(s)})\| ds
\\
&\,+\int_\tau^t \mathbb{E} \|\tilde{\sigma}(s, u(s), u(s-r), \mathcal{L}_{u(s)})-\tilde{\sigma}(s, v(s), v(s-r), \mathcal{L}_{v(s)})\|_{L_2(\ell^2, \ell^2)}^2 ds
\\
&\, +2 \mathbb{E}\sup_{\tau \le \varrho \le t} \|\int_\tau^\varrho (u(s)-v(s), [\tilde{\sigma}(s, u(s), u(s-r), \mathcal{L}_{u(s)})-\tilde{\sigma}(s, v(s), v(s-r), \mathcal{L}_{v(s)})]dW(s))\|.
\end{split}
\end{equation}
We estimate each term in the right-hand side of \eqref{pro1-1} separately. For the first term, by using \eqref{f-Lip} and \eqref{dis-property}
we get that for all $\tau \le t \le T$,
\begin{equation}\label{pro1-2}
\begin{split}
&\, 2 \int_\tau^t \mathbb{E} \|u(s)-v(s)\| \cdot \|f(s, u(s), u(s-r), \mathcal{L}_{u(s)})-f(s, v(s), v(s-r), \mathcal{L}_{v(s)})\| ds
\\
\le &\, \int_\tau^t \mathbb{E} \|u(s)-v(s)\|^2 ds+  \int_\tau^t \mathbb{E} \|f(s, u(s), u(s-r), \mathcal{L}_{u(s)})-f(s, v(s), v(s-r), \mathcal{L}_{v(s)})\|^2 ds
\\
&\, +  4 \|\psi\|_{L^\infty([\tau, T], \ell^2)}^2 \int_\tau^t \mathbb{E}(\|u(s-r)-v(s-r)\|^2+\mathbb{E}\|u(s)-v(s)\|^2) ds
\\
\le &\, (1+2\|\Theta\|_{L^\infty([\tau, T], \ell^2)}^2+8 \|\psi\|_{L^\infty([\tau, T], \ell^2)}^2)
\int_\tau^t \mathbb{E} (\sup_{\tau \le \varrho \le s}\|u(\varrho)-v(\varrho)\|^2) ds.
\end{split}
\end{equation}
For the second term, by using \eqref{dis-property} and \eqref{segama-Lip} we have that for all $\tau \le t \le T$,
\begin{equation} \label{pro1-3}
\begin{split}
&\,  \int_\tau^t \mathbb{E} \|\tilde{\sigma}(s, u(s), u(s-r), \mathcal{L}_{u(s)})-\tilde{\sigma}(s, v(s), v(s-r), \mathcal{L}_{v(s)})\|_{L_2(\ell^2, \ell^2)}^2 ds
\\
\le &\, 3 \|\chi\|_{L^\infty(\mathbb{R}, \ell^2)}^2 \int_\tau^t \mathbb{E} (\sup_{\tau \le \varrho \le s}\|u(\varrho)-v(\varrho)\|^2) ds.
\end{split}
\end{equation}
For the last term, by using the BDG inequality and \eqref{pro1-3}  we find that for all $\tau \le t \le T$,
\begin{equation} \label{pro1-4}\small
\begin{split}
&\,  2 \mathbb{E}\sup_{\tau \le \varrho \le t} \|\int_\tau^\varrho (u(s)-v(s), [\tilde{\sigma}(s, u(s), u(s-r), \mathcal{L}_{u(s)})-\tilde{\sigma}(s, v(s), v(s-r), \mathcal{L}_{v(s)})]dW(s))\|
\\
\le &\, 2 c_1 \mathbb{E}\{ \int_\tau^t
\|(u(s)-v(s), [\tilde{\sigma}(s, u(s), u(s-r), \mathcal{L}_{u(s)})-\tilde{\sigma}(s, v(s), v(s-r), \mathcal{L}_{v(s)})]\cdot)
\|_{L_2(\ell^2, \ell^2)}^2 ds \}^{\frac{1}{2}}
\\
\le &\, 2c_1\mathbb{E}\{\int_\tau^t \|u(s)-v(s)\|^2 \cdot \|\tilde{\sigma}(s, u(s), u(s-r), \mathcal{L}_{u(s)})
-\tilde{\sigma}(s, v(s), v(s-r), \mathcal{L}_{v(s)})\|_{L_2(\ell^2, \ell^2)}^2 ds\}^{\frac{1}{2}}
\end{split}
\end{equation}
\begin{equation*}\small
\begin{split}
\le &\, 2c_1\mathbb{E}\{\sup_{\tau \le \varrho \le t}\|u(\varrho)-v(\varrho)\|^2
 (\int_\tau^t \|\tilde{\sigma}(s, u(s), u(s-r), \mathcal{L}_{u(s)})-\tilde{\sigma}(s,v(s), v(s-r), \mathcal{L}_{v(s)})\|_{L_2(\ell^2, \ell^2)}^2 ds)^{\frac{1}{2}}\}
\\
\le &\, \frac{1}{2} \mathbb{E}(\sup_{\tau \le \varrho \le t}\|u(\varrho)-v(\varrho)\|^2)
+ 2c_1^2 \int_\tau^t \mathbb{E}\|\tilde{\sigma}(s, u(s), u(s-r), \mathcal{L}_{u(s)})-\tilde{\sigma}(s,v(s), v(s-r), \mathcal{L}_{v(s)})\|_{L_2(\ell^2, \ell^2)}^2 ds
\\
\le &\, \frac{1}{2} \mathbb{E}(\sup_{\tau \le \varrho \le t}\|u(\varrho)-v(\varrho)\|^2)
+6c_1^2\|\chi\|_{L^\infty(\mathbb{R}, \ell^2)}^2 \int_\tau^t \mathbb{E} (\sup_{\tau \le \varrho \le s}\|u(\varrho)-v(\varrho)\|^2) ds.
\end{split}
\end{equation*}
Combining \eqref{pro1-1}-\eqref{pro1-4}, we get that for all $\tau \le t \le T$,
\begin{equation*}
\begin{split}
\mathbb{E}(\sup_{\tau \le \varrho \le t}\|u(\varrho)-v(\varrho)\|^2)
\le &\, 2 \left[ 1+2\|\Theta\|_{L^\infty([\tau, T], \ell^2)}^2+8 \|\psi\|_{L^\infty([\tau, T], \ell^2)}^2+
3 \|\chi\|_{L^\infty(\mathbb{R}, \ell^2)}^2(1+2c_1^2)\right]
\\
&\, \times \int_\tau^t \mathbb{E} (\sup_{\tau \le \varrho \le s}\|u(\varrho)-v(\varrho)\|^2) ds.
\end{split}
\end{equation*}
Employing the Gronwall inequality, we find that
\begin{eqnarray*}
\mathbb{E}(\sup_{r \le \varrho \le T}\|u(\varrho)-v(\varrho)\|^2)=0
\end{eqnarray*}
which implies that the uniqueness holds.
\\

{\bf Step 3.} Finally, we prove that $\mathcal{L}_{u_t} \in \mathcal{P}_2(\mathcal{C}_r)$ for any $\tau \le t \le T$. In fact,
by using \eqref{inte-ob-eq} and applying the It\^o formula, we have that
\begin{eqnarray*}
\!\!\!& &\!\!\! \|u(t)\|^2 + 2\nu \int_\tau^t\|Bu(s)\|^2 ds +2\lambda \int_\tau^t \|u(s)\|^2 ds
\\
\!\!\!&=&\!\!\! 2 \int_\tau^t(u(s), f(s, u(s), u(s-r), \mathcal{L}_{u(s)})) ds
+\int_\tau^t \| \tilde{\sigma}(s, u(s), u(s-r), \mathcal{L}_{u(s)})\|_{L_2(\ell^2, \ell^2)}^2 ds
\\
\!\!\!&&\!\!\!+2 \int_\tau^t (u(s), \tilde{\sigma}(u(s), u(s-r), \mathcal{L}_{u(s)})dW(s)),\ \forall t \in [\tau, T],
\end{eqnarray*}
which implies that for any $\tau \le t \le T$,
\begin{equation} \label{pro1-11}
\begin{split}
 \mathbb{E}(\sup_{\tau \le \varrho \le t}\|u(\varrho)\|^2)
\le &\, 2 \int_\tau^t \mathbb{E} \left( \|u(s)\| \cdot \|f(s, u(s), u(s-r), \mathcal{L}_{u(s)})\| \right) ds
\\
&\, +\int_\tau^t \mathbb{E} \|\tilde{\sigma}(s, u(s), u(s-r), \mathcal{L}_{u(s)})\|_{L_2(\ell^2, \ell^2)}^2 ds
\\
&\, +2 \mathbb{E}\sup_{\tau \le \varrho \le t} \|\int_\tau^\varrho (u(s), \tilde{\sigma}(s, u(s), u(s-r), \mathcal{L}_{u(s)})dW(s))\|.
\end{split}
\end{equation}
We estimate each term in the right-hand side of \eqref{pro1-11}, separately. For the first term, by using \eqref{dis-property} and \eqref{bound-1}
we get that for all $\tau \le t \le T$,
\begin{equation}\label{pro1-12}
\begin{split}
&\, 2 \int_\tau^t \mathbb{E} \|u(s)\| \cdot \|f(s, u(s), u(s-r), \mathcal{L}_{u(s)})\| ds
\\
\le &\, \int_\tau^t \mathbb{E} \|u(s)\|^2 ds+  \int_\tau^t \mathbb{E} \|f(s, u(s), u(s-r), \mathcal{L}_{u(s)})\|^2 ds
\\
\le &\, (1+12\|\gamma\|_{L^\infty([\tau, T], \ell^2)}^2)
\int_\tau^t \mathbb{E} (\sup_{\tau \le \varrho \le s}\|u(\varrho))\|^2) ds+4\|\gamma\|_{L^\infty([\tau, T], \ell^2)}^2
\left[ (T-\tau)+r\mathbb{E} \|\zeta\|_{\mathcal{C}_r}^2\right].
\end{split}
\end{equation}
For the second term, by using \eqref{dis-property} and \eqref{bound-2} we have that for all $\tau \le t \le T$,
\begin{equation} \label{pro1-13}
\begin{split}
&\,  \int_\tau^t \mathbb{E} \|\tilde{\sigma}(s, u(s), u(s-r), \mathcal{L}_{u(s)})\|_{L_2(\ell^2, \ell^2)}^2 ds
\\
\le &\, 18\|\chi\|_{L^\infty(\mathbb{R}, \ell^2)}^2 \int_\tau^t \mathbb{E} (\sup_{\tau \le \varrho \le s}\|u(\varrho))\|^2) ds
+6r\|\chi\|_{L^\infty(\mathbb{R}, \ell^2)}^2 \mathbb{E} \|\zeta\|_{\mathcal{C}_r}^2
+2 \|\kappa\|_{L^\infty(\mathbb{R},\ell^2)}^2(T-\tau).
\end{split}
\end{equation}
For the last term, by using the BDG inequality and \eqref{pro1-13}, we find that for all $\tau \le t \le T$,
\begin{equation} \label{pro1-14}
\begin{split}
&\,  2 \mathbb{E}(\sup_{\tau \le \varrho \le t} \|\int_\tau^\varrho (u(s), \tilde{\sigma}(s, u(s), u(s-r), \mathcal{L}_{u(s)})dW(s))\|)
\\
\le &\, 2 c_1 \mathbb{E}\{ \int_\tau^t
\|(u(s), \tilde{\sigma}(s, u(s), u(s-r), \mathcal{L}_{u(s)}) )
\|_{L_2(\ell^2, \ell^2)}^2 ds \}^{\frac{1}{2}}
\\
\le &\, 2c_1\mathbb{E}\{\int_\tau^t \|u(s)\|^2 \cdot \|\tilde{\sigma}(s, u(s), u(s-r), \mathcal{L}_{u(s)})\|_{L_2(\ell^2, \ell^2)}^2 ds\}^{\frac{1}{2}}
\\
\le &\, 2c_1\mathbb{E}\{\sup_{\tau \le \varrho \le t}\|u(\varrho)\|^2 \cdot (\int_\tau^t \|\tilde{\sigma}(s, u(s), u(s-r), \mathcal{L}_{u(s)})\|_{L_2(\ell^2, \ell^2)}^2 ds)^{\frac{1}{2}}\}
\\
\le &\, \frac{1}{2} \mathbb{E}(\sup_{\tau \le \varrho \le t}\|u(\varrho)\|^2)
+ 2c_1^2 \int_\tau^t \mathbb{E}\|\tilde{\sigma}(s, u(s), u(s-r), \mathcal{L}_{u(s)})\|_{L_2(\ell^2, \ell^2)}^2 ds
\\
\le &\, \frac{1}{2} \mathbb{E}(\sup_{\tau \le \varrho \le t}\|u(\varrho)\|^2)+36c_1^2\|\chi\|_{L^\infty(\mathbb{R}, \ell^2)}^2 \int_\tau^t \mathbb{E} (\sup_{\tau \le \varrho \le s}\|u(\varrho))\|^2) ds
\\
&\, +12r c_1^2\|\chi\|_{L^\infty(\mathbb{R}, \ell^2)}^2 \mathbb{E} \|\zeta\|_{\mathcal{C}_r}^2
+4c_1^2 \|\kappa\|_{L^\infty(\mathbb{R},\ell^2)}^2(T-\tau).
\end{split}
\end{equation}
Combining \eqref{pro1-11}-\eqref{pro1-14}, we get that for all $\tau \le t \le T$,
\begin{equation*}
\begin{split}
&\, \mathbb{E}(\sup_{\tau \le \varrho \le t}\|u(\varrho)\|^2)
\\
\le &\, 8\|\gamma\|_{L^\infty([\tau, T], \ell^2)}^2
\left[ (T-\tau)+r\mathbb{E} \|\zeta\|_{\mathcal{C}_r}^2\right]
\\
&\, +2(1+2c_1^2)\left[ \|\kappa\|_{L^\infty(\mathbb{R},\ell^2)}^2(T-\tau)+6r\|\chi\|_{L^\infty(\mathbb{R}, \ell^2)}^2 \mathbb{E} \|\zeta\|_{\mathcal{C}_r}^2\right]
\\ &\,+2\left[ 1+12\|\gamma\|_{L^\infty([\tau, T], \ell^2)}^2+18(1+2c_1^2)\|\chi\|_{L^\infty(\mathbb{R}, \ell^2)}^2\right]
\int_\tau^t \mathbb{E} (\sup_{\tau \le \varrho \le s}\|u(\varrho))\|^2) ds.
\end{split}
\end{equation*}
Using the Gronwall inequality, we have that for all $\tau-r \le t \le T$,
\begin{equation*}
\begin{split}
 \mathbb{E}(\sup_{\tau-r \le \varrho \le T}\|u(\varrho)\|^2)
=&\, \left\{8\|\gamma\|_{L^\infty([\tau, T], \ell^2)}^2
[(T-\tau)+r\mathbb{E} \|\zeta\|_{\mathcal{C}_r}^2]  \right.
\\
&\, \left. +2(1+2c_1^2)[\|\kappa\|_{L^\infty(\mathbb{R},\ell^2)}^2(T-\tau)+6r\|\chi\|_{L^\infty(\mathbb{R}, \ell^2)}^2 \mathbb{E} \|\zeta\|_{\mathcal{C}_r}^2] \right\}
\\
&\, \cdot e^{2 \left[ 1+12\|\gamma\|_{L^\infty([\tau, T], \ell^2)}^2+18(1+2c_1^2)\|\chi\|_{L^\infty(\mathbb{R}, \ell^2)}^2\right]
(T-\tau)}+ \mathbb{E}\|\zeta\|_{\mathcal{C}_r}^2.
\end{split}
\end{equation*}
This yields that for all $\tau \le t \le T$,
\begin{eqnarray*}
\mathbb{E}\|u_t\|_{\mathcal{C}_r}^2  \le \mathbb{E}(\sup_{\tau-r \le \varrho \le T}\|u(\varrho)\|^2)<\infty.
\end{eqnarray*}
Hence, $\mathcal{L}_{u_t} \in \mathcal{P}_2(\mathcal{C}_r)$ for all $\tau \le t \le T$. This completes the proof.
\end{proof}
\begin{rmk}
Furthermore, using similar methods as Step 1 in Proposition \ref{Pro1},
if $\zeta \in L^4(\Omega, \mathcal{C}_r)$ is an $\mathcal{F}_r$-measurable random variable and conditions {\bf (H1)}-{\bf (H3)} are satisfied,  then system \eqref{ob-3}-\eqref{ob-4}
has  a strong solution $u \in L^4(\Omega, C([\tau-r,T],\ell^2))$ for any $T>\tau-r$.
\end{rmk}

The following proposition gives the existence and uniqueness of weak solutions for stochastic delay lattice system \eqref{ob-3}.
\begin{pro}\label{weak-uniquess}
Equation \eqref{ob-3}  admits  a unique weak solution for any an initial distribution $\mu_0 \in \mathcal{P}_2(\mathcal{C}_r)$.
Weak uniqueness means that if $(\{u(t)\}_{t \ge \tau-r}, \{W(t)\}_{t \ge \tau})$ is a weak solution under the stochastic basis
$(\Omega, \mathcal{F}, \{\mathcal{F}_t\}_{t \ge \tau}, \mathbb{P})$ and
$(\{\tilde{u}(t)\}_{t \ge \tau-r}, \{\tilde{W}(t)\}_{t \ge \tau})$ is a weak solution under the stochastic basis
$(\tilde{\Omega}, \tilde{\mathcal{F}}, \{\tilde{\mathcal{F}_t}\}_{t \ge \tau}, \tilde{\mathbb{P}})$
to system \eqref{ob-3} with the same initial distribution
$\mathcal{L}_{u_\tau}|_{\mathbb{P}}=\mathcal{L}_{\tilde{u}_\tau}|_{\tilde{\mathbb{P}}}$,
then $\mathcal{L}_{u(\cdot)}|_{\mathbb{P}}=\mathcal{L}_{\tilde{u}(\cdot)}|_{\tilde{\mathbb{P}}}$.
\end{pro}

\begin{proof}
The proof follows a similar manner to \cite{Huang2019}. For the convenience of the reader, we outline the key points as follows.
Suppose that  $(\{u(t)\}_{t \ge \tau-r}, \{W(t)\}_{t \ge \tau})$ is a weak solution under the stochastic basis
$(\Omega, \mathcal{F}, \{\mathcal{F}_t\}_{t \ge \tau}, \mathbb{P})$ and
$(\{\tilde{u}(t)\}_{t \ge \tau-r}, \{\tilde{W}(t)\}_{t \ge \tau})$ is a weak solution under the stochastic basis
$(\tilde{\Omega}, \tilde{\mathcal{F}}, \{\tilde{\mathcal{F}_t}\}_{t \ge \tau}, \tilde{\mathbb{P}})$
to system \eqref{ob-3}, both with the same initial distribution
$\mathcal{L}_{u_0}|_{\mathbb{P}}=\mathcal{L}_{\tilde{u}_0}|_{\tilde{\mathbb{P}}}$.
To establish weak uniqueness,
it is sufficient to prove that $\mathcal{L}_{u(\cdot)}|_{\mathbb{P}}=\mathcal{L}_{\tilde{u}(\cdot)}|_{\tilde{\mathbb{P}}}$.
Notice that $\tilde{u}$ satisfies the following stochastic differential equation:
\begin{equation}\label{weakunique-1}
\begin{split}
&\; d\tilde{u}(t)+\nu A \tilde{u}(t) dt+\lambda \tilde{u}(t) dt
\\
&= f(t, \tilde{u}(t),\tilde{u}(t-r), \mathcal{L}_{\tilde{u}(t)}) dt
+\tilde{\sigma}(t, \tilde{u}(t), \tilde{u}(t-r), \mathcal{L}_{\tilde{u}(t)})d\tilde{W}(t).
\end{split}
\end{equation}
Letting $\mathcal{L}_{u(t)}|_{\mathbb{P}}=\mu_t$, we consider the stochastic differential equation
\begin{equation}\label{weakunique-2}
\begin{split}
&\; d\bar{u}(t)+\nu A \bar{u}(t) dt+\lambda \bar{u}(t) dt
\\
&=\; f(t, \bar{u}(t),\bar{u}(t-r), \mu_t) dt+\tilde{\sigma}(t,\bar{u}(t),\bar{u}(t-r), \mu_t)d\tilde{W}(t),
\,\, \bar{u}_\tau=\tilde{u}_\tau.
\end{split}
\end{equation}
By using the usual arguments, we find that equation \eqref{weakunique-2} admits a unique strong solution. Using the Yamada-Watanabe
theorem, one has that the weak solution of equation \eqref{weakunique-2} is unique. Notice that
\begin{equation*}
\begin{split}
&\; du(t)+\nu A u(t) dt+\lambda u(t) dt
\\
&=\; f(t, u(t),u(t-r), \mu_t) dt+\tilde{\sigma}(t, u(t),u(t-r), \mu_t)dW(t), \,\, \mathcal{L}_{u_\tau}|_{\mathbb{P}}=\mathcal{L}_{\tilde{u}_\tau}|_{\tilde{\mathbb{P}}}.
\end{split}
\end{equation*}
It then follows that
\begin{eqnarray*}
\mathcal{L}_{u(\cdot)}|_{\mathbb{P}}=\mathcal{L}_{\bar{u}(\cdot)}|_{\tilde{\mathbb{P}}}.
\end{eqnarray*}
In view of this fact, we rewrite \eqref{weakunique-2} as
\begin{equation*}
\begin{split}
&\, d\bar{u}(t)+\nu A \bar{u}(t) dt+\lambda \bar{u}(t) dt
\\
&=\, f(t, \bar{u}(t),\bar{u}(t-r), \mathcal{L}_{\tilde{u}(\cdot)}|_{\tilde{\mathbb{P}}}) dt+\tilde{\sigma}(t, \bar{u}(t),\bar{u}(t-r), \mathcal{L}_{\tilde{u}(\cdot)}|_{\tilde{\mathbb{P}}})d\tilde{W}(t),
\,\, \bar{u}_\tau=\tilde{u}_\tau.
\end{split}
\end{equation*}
Then by the strong uniqueness of equation \eqref{weakunique-1}, we obtain $\tilde{u}=\bar{u}$. Therefore
\begin{eqnarray*}
\mathcal{L}_{\bar{u}(\cdot)}|_{\tilde{\mathbb{P}}}=\mathcal{L}_{\tilde{u}(\cdot)}|_{\tilde{\mathbb{P}}}=\mathcal{L}_{u(\cdot)}|_{\mathbb{P}}.
\end{eqnarray*}
Then the proof of Proposition \ref{weak-uniquess} is complete.
\end{proof}

\section{uniform estimates of solutions}

This section is devoted to deriving various uniform estimates for the solutions of equation \eqref{ob-3}, which will be used to prove the existence of pullback measure attractors.

We denote by $u(t, \tau-t, u_{\tau-t})$ the solution of \eqref{ob-3} with initial condition $u_{\tau-t} \in \mathcal{C}_r$.
In this part,  we first present the uniform estimates of the solutions in $L^2(\Omega, \mathbb{P})$.
\begin{lm} \label{est-1}
Assume that {\bf (H1)}-{\bf (H3)} and \eqref{lamda} hold.
Then for every $ \tau \in \mathbb{R}$ and $D=\{D ( t ) : t \in \mathbb{R}\} \in \mathcal{D}_0$,
there exists $T=T(\tau, D )>r$, such that for all $t \ge T$, the solution $u $ of \eqref{ob-3}-\eqref{ob-4} satisfies
\begin{equation} \label{est1-1}
\begin{split}
& \, \mathbb{E} \left (\|u (\tau, \tau-t, u_{\tau -t} ) \|^2 \right)+\mathbb{E}\int_ {\tau-t}^{\tau} e^{\varepsilon (s-\tau)}\|u(s, \tau-t, u_{\tau-t})\|^2 ds
\\
&\, +2\alpha \mathbb{E}(\int_{\tau}^{t}e^{\varepsilon (s-\tau)}\|u (\tau, \tau-t, u_{\tau -t} )\|_p^pds)
 \\
& \le  M_1+ M_1\int_{-\infty }^\tau e^{\varepsilon(s-\tau)}\left(\|g(s)\|^2+\|\eta(s)\|_1 \right)ds,
\end{split}
 \end{equation}
where $u_{\tau-t}\in L^2(\Omega, \mathcal{F}_{\tau-t}, \mathcal{C}_r)$ with $\mathcal{L}_{u_{\tau-t}} \in D(\tau-t)$,
$\varepsilon>0$ is the same number as in \eqref{lamda-1}, and $M_1>0$ is a constant independent of $\tau$ and $D$.
\end{lm}

\begin{proof}
For any $u_\tau \in \mathcal{C}_r$, we denote by $u(t)=u(t, \tau, u_\tau)$ the solution of \eqref{ob-3}-\eqref{ob-4}
with initial data $u_\tau$ at initial time $\tau$ for brevity.

Following the It\^o formula, we obtain the following for all for all $ t\ge \tau $
\begin{equation}\label{uniform-esti-2}
  \begin{split}
  &\, \mathbb{E}(e^{\varepsilon t}\|u(t)\|^2)
  +2\nu \mathbb{E} \int_{\tau}^t e^{\varepsilon s} \|Bu(s)\|^2 ds
  +(2\lambda-\varepsilon) \mathbb{E}\int_ {\tau}^t e^{\varepsilon s}\|u( s )\|^2 ds
  \\
  =&\,  e^{\varepsilon \tau} \mathbb{E} \|u_ {\tau}(0)\|^2+2 \mathbb{E} \int_{\tau}^t e^{\varepsilon s}
  (f(s, u(s), u(s-r), \mathcal{L}_{u(s)}), u(s)) ds
  \\
  &+ 2 \mathbb{E} \int_ {\tau}^t e^{\varepsilon s}(g(s), u(s)) ds
 +\mathbb{E} \int_{\tau}^t e^{\varepsilon s}\|\tilde{\sigma}(s, u  (s), u(s-r), \mathcal{L}_{u(s)})\|_{L_2(\ell^2,\ell^2)}^2 ds.
   \end{split}
\end{equation}

Next, we estimate each term on the right-hand side of \eqref{uniform-esti-2} separately.
For the second term,  using \eqref{fi-1} we have

\begin{equation}\label{uniform-esti-3}
\begin{split}
& 2 \mathbb{E} \int_ {\tau}^{t} e^{\varepsilon s}(f(s, u(s), u(s-r), \mathcal{L}_{u(s)}), u(s)) ds
\\
\le &
-2\alpha \mathbb{E}(\int_ { \tau} ^ {t} e^{\varepsilon s}\| u(s) \|_p^p ds)
+2 \mathbb{E}(\int_ {\tau}^{t}e^{\varepsilon s} \|\eta(s)\|_1 ds)
\\
&+2 \mathbb{E}(\int_ {\tau}^{t}e^{\varepsilon s} \|\eta(s)\|_{\infty}(\|u(s)\|^2+ \|u(s-r)\|^2
+\mathbb{E}\|u(s)\|^2) ds)
\\
\le & -2\alpha \mathbb{E}(\int_{\tau}^{t}e^{\varepsilon s}\|u(s)\|_p^pds)
 +2 \int_ {\tau}^{t}e^{\varepsilon s} \|\eta(s)\|_1 ds
 \\
 &  +2\|\eta\|_{L^{\infty}(\mathbb{R},\ell^2)}(2+e^{\varepsilon r})\int_{\tau}^{t}e^{\varepsilon s} \mathbb{E}\|u(s)\|^2 ds
 +2 \varepsilon^{-1} e^{\varepsilon (r+\tau)} \|\eta\|_{L^{\infty}(\mathbb{R},\ell^2)} \mathbb{E} \|u_\tau\|_{\mathcal{C}_r}^2.
\end{split}
\end{equation}
For the third term, we have
\begin{equation}\label{uniform-esti-5}
\begin{split}
 2 \mathbb{E} \int_ {\tau}^{t} e^{\varepsilon s}(g(s), u(s)) ds
\le  \frac{1}{\lambda} \int_ {\tau}^{t} e^{\varepsilon s}\|g(s)\|^2 ds
+\lambda \int_{\tau}^t  e^{\varepsilon s} \mathbb{E} \|u(s)\|^2 ds.
\end{split}
\end{equation}
For the fourth term,  using \eqref{dis-property} and \eqref{bound-2} we have
\begin{equation}\label{uniform-esti-6}
\begin{split}
& \mathbb{E} \int_{\tau}^{t} e^{\varepsilon s}\|\tilde{\sigma}(s, u  (s), u(s-r), \mathcal{L}_{u(s)})\|_{L_2(\ell^2,\ell^2)}^2 ds
\\
\le & 6  \|\chi\|^2_{L^{\infty}(\mathbb{R},\ell^2)}(2+e^{\varepsilon r})\int_{\tau}^t e^{\varepsilon s} \mathbb{E}\|u(s)\|^2 ds
+ 6  \varepsilon^{-1} e^{\varepsilon (\tau+r)} \|\chi\|^2_{L^{\infty}(\mathbb{R},\ell^2)} \mathbb{E}\|u_{\tau}\|_{\mathcal{C}_r}^2
\\
&+2 \int_{\tau}^t e^{\varepsilon s} \|\kappa(s)\|^2 ds.
\end{split}
\end{equation}
It then follows from \eqref{uniform-esti-2}-\eqref{uniform-esti-6} that
\begin{equation}\label{uniform-esti-7}
\begin{split}
& \mathbb{E}(e^{\varepsilon t}\|u(t)\|^2)
  +2\nu \mathbb{E} \int_{\tau}^{t} e^{\varepsilon s} \|Bu(s)\|^2 ds
  +(2\lambda-\varepsilon) \mathbb{E}\int_ {\tau}^{t} e^{\varepsilon s}\|u( s )\|^2 ds
\\
\le & \, -2\alpha \mathbb{E}(\int_ { \tau} ^ {t} e^{\varepsilon s}\| u(s) \|_p^p ds)+(1+2 \varepsilon^{-1} e^{\varepsilon r} \|\eta\|_{L^{\infty}(\mathbb{R},\ell^2)}
+ 6 \varepsilon^{-1} e^{\varepsilon r} \|\chi\|^2_{L^{\infty}(\mathbb{R},\ell^2)})e^{\varepsilon \tau} \mathbb{E} \|u_ {\tau}\|_{\mathcal{C}_r}^2 
\\
&\, +\left[ 2\|\eta\|_{L^{\infty}(\mathbb{R},\ell^2)}(2+e^{\varepsilon r})+\lambda
+6\|\chi\|^2_{L^{\infty}(\mathbb{R},\ell^2)}(2+e^{\varepsilon r})\right] \int_{\tau}^{t}e^{\varepsilon s} \mathbb{E}\|u(s)\|^2 ds
\\
&\, +2 \int_ {\tau}^{t}e^{\varepsilon s} \|\eta(s)\|_1 ds+\frac{1}{\lambda}\int_ {\tau}^{t} e^{\varepsilon s}\|g(s)\|^2 ds
+2 \int_{\tau}^t e^{\varepsilon s} \|\kappa(s)\|^2 ds.
\end{split}
\end{equation}
By \eqref{lamda} we get that
for all $t\ge \tau$,
\begin{equation}\label{uniform-esti-8}
\begin{split}
& \mathbb{E}\|u(t)\|^2
  +2\nu \mathbb{E} \int_{\tau}^{t} e^{\varepsilon (s-t)} \|Bu(s)\|^2 ds
  +\varepsilon \mathbb{E}\int_ {\tau}^{t} e^{\varepsilon (s-t)}\|u( s )\|^2 ds
  \\
  &\, +2\alpha \mathbb{E}(\int_ { \tau} ^ {t} e^{\varepsilon (s-t)}\| u(s) \|_p^p ds)
\\
\le & \, (1+2 \varepsilon^{-1} e^{\varepsilon r} \|\eta\|_{L^{\infty}(\mathbb{R},\ell^2)}
+ 6\varepsilon^{-1} e^{\varepsilon r} \|\chi\|^2_{L^{\infty}(\mathbb{R},\ell^2)})e^{\varepsilon (\tau-t)} \mathbb{E} \|u_ {\tau}\|_{\mathcal{C}_r}^2
\\
&\, +2 \int_ {\tau}^{t}e^{\varepsilon (s-t)} \|\eta(s)\|_1 ds+\frac{1}{\lambda}\int_ {\tau}^{t} e^{\varepsilon (s-t)}\|g(s)\|^2 ds
+2\int_{\tau}^t e^{\varepsilon (s-t)} \|\kappa(s)\|^2 ds.
\end{split}
\end{equation}
 Replacing $\tau$ and $t$ in \eqref{uniform-esti-8} by $\tau-t$ and $\tau$, respectively, we get that
\begin{equation}\label{uniform-esti-9}
\begin{split}
& \mathbb{E}\|u(\tau, \tau-t, u_{\tau-t})\|^2
  +2\nu \mathbb{E} \int_{\tau-t}^{\tau} e^{\varepsilon (s-\tau)} \|Bu(s, \tau-t, u_{\tau-t})\|^2 ds
    \\
  & +\varepsilon \mathbb{E}\int_ {\tau-t}^{\tau} e^{\varepsilon (s-\tau)}\|u(s, \tau-t, u_{\tau-t})\|^2 ds
  +2\alpha \mathbb{E}(\int_ { \tau-t} ^ {\tau} e^{\varepsilon (s-\tau)}\|u(s, \tau-t, u_{\tau-t})\|_p^p ds)
\\
\le & \, (1+2 \varepsilon^{-1} e^{\varepsilon r} \|\eta\|_{L^{\infty}(\mathbb{R},\ell^2)}
+ 6\varepsilon^{-1} e^{\varepsilon r} \|\chi\|^2_{L^{\infty}(\mathbb{R},\ell^2)})e^{- \varepsilon t}
\mathbb{E} \|u_ {\tau-t}\|_{\mathcal{C}_r}^2
\\
&\, +2 \int_ {\tau-t}^{\tau}e^{\varepsilon (s-\tau)} \|\eta(s)\|_1 ds+\frac{1}{\lambda} \int_ {\tau-t}^{\tau} e^{\varepsilon (s-\tau)}\|g(s)\|^2 ds
+2\int_{\tau-t}^\tau e^{\varepsilon (s-\tau)} \|\kappa(s)\|^2 ds.
\end{split}
\end{equation}
In view of facts that $\mathcal{L}_{u_{\tau-t}} \in D(\tau-t)$ and $D \in \mathcal{D}_0$, we have that
\begin{eqnarray*}
e^{- \varepsilon t} \mathbb{E} \|u_ {\tau-t}\|_{\mathcal{C}_r}^2
\!\!\!&\le&\!\!\!
e^{-\varepsilon \tau} \left[ e^{\varepsilon(\tau-t)}\|D(\tau-t)\|_{\mathcal{P}_2(\mathcal{C}_r)}^2\right] \rightarrow 0
\end{eqnarray*}
as $t \rightarrow \infty$. Hence, there exists $T=T(\tau, D)>r$ such that $t \ge T$,
\begin{eqnarray*}
e^{- \varepsilon t} \mathbb{E} \|u_ {\tau-t}\|_{\mathcal{C}_r}^2 \le 1.
\end{eqnarray*}
Along with \eqref{uniform-esti-9}, for all $t \ge T$, by \eqref{g^2} we obtain
\begin{equation*}
\begin{split}
& \mathbb{E}\|u(\tau, \tau-t, u_{\tau-t})\|^2
  +2\nu \mathbb{E} \int_{\tau-t}^{\tau} e^{\varepsilon (s-\tau)} \|Bu(s, \tau-t, u_{\tau-t})\|^2 ds
    \\
  & +\varepsilon \mathbb{E}\int_ {\tau-t}^{\tau} e^{\varepsilon (s-\tau)}\|u(s, \tau-t, u_{\tau-t})\|^2 ds+ 2\alpha \mathbb{E}(\int_{\tau-t}^{\tau}e^{\varepsilon (s-\tau)}\|u(s, \tau-t, u_{\tau-t})\|_p^pds)
\\
\le & \, 1+2 \varepsilon^{-1} e^{\varepsilon r} \|\eta\|_{L^{\infty}(\mathbb{R},\ell^2)}
+ 6\varepsilon^{-1} e^{\varepsilon r} \|\chi\|^2_{L^{\infty}(\mathbb{R},\ell^2)}+2\varepsilon^{-1} \|\kappa\|^2_{L^{\infty}(\mathbb{R}, \ell^2)}
\\
&\, +2 \int_ {-\infty}^{\tau}e^{\varepsilon (s-\tau)} \|\eta(s)\|_1 ds+\frac{1}{\lambda} \int_ {-\infty}^{\tau} e^{\varepsilon (s-\tau)}\|g(s)\|^2 ds
.
\end{split}
\end{equation*}
This completes the proof of Lemma \ref{est-1}.
\end{proof}

Next, we derive the uniform estimates for segments of the solutions in $L^2(\Omega, \mathcal{C}_r)$.
\begin{lm} \label{est-1+}
Assume that {\bf (H1)}-{\bf (H3)} and \eqref{lamda} hold.
Then, for every $ \tau \in \mathbb{R}$ and $D=\{D ( t ) : t \in \mathbb{R}\} \in \mathcal{D}_0$,
there exists $T=T(\tau, D )>r$, such that for all $t \ge T$, the solution $u $ of \eqref{ob-3}-\eqref{ob-4} satisfies
\begin{equation} \label{est1-1+}
\begin{split}
 \mathbb{E} \left (\|u_\tau(\cdot, \tau-t, u_{\tau -t} ) \|_{\mathcal{C}_r}^2 \right)
\le  M_2+ M_2\int_{-\infty }^\tau e^{\varepsilon(s-\tau)}\left(\|g(s)\|^2+\|\eta(s)\|_1 \right)ds,
\end{split}
 \end{equation}
where $u_{\tau-t}\in L^2(\Omega, \mathcal{F}_{\tau-t}, \mathcal{C}_r)$ with $\mathcal{L}_{u_{\tau-t}} \in D(\tau-t)$,
$\varepsilon>0$ is the same number as in \eqref{lamda-1}, and $M_2>0$ is a constant independent of $\tau$ and $D$.
\end{lm}

\begin{proof}
For convenience, we denote $u(t):= u(t, \tau-t, u_{\tau-t})$ as the solution of \eqref{ob-3} with initial data $u_{\tau-t}$
at initial time $\tau-t$.

By applying the It\^{o} formula and taking the supremum and expectation of \eqref{ob-3}, we obtain for all $\varrho \ge \tau-t$,
\begin{equation} \label{est1-1+3}
 \begin{split}
 & \mathbb{E} (\sup_{\tau-t \le \varrho \le \tau} e^{\varepsilon \varrho} \|u(\varrho)\|^2)
 +(2 \lambda-\varepsilon)\int_ {\tau-t}^{\tau} e^{\varepsilon s} \mathbb{E}\|u(s)\|^2 ds
 \\
  \le & 2 e^{\varepsilon (\tau-t)} \mathbb{E}\|u_{\tau-t}(0)\|^2
  +4  \mathbb{E} \sup_{\tau-t \le \varrho \le \tau} \int_ {\tau-t}^{\varrho} e^{\varepsilon s}
 (f(s, u(s), u(s-r), \mathcal{L}_{u (s)}), u(s))ds
 \\
  & + 4  \mathbb{E} \sup_{\tau-t \le \varrho \le \tau}  \int_ {\tau-t}^{\varrho} e^{\varepsilon s}
 (g(s), u(s))ds
 \\
 & + 2 \mathbb{E} \sup_{\tau-t \le \varrho \le \tau}  \int_ {\tau-t}^{\varrho} e^{\varepsilon s}\|\tilde{\sigma}(s, u(s), u(s-r), \mathcal{L}_{u(s)})\|_{L_2(\ell^2, \ell^2)}^2 ds
 \\
 &+ 4 \mathbb{E} \sup_{\tau-t \le \varrho \le \tau}  \biggl|\int_ {\tau-t}^{\varrho} e^{\varepsilon s}\left(u(s), \tilde{\sigma}(s, u(s), u(s-r),\mathcal{L}_{u(s)}) dW(s)\right) \biggr|.
 \end{split}
 \end{equation}
Next, we estimate all terms on the right-hand side of \eqref{est1-1+3} separately. For the second term on the right-hand side of
\eqref{est1-1+3}, by \eqref{fi-1} we have
\begin{equation}\label{est1-1+4}
\begin{split}
& 4  \mathbb{E} \sup_{\tau-t \le \varrho \le \tau} \int_ {\tau-t}^{\varrho} e^{\varepsilon s}
 (f(s, u(s), u(s-r), \mathcal{L}_{u (s)}),  u(s))ds
\\
\le & \,  -4 \alpha \mathbb{ E} \int_{\tau-t}^\tau e^{\varepsilon s}\|u(s)\|_p^p ds+4 \mathbb{ E} \int_{\tau-t}^\tau e^{\varepsilon s} \|\eta(s)\|_1 ds
\\
&\, +4 \mathbb{ E} \int_{\tau-t}^\tau e^{\varepsilon s} \|\eta(s)\|(\|u(s)\|^2+\|u(s-r)\|^2+  \mathbb{E}\|u(s)\|^2)ds
\\
\le & \,  4\int_{\tau-t}^\tau e^{\varepsilon s}\|\eta(s)\|_1 ds
+ 4\|\eta\|_{L^{\infty}(\mathbb{R},\ell^2)}(2+e^{\varepsilon r})\int_{\tau-t}^\tau e^{\varepsilon s} \mathbb{E} \|u(s)\|^2 ds
\\
&\,+4 \varepsilon^{-1} e^{\varepsilon(\tau-t+r)} \mathbb{E} \|u_{\tau-t}\|_{\mathcal{C}_r}^2.
\end{split}
\end{equation}
For the third term on the right-hand side of \eqref{est1-1+3}, we have
\begin{equation}\label{est1-1+5}
\begin{split}
4  \mathbb{E} \sup_{\tau-t \le \varrho \le \tau}  \int_ {\tau-t}^{\varrho} e^{\varepsilon s}
 (g(s), u(s))ds
\le & \frac{8}{3\lambda} \int_{\tau-t}^\tau e^{\varepsilon s}\|g(s)\|^2 ds
+\frac{3\lambda}{2}\int_{\tau-t}^\tau e^{\varepsilon s}  \mathbb{ E} \|u(s)\|^2 ds.
\end{split}
\end{equation}
For the fourth term on the right-hand side of \eqref{est1-1+3}, by \eqref{bound-2} we have
\begin{equation}\label{est1-1+6}
\begin{split}
& 2 \mathbb{E} \sup_{\tau-t \le \varrho \le \tau}  \int_ {\tau-t}^{\varrho} e^{\varepsilon s}
\|\tilde{\sigma}(s, u(s), u(s-r), \mathcal{L}_{u(s)})\|_{L_2(\ell^2, \ell^2)}^2 ds
\\
\le & \,  12 \|\chi\|^2_{L^{\infty}(\mathbb{R}, \ell^2)}(2+e^{\varepsilon r}) \int_ {\tau-t}^{\tau} e^{\varepsilon s} \mathbb{E}
\|u(s)\|^2 ds + 12 \|\chi\|^2_{L^{\infty}(\mathbb{R}, \ell^2)} \varepsilon^{-1}
e^{\varepsilon(\tau-t+r)} \mathbb{E}(\|u_{\tau-t}\|_{\mathcal{C}_r}^2)
\\
& \, +4 \int_{\tau-t}^\tau e^{\varepsilon s} \|\kappa(s)\|^2 ds.
\end{split}
\end{equation}
For the last term on the right-hand side of \eqref{est1-1+3}, similar to \eqref{pro1-4}, by \eqref{est1-1+6} we have
\begin{equation}\label{est1-1+7}
\begin{split}
&   4 \mathbb{E} \sup_{\tau-t \le \varrho \le \tau}  \biggl|\int_ {\tau-t}^{\varrho} e^{\varepsilon s}
\left(u(s), \tilde{\sigma}(s, u(s), u(s-r),\mathcal{L}_{u(s)}) dW(s)\right) \biggr|
\\
\le &\, 4c_1\mathbb{E}\{\sup_{\tau-t \le \varrho \le \tau} e^{\varepsilon \varrho}\|u(\varrho)\|^2 \cdot
(\int_{\tau-t}^\tau \|\tilde{\sigma}(s, u(s), u(s-r), \mathcal{L}_{u(s)})\|_{L_2(\ell^2, \ell^2)}^2 ds)\}^{\frac{1}{2}}
\\
\le &\, \frac{1}{2} \mathbb{E}(\sup_{\tau-t \le \varrho\le \tau}e^{\varepsilon \varrho}\|u(\varrho)\|^2)
\\
&\,+ 8 c_1^2 \int_{\tau-t}^\tau
\mathbb{E}\|\tilde{\sigma}(s, u(s), u(s-r), \mathcal{L}_{u(s)})\|_{L_2(\ell^2, \ell^2)}^2 ds
\\
\le &\,  \frac{1}{2}\mathbb{E}(\sup_{\tau-t \le \varrho\le \tau}e^{\varepsilon \varrho}\|u(\varrho)\|^2)
+48c_1^2 \|\chi\|^2_{L^{\infty}(\mathbb{R}, \ell^2)}(2+e^{\varepsilon r}) \int_ {\tau-t}^{\tau} e^{\varepsilon s} \mathbb{E}
\|u(s)\|^2 ds
\\
&\, + 48c_1^2\|\chi\|^2_{L^{\infty}(\mathbb{R}, \ell^2)} \varepsilon^{-1}
e^{\varepsilon(\tau-t+r)} \mathbb{E}(\|u_{\tau-t}\|_{\mathcal{C}_r}^2)
+16 c_1^2\int_ {\tau-t}^{\tau} e^{\varepsilon s}\|\kappa(s)\|^2 ds.
\end{split}
\end{equation}
By combining \eqref{est1-1+3} through \eqref{est1-1+7}, we obtain that
 \begin{equation*}
\begin{split}
& \mathbb{E} (\sup_{\tau-t \le \varrho \le \tau} e^{\varepsilon \varrho} \|u(\varrho)\|^2)
 +(2 \lambda-\varepsilon)\int_ {\tau-t}^{\tau} e^{\varepsilon s} \mathbb{E}\|u(s)\|^2 ds
\\
\le &\, \frac{1}{2}\mathbb{E}(\sup_{\tau-t \le \varrho\le \tau}e^{\varepsilon \varrho}\|u(\varrho)\|^2)
\\
& \, +\left[ 2+4 \varepsilon^{-1} e^{\varepsilon r}+12\|\chi\|^2_{L^{\infty}(\mathbb{R}, \ell^2)} \varepsilon^{-1} e^{\varepsilon r}(1+4c_1^2) \right]
e^{\varepsilon (\tau-t)}\mathbb{E}\|u_{\tau-t}\|_{\mathcal{C}_r}^2
\\
& +\left[ 4\|\eta\|_{L^{\infty}(\mathbb{R},\ell^2)}(2+e^{\varepsilon r})+\frac{3\lambda}{2}
+12 \|\chi\|^2_{L^{\infty}(\mathbb{R}, \ell^2)}(2+e^{\varepsilon r})(1+4c_1^2)
\right] \int_{\tau-t}^\tau e^{\varepsilon s} \mathbb{E}\|u(s)\|^2 ds
\\
&\, +4\int_{\tau-t}^\tau e^{\varepsilon s}\|\eta(s)\|_1 ds+\frac{8}{3\lambda} \int_{\tau-t}^\tau e^{\varepsilon s}\|g(s)\|^2 ds
+4 (1+4c_1^2) \int_ {\tau-t}^{\tau} e^{\varepsilon s} \|\kappa(s)\|^2 ds.
\end{split}
\end{equation*}
This, together with \eqref{lamda}, implies that

\begin{equation}\label{est1-1+8}
\begin{split}
& \mathbb{E} (\sup_{\tau-t \le \varrho \le \tau} e^{\varepsilon \varrho} \|u(\varrho)\|^2)
\\
\le & 2\left[ 2+4 \varepsilon^{-1} e^{\varepsilon r}+12 \|\chi\|^2_{L^{\infty}(\mathbb{R}, \ell^2)} \varepsilon^{-1}
e^{\varepsilon r}(1+4c_1^2)\right]
e^{\varepsilon (\tau-t)}\mathbb{E}\|u_{\tau-t}\|_{\mathcal{C}_r}^2
\\
& +8\int_{\tau-t}^\tau e^{\varepsilon s}\|\eta(s)\|_1 ds+\frac{16}{3\lambda} \int_{\tau-t}^\tau e^{\varepsilon s}\|g(s)\|^2 ds
+8(1+4c_1^2) \int_ {\tau-t}^{\tau} e^{\varepsilon s} \|\kappa(s)\|^2 ds.
\end{split}
\end{equation}
On the one hand, recalling that $\mathcal{L}_{u_{\tau-t}} \in D(\tau-t)$, we find that
\begin{eqnarray*}
e^{\varepsilon (r-t)} \mathbb{E}\|u_{\tau-t}\|_{\mathcal{C}_r}^2\!\!\!& \le &\!\!\! e^{\varepsilon (r-\tau)} \left[ e^{\varepsilon (\tau-t)}
\mathbb{E}\|D(\tau-t)\|_{\mathcal{P}_2(\mathcal{C}_r)}^2\right]
\rightarrow 0
\end{eqnarray*}
as $t \rightarrow \infty$. Thus, there exists $T=T(\tau, D)>r$ such that for all $t \ge  T$,
\begin{eqnarray}\label{est1-1+9}
e^{\varepsilon (r-t)} \mathbb{E}\|u_{\tau-t}\|_{\mathcal{C}_r}^2<1.
\end{eqnarray}
On the other hand, taking $t \ge T$, by \eqref{est1-1+8} and \eqref{est1-1+9} we have that
\begin{equation*}
\begin{split}
 \mathbb{E}(\sup_{\tau-r \le \varrho \le \tau}\|u(\varrho)\|^2)
\le \,   e^{-\varepsilon(\tau-r)}
\mathbb{E}(\sup_{\tau-t \le \varrho \le \tau} e^{\varepsilon \varrho}\|u(\varrho)\|^2)
\end{split}
\end{equation*}
\begin{equation*}
\begin{split}
\le & \; 2\left[ 2+4 \varepsilon^{-1} e^{\varepsilon r}+12\|\chi\|^2_{L^{\infty}(\mathbb{R}, \ell^2)} \varepsilon^{-1}
e^{\varepsilon r}(1+4c_1^2)\right]
e^{\varepsilon (r-t)}\mathbb{E}\|u_{\tau-t}\|_{\mathcal{C}_r}^2
\\
& +8 e^{\varepsilon r}\int_{\tau-t}^\tau e^{\varepsilon (s-\tau)}\|\eta(s)\|_1 ds
+\frac{16}{3\lambda} e^{\varepsilon r} \int_{\tau-t}^\tau e^{\varepsilon (s-\tau)}\|g(s)\|^2 ds
\\
& +8(1+4c_1^2) e^{\varepsilon r} \int_ {\tau-t}^{\tau} e^{\varepsilon (s-\tau)} \|\kappa(s)\|^2 ds
\\
\le & 2\left[ 2+4 \varepsilon^{-1} e^{\varepsilon r}+12\|\chi\|^2_{L^{\infty}(\mathbb{R}, \ell^2)} \varepsilon^{-1}
e^{\varepsilon r}(1+4c_1^2)\right]+8(1+4c_1^2)\varepsilon^{-1} e^{\varepsilon r}\|\kappa\|^2_{L^{\infty}(\mathbb{R},\ell^2)}
\\
& +8 e^{\varepsilon r}\int_{-\infty}^\tau e^{\varepsilon (s-\tau)}\|\eta(s)\|_1 ds
+\frac{16}{3\lambda} e^{\varepsilon r} \int_{-\infty}^\tau e^{\varepsilon (s-\tau)}\|g(s)\|^2 ds,
\end{split}
\end{equation*}
where the convergence of the infinite integral is guaranteed by \eqref{g^2}. Thus, the proof of Lemma \ref{est-1+} is complete.
\end{proof}

Next, we derive the uniform estimates for the tails of segments of solutions. These estimates will be crucial for addressing the non-compactness of Sobolev embeddings on unbounded domains, and will play a key role in proving the asymptotic compactness of the distributions of solutions.

\begin{lm}\label{est-2}
Assume that  {\bf  (H1)}-{\bf (H3)}, \eqref{lamda} and \eqref{g^2} hold.
Then, for every $\delta>0$, $\tau \in \mathbb{R}$ and $D=\{D(t) : t \in \mathbb{R}\} \in \mathcal{D}_0$,
there exist $T=T(\delta, \tau, D)>r$ and $N=N(\delta, \tau)\in \mathbb{N}^+$
such that for all $ t \ge T $ and $n\ge N$, the solution $u$ of \eqref{ob-3}-\eqref{ob-4} satisfies,
$$\mathbb{E}(\sup_{\tau-r \le \varrho \le \tau} \sum_{|i| \ge n} |u_i(\varrho, \tau - t, u_{ \tau -t })|^2)<\delta,$$
where
$u_{ \tau - t }  \in L^2  ( \Omega, \mathcal{F}_{ \tau - t }, \mathcal{C}_r)$
with
$\mathcal{L}_{ u_{ \tau - t } } \in  D(\tau-t)$.
\end{lm}

\begin{proof}
Let $\theta: \mathbb{R} \rightarrow [0,1]$ be a smooth function such that
\begin{align}\label{tail-est-1}
   \theta(s)=0  \  \text{ for }
   |s| \le 1;
   \quad
   \theta(s)=1 \ \text{ for } \   |s| \ge 2.
\end{align}
Given $n\in \mathbb{N}$, denote by  $\theta_n=( \theta(\frac{i}{n}))_{i \in \mathbb{Z}}$ and
$\theta_n u=( \theta(\frac{i}{n})u_i)_{i \in \mathbb{Z}}$ for $u=(u_i)_{i \in \mathbb{Z}}$. For convenience,
we denote by $u(t):= u(t, \tau, u_\tau)$  the solution of \eqref{ob-3}-\eqref{ob-4} with initial data $u_\tau$
at initial time $ \tau$.

By using \eqref{ob-3}, we have
\begin{equation*}\label{tail-est-2}
\begin{split}
&\; d(\theta_n u(t))+ \nu  \theta_n  A u (t)+ \lambda(\theta_n u (t))dt
\\
&=\; (\theta_n f(t, u(t), u(t-r),\mathcal{L}_{u(t)})+ \theta_n g(t))dt+\theta_n \tilde{\sigma}(t, u(t),u(t-r), \mathcal{L}_{u(t)}) dW(t).
\end{split}
\end{equation*}
Applying the It\^{o} formula, we get for all $\varrho \ge \tau-t$,
 \begin{equation} \label{tail-est-3}
 \begin{split}
 & e^{\varepsilon \varrho} \|\theta_n u(\varrho)\|^2+2\nu \int_{\tau-t}^\varrho e^{\varepsilon s}(Bu(s), B(\theta_n^2 u(s))) ds
 +(2 \lambda-\varepsilon)\int_ {\tau-t}^\varrho e^{\varepsilon s}\|\theta_n u(s)\|^2 ds
 \\
 = &\, e^{\varepsilon (\tau-t)}\|\theta_n u_{\tau-t}(0)\|^2+ 2 \int_ {\tau-t}^\varrho e^{\varepsilon s}
 (\theta_n f(s, u(s), u(s-r), \mathcal{L}_{u (s)}),  \theta_n u(s))ds
 \\
 &+ 2 \int_ {\tau-t}^\varrho e^{\varepsilon s}
 (\theta_n g(s),  \theta_n u(s))ds
 \\
 & +\int_ {\tau-t}^\varrho e^{\varepsilon s}\|\theta_n \tilde{\sigma}(s, u(s), u(s-r), \mathcal{L}_{u(s)})\|_{L_2(\ell^2, \ell^2)}^2 ds
 \\
 &+ 2 \int_ {\tau-t}^{\varrho} e^{\varepsilon s}\left(\theta_n^2  u(s), \tilde{\sigma}(s, u(s), u(s-r),\mathcal{L}_{u(s)}) dW(s)\right),
 \end{split}
 \end{equation}
$\mathbb{P}$-almost surely. In view of the following fact that
\begin{eqnarray*}
(Bu(s), B(\theta_n^2 u(s)))\!\!\!&=&\!\!\! \sum_{i\in \mathbb{Z}}(u_{i+1}(s)-u_i(s))(\theta(\frac{i+1}{n})^2u_{i+1}(s)-\theta(\frac{i}{n})^2u_i(s))
\\
\!\!\!&=&\!\!\! \sum_{i\in \mathbb{Z}} \theta(\frac{i+1}{n})^2(u_{i+1}(s)-u_{i}(s))^2
+\sum_{i\in \mathbb{Z}}(\theta(\frac{i+1}{n})^2-\theta(\frac{i}{n})^2)(u_{i+1}(s)-u_{i}(s))  \textcolor{red}{ u_{i}(s)},
\end{eqnarray*}
taking the supremum and expectation of \eqref{tail-est-3}, we obtain
\begin{equation} \label{tail-est-5}\small
 \begin{split}
 & \mathbb{E} (\sup_{\tau-t \le \varrho \le \tau} e^{\varepsilon \varrho} \|\theta_n u(\varrho)\|^2)
 +(2 \lambda-\varepsilon)\int_ {\tau-t}^{\tau} e^{\varepsilon s} \mathbb{E}\|\theta_n u(s)\|^2 ds
 \\
  \le & 2 e^{\varepsilon (\tau-t)} \mathbb{E}\|\theta_n u_{\tau-t}(0)\|^2
  +4 \nu \mathbb{E} \sup_{\tau-t \le \varrho \le \tau} \int_{\tau-t}^{\varrho}
 e^{\varepsilon s}\sum_{i\in \mathbb{Z}}(\theta(\frac{i+1}{n})^2-\theta(\frac{i}{n})^2)(u_{i+1}(s)-u_{i}(s))u_{i}(s) ds
 \\
 &+4  \mathbb{E} \sup_{\tau-t \le \varrho \le \tau} \int_ {\tau-t}^{\varrho} e^{\varepsilon s}
 (\theta_n f(s, u(s), u(s-r), \mathcal{L}_{u (s)}),  \theta_n u(s))ds
 \\
  & + 4  \mathbb{E} \sup_{\tau-t \le \varrho \le \tau}  \int_ {\tau-t}^{\varrho} e^{\varepsilon s}
 (\theta_n g(s),  \theta_n u(s))ds
 \\
 & + 2 \mathbb{E} \sup_{\tau-t \le \varrho \le \tau}  \int_ {\tau-t}^{\varrho} e^{\varepsilon s}\|\theta_n \tilde{\sigma}(s, u(s), u(s-r), \mathcal{L}_{u(s)})\|_{L_2(\ell^2, \ell^2)}^2 ds
 \\
 &+ 4 \mathbb{E} \sup_{\tau-t \le \varrho \le \tau}  \biggl|\int_ {\tau-t}^{\varrho} e^{\varepsilon s}\left(\theta_n^2  u(s), \tilde{\sigma}(s, u(s), u(s-r),\mathcal{L}_{u(s)}) dW(s)\right) \biggr|.
 \end{split}
 \end{equation}
Next, we estimate each terms on the right-hand side of \eqref{tail-est-5} individually.  For the second term on the right-hand side of
\eqref{tail-est-5}, we have
\begin{equation}\label{tail-est-6}
\begin{split}
& 4 \nu \mathbb{E} \sup_{\tau-t \le \varrho \le \tau} \int_{\tau-t}^{\varrho}
 e^{\varepsilon s}\sum_{i\in \mathbb{Z}}(\theta(\frac{i+1}{n})^2-\theta(\frac{i}{n})^2)(u_{i+1}(s)-u_{i}(s))u_{i}(s) ds
\\
\le & \,  8 \nu c_0 n^{-1}\mathbb{ E} \int_{\tau-t}^\tau e^{\varepsilon s}\sum_{i \in \mathbb{Z}}|u_{i+1}(s)-u_i(s)| \cdot |u_i(s)|ds
\\
\le & \,  16 \nu c_0 n^{-1} \int_{\tau-t}^\tau e^{\varepsilon s} \mathbb{E}\|u(s)\|^2 ds,
\end{split}
\end{equation}
where $c_0$ is a constant such that $|\theta'(s)| \le c_0$ for all $s \in \mathbb{R}$. For the third term on the right-hand side of
\eqref{tail-est-5}, by \eqref{fi-1} we have
\begin{equation}\label{tail-est-7}
\begin{split}
& 4  \mathbb{E} \sup_{\tau-t \le \varrho \le \tau} \int_ {\tau-t}^{\varrho} e^{\varepsilon s}
 (\theta_n f(s, u(s), u(s-r), \mathcal{L}_{u (s)}),  \theta_n u(s))ds
\\
\le & \,  -4 \alpha \mathbb{ E} \int_{\tau-t}^\tau e^{\varepsilon s}\sum_{i \in \mathbb{Z}}
\theta^2(\frac{i}{n})|u_i(s)|^p ds
\\
&\, +4 \mathbb{ E} \int_{\tau-t}^\tau e^{\varepsilon s}\sum_{i \in \mathbb{Z}}
\theta^2(\frac{i}{n})\eta_i(s)(1+|u_i(s)|^2+|u_i(s-r)|^2+  \mathbb{E}|u_i(s)|^2)ds
\\
\le & \,  4\int_{\tau-t}^\tau e^{\varepsilon s} \sum_{|i| \ge n} |\eta_i(s)| ds
+ 4\|\eta\|_{L^{\infty}(\mathbb{R},\ell^2)}(2+e^{\varepsilon r})\int_{\tau-t}^\tau e^{\varepsilon s} \mathbb{E} \|\theta_n u(s)\|^2 ds
\\
&\,+4 \varepsilon^{-1} e^{\varepsilon(\tau-t+r)} \mathbb{E} \|\theta_n u_{\tau-t}\|_{\mathcal{C}_r}^2.
\end{split}
\end{equation}
For the fourth term on the right-hand side of \eqref{tail-est-5}, we have
\begin{equation}\label{tail-est-8}
\begin{split}
 4  \mathbb{E} \sup_{\tau-t \le \varrho \le \tau}  \int_ {\tau-t}^{\varrho} e^{\varepsilon s}
 (\theta_n g(s),  \theta_n u(s))ds
\le  \, \frac{8}{3\lambda} \int_{\tau-t}^\tau e^{\varepsilon s}\|\theta_n g(s)\|^2 ds
+\frac{3\lambda}{2}\int_{\tau-t}^\tau e^{\varepsilon s}  \mathbb{ E} \|\theta_n u(s)\|^2 ds.
\end{split}
\end{equation}
To estimate the fifth term on the right-hand side of \eqref{tail-est-5}, we apply \eqref{segamai-1} and observe that
for any $u=(u_k)_{k \in \mathbb{Z}}, v=(v_k)_{k \in \mathbb{Z}} \in \ell^2$ and $\mu=(\mu_k)_{k \in \mathbb{Z}} \in \mathcal{L}^2$,
\begin{equation*}
\begin{split}
 \|\theta_n \tilde{\sigma}(s,u,v,\mu)\|_{L_2(\ell^2,\ell^2)}^2
\le  6 \|\chi(s)\|^2(\|\theta_n u\|^2+\|\theta_n v\|^2
+\sum_{k \in \mathbb{Z}}\theta^2(\frac{k}{n})\mathbb{W}_2^{\mathbb{R}}(\mu_k,\delta_{0})^2)+2 \|\theta_n \kappa(s)\|^2
\end{split}
\end{equation*}
which implies that
\begin{equation}\label{tail-est-9}
\begin{split}
& 2 \mathbb{E} \sup_{\tau-t \le \varrho \le \tau}  \int_ {\tau-t}^{\varrho} e^{\varepsilon s}\|\theta_n \tilde{\sigma}(s, u(s), u(s-r), \mathcal{L}_{u(s)})\|_{L_2(\ell^2, \ell^2)}^2 ds
\\
\le & \,  12\|\chi\|^2_{L^{\infty}(\mathbb{R}, \ell^2)}(2+e^{\varepsilon r}) \int_ {\tau-t}^{\tau} e^{\varepsilon s} \mathbb{E} \|\theta_n u(s)\|^2 ds
\\
& \, + 12 \|\chi\|^2_{L^{\infty}(\mathbb{R}, \ell^2)} \varepsilon^{-1}
e^{\varepsilon(\tau-t+r)} \mathbb{E}(\|\theta_n u_{\tau-t}\|_{\mathcal{C}_r}^2)
+4\int_ {\tau-t}^{\tau} e^{\varepsilon s} \sum_{|i| \ge n}|\kappa_i(s)|^2 ds.
\end{split}
\end{equation}
For the last term on the right-hand side of \eqref{tail-est-5}, by \eqref{tail-est-9} we have
\begin{equation}\label{tail-est-10}
\begin{split}
&   4 \mathbb{E} \sup_{\tau-t \le \varrho \le \tau}  \biggl|\int_ {\tau-t}^{\varrho} e^{\varepsilon s}\left(\theta_n^2  u(s), \tilde{\sigma}(s, u(s), u(s-r),\mathcal{L}_{u(s)}) dW(s)\right) \biggr|
\\
\le &\, 4c_1\mathbb{E}\{\sup_{\tau-t \le \varrho \le \tau} e^{\varepsilon \varrho}\|\theta_n u(\varrho)\|^2 \cdot
(\int_{\tau-t}^\tau \|\theta_n \tilde{\sigma}(s, u(s), u(s-r), \mathcal{L}_{u(s)})\|_{L_2(\ell^2, \ell^2)}^2 ds)\}^{\frac{1}{2}}
\\
\le &\, \frac{1}{2} \mathbb{E}(\sup_{\tau-t \le \varrho\le \tau}e^{\varepsilon \varrho}\|\theta_n u(\varrho)\|^2)
+ 8 c_1^2 \int_{\tau-t}^\tau
\mathbb{E}\|\theta_n \tilde{\sigma}(s, u(s), u(s-r), \mathcal{L}_{u(s)})\|_{L_2(\ell^2, \ell^2)}^2 ds
\\
\le &\,  \frac{1}{2}\mathbb{E}(\sup_{\tau-t \le \varrho\le \tau}e^{\varepsilon \varrho}\|\theta_n u(\varrho)\|^2)
+48 c_1^2 \|\chi\|^2_{L^{\infty}(\mathbb{R}, \ell^2)}(2+e^{\varepsilon r}) \int_ {\tau-t}^{\tau} e^{\varepsilon s} \mathbb{E} \|\theta_n u(s)\|^2 ds
\\
&+48 c_1^2\|\chi\|^2_{L^{\infty}(\mathbb{R}, \ell^2)} \varepsilon^{-1} e^{\varepsilon(\tau-t+r)}
\mathbb{E}\|\theta_n u_{\tau-t}\|_{\mathcal{C}_r}^2
 +16 c_1^2 \int_ {\tau-t}^{\tau} e^{\varepsilon s} \sum_{|i| \ge n}|\kappa_i(s)|^2 ds,
\end{split}
\end{equation}
Combining \eqref{tail-est-5}-\eqref{tail-est-10}, we get that
 \begin{equation*}
\begin{split}
& \mathbb{E} (\sup_{\tau-t \le \varrho \le \tau} e^{\varepsilon \varrho} \|\theta_n u(\varrho)\|^2)
 +(2 \lambda-\varepsilon)\int_ {\tau-t}^{\tau} e^{\varepsilon s} \mathbb{E}\|\theta_n u(s)\|^2 ds
\\
\le &\, \frac{1}{2}\mathbb{E}(\sup_{\tau-t \le \varrho\le \tau}e^{\varepsilon \varrho}\|\theta_n u(\varrho)\|^2)
   +16 \nu c_0 n^{-1} \int_{\tau-t}^\tau e^{\varepsilon s} \mathbb{E}\|u(s)\|^2 ds
\\
& \, +\left[ 2+4 \varepsilon^{-1} e^{\varepsilon r}+12\|\chi\|^2_{L^{\infty}(\mathbb{R}, \ell^2)} \varepsilon^{-1} e^{\varepsilon r}(1+4c_1^2) \right]
e^{\varepsilon (\tau-t)}\mathbb{E}\|\theta_n u_{\tau-t}\|_{\mathcal{C}_r}^2
\\
& +\left[ 4\|\eta\|_{L^{\infty}(\mathbb{R},\ell^2)}(2+e^{\varepsilon r})+\frac{3\lambda}{2}
         +12 \|\chi\|^2_{L^{\infty}(\mathbb{R}, \ell^2)}(2+e^{\varepsilon r})(1+4c_1^2)
\right] \int_{\tau-t}^\tau e^{\varepsilon s} \mathbb{E}\|\theta_n u(s)\|^2 ds
\\
&\, +4 \int_{\tau-t}^\tau e^{\varepsilon s} \sum_{|i| \ge n} |\eta_i(s)| ds+ \frac{8}{3\lambda}
\int_{\tau-t}^\tau e^{\varepsilon s} \sum_{|i|\ge n}|g_i(s)|^2 ds
+4(1+4c_1^2) \int_ {\tau-t}^{\tau} e^{\varepsilon s} \sum_{|i| \ge n}|\kappa_i(s)|^2 ds.
\end{split}
\end{equation*}
Following \eqref{lamda} and Lemma \ref{est-2},  there exist $b_1=b_1(\tau)$ and $T_1=T_1(\tau, D) \ge r$ such that for all $t \ge T_1$,
\begin{equation*}
\begin{split}
& \mathbb{E} (\sup_{\tau-t \le \varrho \le \tau} e^{\varepsilon \varrho} \|\theta_n u(\varrho)\|^2)
\\
\le & 2\left[ 2+4 \varepsilon^{-1} e^{\varepsilon r}+12 \|\chi\|^2_{L^{\infty}(\mathbb{R}, \ell^2)} \varepsilon^{-1}
e^{\varepsilon r}(1+4c_1^2)\right]
e^{\varepsilon (\tau-t)}\mathbb{E}\|\theta_n u_{\tau-t}\|_{\mathcal{C}_r}^2
+32 \nu c_0 n^{-1}b_1
\\
& +8 \int_{\tau-t}^\tau e^{\varepsilon s} \sum_{|i| \ge n} |\eta_i(s)| ds+ \frac{16}{3\lambda}
\int_{\tau-t}^\tau e^{\varepsilon s}  \sum_{|i|\ge n}|g_i(s)|^2 ds
+8(1+4c_1^2) \int_ {\tau-t}^{\tau} e^{\varepsilon s} \sum_{|i| \ge n}|\kappa_i(s)|^2 ds.
\end{split}
\end{equation*}
For $t \ge T_1$,  we further find that
\begin{equation} \label{tail-est-12}
\begin{split}
\quad &\; \mathbb{E}(\sup_{\tau-r \le \varrho \le \tau}\|\theta_n u(\varrho)\|^2)  \le  e^{-\varepsilon(\tau-r)}
\mathbb{E}(\sup_{\tau-t \le \varrho \le \tau} e^{\varepsilon \varrho}\|\theta_n u(\varrho)\|^2)
\\
\le & \; 2\left[ 2+4 \varepsilon^{-1} e^{\varepsilon r}+12\|\chi\|^2_{L^{\infty}(\mathbb{R}, \ell^2)} \varepsilon^{-1}
e^{\varepsilon r}(1+4c_1^2)\right]e^{\varepsilon (r-t)} \mathbb{E}\|u_{\tau-t}\|_{\mathcal{C}_r}^2
+32 e^{-\varepsilon(\tau-r)}\nu c_0 n^{-1}b_1
\\
& +8 e^{\varepsilon r} \int_{-\infty}^\tau e^{\varepsilon (s-\tau)} \sum_{|i| \ge n} |\eta_i(s)| ds
+ \frac{16}{3\lambda} e^{\varepsilon r}\int_{-\infty}^\tau e^{\varepsilon (s-\tau)} \sum_{|i|\ge n}|g_i(s)|^2 ds
\\
&\,+8 e^{\varepsilon r}(1+4c_1^2) \int_ {-\infty}^{\tau} e^{\varepsilon (s-\tau)} \sum_{|i| \ge n}|\kappa_i(s)|^2 ds.
\end{split}
\end{equation}
For the first term on the right-hand side of \eqref{tail-est-12}, since $\mathcal{L}_{u_{\tau-t}} \in D(\tau-t)$, we get that
\begin{eqnarray*}
e^{\varepsilon (r-t)} \mathbb{E}\|u_{\tau-t}\|_{\mathcal{C}_r}^2\!\!\!& \le &\!\!\! e^{\varepsilon (r-\tau)} \{e^{\varepsilon (\tau-t)}
\mathbb{E}\|D(\tau-t)\|_{\mathcal{P}_2(\mathcal{C}_r)}^2\}\rightarrow 0
\end{eqnarray*}
as $t \rightarrow \infty$. Thus, there exists $T_2=T_2(\delta, \tau, D) \ge T_1$ such that for all $t \ge  T_2$,
$$
 2\left[ 2+4 \varepsilon^{-1} e^{\varepsilon r}+12\|\chi\|^2_{L^{\infty}(\mathbb{R}, \ell^2)} \varepsilon^{-1}
e^{\varepsilon r}(1+4c_1^2)\right] e^{\varepsilon (r-t)} \mathbb{E}\|u_{\tau-t}\|_{\mathcal{C}_r}^2< \frac{\delta}{4}.
$$
For the second term on the right-hand side of \eqref{tail-est-12}, there exists $N_1=N_1(\delta, \tau) \in \mathbb{N}^+$
such that for all $n \ge N_1$,
\begin{eqnarray*}
32 \nu c_0 n^{-1}b_1<\frac{\delta}{4}.
\end{eqnarray*}
For the third and fourth terms on the right-hand side of \eqref{tail-est-12}, by using \eqref{g^2} we notice that
\begin{eqnarray*}
8 e^{\varepsilon r} \int_{-\infty}^\tau e^{\varepsilon (s-\tau)} \sum_{|i| \ge n} |\eta_i(s)| ds \!\!\!&\le&\!\!\! 8 e^{\varepsilon r} \int_{-\infty}^\tau e^{\varepsilon (s-\tau)}\|\eta(s)\|_1 ds<\infty
\end{eqnarray*}
and
\begin{eqnarray*}
\frac{16}{3\lambda} e^{\varepsilon r}\int_{-\infty}^\tau e^{\varepsilon (s-\tau)} \sum_{|i|\ge n}|g_i(s)|^2 ds
\!\!\!&\le&\!\!\! \frac{16}{3\lambda} e^{\varepsilon r}\int_{-\infty}^\tau e^{\varepsilon (s-\tau)}\|g(s)\|^2 ds<\infty.
\end{eqnarray*}
By the dominated convergence theorem, there exists $N_2=N_2(\delta, \tau) \ge N_1$ such that for all $n \ge N_2$,
\begin{eqnarray*}
8 e^{\varepsilon r} \int_{-\infty}^\tau e^{\varepsilon (s-\tau)} \sum_{|i| \ge n} |\eta_i(s)| ds+\frac{16}{3\lambda} e^{\varepsilon r}\int_{-\infty}^\tau e^{\varepsilon (s-\tau)} \sum_{|i|\ge n}|g_i(s)|^2 ds <\frac{\delta}{4},
\end{eqnarray*}
For the last term on the right-hand side of \eqref{tail-est-12}, using the fact that $\kappa \in L^\infty(\mathbb{R}, \ell^2)$
we find that
\begin{eqnarray*}
8 e^{\varepsilon r}(1+4c_1^2) \int_ {-\infty}^{\tau} e^{\varepsilon (s-\tau)} \sum_{|i| \ge n}|\kappa_i(s)|^2 ds
\!\!\!& \le &\!\!\! 8 e^{\varepsilon r}(1+4c_1^2)\int_ {-\infty}^{\tau} e^{\varepsilon (s-\tau)}|\kappa(s)|^2 ds
\\
\!\!\!& \le &\!\!\! 8 e^{\varepsilon r}\varepsilon^{-1}(1+4c_1^2) \|\kappa\|_{L^{\infty}(\mathbb{R}, \ell^2)}.
\end{eqnarray*}
By the dominated convergence theorem again, there exists $N_3=N_3(\delta, \tau) \ge N_3$ such that for all $n \ge N_3$,
\begin{eqnarray*}
8 e^{\varepsilon r}(1+4c_1^2) \int_ {-\infty}^{\tau} e^{\varepsilon (s-\tau)} \sum_{|i| \ge n}|\kappa_i(s)|^2 ds
<\frac{\delta}{4},
\end{eqnarray*}
Combining the above analysis, we get for all $t \ge T_2$ and $n \ge N_3$ that
\begin{eqnarray*}
\mathbb{E}(\sup_{\tau-r \le \varrho \le \tau}\|\theta_n u(\varrho)\|^2)<\delta,
\end{eqnarray*}
which implies that for all $t \ge T_2$ and $n \ge N_3$,
\begin{eqnarray*}
\mathbb{E}(\sup_{\tau-r \le \varrho \le \tau}\sum_{|i| \ge 2n}\|u_i(\varrho)\|^2)<\mathbb{E}(\sup_{\tau-r \le \varrho \le \tau}\|\theta_n u(\varrho)\|^2)<\delta.
\end{eqnarray*}
This completes the proof of Lemma \ref{est-2}.
\end{proof}

At the end of this section, we derive the uniform estimates of segemts of solutions of \eqref{ob-3}-\eqref{ob-4}
in $L^4(\Omega, \mathcal{C}_r)$.

\begin{lm}\label{est-3}
Assume that {\bf  (H1)}-{\bf (H3)}, \eqref{lamda} and \eqref{g^4}
hold. Then, for every $\tau \in \mathbb{R}$ and $D=\{D(t): t \in \mathbb{R}\} \in \mathcal{D}$,
there exists $T=T(\tau, D)>r$ such that for all $t \ge T$, the solution $u_\tau$ of
\eqref{ob-3}-\eqref{ob-4} satisfies
\begin{align}\label{est3-1}
\mathbb{E}\left(\|u_\tau(\cdot, \tau - t, u_{\tau -t}  )\|_{\mathcal{C}_r}^4\right)
\le M_3+M_3\int_{-\infty}^\tau e^{2\varepsilon(s-\tau)}\|g(s)\|^4 ds,
\end{align}
where
$u_{\tau-t} \in L^4(\Omega, \mathcal{F}_{\tau-t}, \mathcal{C}_r)$
with
$\mathcal{L}_{u_{ \tau - t } } \in D(\tau-t)$, $\varepsilon>0$ is the same number as in \eqref{lamda-1},
and $ M_3>0 $  is a constant independent of $\tau$ and $D$.
\end{lm}

\begin{proof}
 Let  $u(t)= u(t, \tau-t, u_{\tau-t)}$ be the solution of \eqref{ob-3}-\eqref{ob-4} with initial data $u_{\tau-t}$
at initial time $\tau-t$.

By the It\^{o} formula, taking the supremum and expectation, we obtain that for all $\varrho\ge \tau-t$,
\begin{equation}\label{segment-est-3}
\begin{split}
& \mathbb{E}(\sup_{\tau-t \le \varrho \le \tau} e^{2\varepsilon \varrho} \|u(\varrho)\|^4)
 + 2 (2 \lambda-\varepsilon) \mathbb{E} \int_ {\tau-t}^{\tau} e^{ 2 \varepsilon s}\|u( s )\|^4 ds
\\
\le &\, 2 e^{2 \varepsilon (\tau-t)} \mathbb{E}\|u_{\tau-t}(0)\|^4
+8 \mathbb{E} \sup_{\tau-t \le \varrho \le \tau} \int_ {\tau-t}^{\varrho}
      e^{2 \varepsilon s}\|u (s)\|^2 (f(s, u(s), u(s-r), \mathcal{L}_{u (s)} ),  u(s)) ds
\\
& + 8 \mathbb{E} \sup_{\tau-t \le \varrho \le \tau} \int_ {\tau-t}^{\varrho} e^{2\varepsilon s}\|u (s)\|^2  (g(s), u(s)) ds
\\
&+ 4 \mathbb{E} \sup_{\tau-t \le \varrho \le \tau} \int_ {\tau-t}^{\varrho} e^{2 \varepsilon s}\|u (s)\|^2
\|\tilde{\sigma}(s, u(s), u(s-r), \mathcal{L}_{u  (s)} )\|_{L_2(\ell^2,\ell^2)}^2 ds
\\
 &+ 8 \mathbb{E} \sup_{\tau-t \le \varrho \le \tau} \int_ {\tau-t}^{\varrho}
 e^{2 \varepsilon s} {\rm tr} \{(u(s)\otimes u(s))(\tilde{\sigma}(s, u(s), u(s-r), \mathcal{L}_{u  (s)}))
 \\
 & \qquad\qquad\qquad\qquad\qquad\quad  \cdot(\tilde{\sigma}(s, u(s), u(s-r), \mathcal{L}_{u  (s)}))^*\} ds
\\
& + 8 \mathbb{E} \sup_{\tau-t \le \varrho \le \tau} \biggl|\int_ {\tau-t}^{\varrho}
e^{2\varepsilon s} \|u (s)\|^2 \left( u (s), \tilde{\sigma}(s, u(s),u(s-r), \mathcal{L}_{u(s)}) dW(s)\right) \biggr|,
 \end{split}
 \end{equation}
where $u(s)\otimes u(s)=u(s)(u(s), \cdot)$.

Next, we estimate each terms on the right-hand side of \eqref{segment-est-3} individually.  For the second term on the right-hand side of \eqref{segment-est-3}, by \eqref{fi-2} we have
\begin{equation}\label{segment-est-4}
\begin{split}
& 8 \mathbb{E} \sup_{\tau-t \le \varrho \le \tau} \int_ {\tau-t}^{\varrho}
      e^{2 \varepsilon s}\|u (s)\|^2 (f(s, u(s), u(s-r), \mathcal{L}_{u (s)} ),  u(s)) ds
\\
\le &  \;  8 \mathbb{E} \sup_{\tau-t \le \varrho \le \tau}  \left(-\alpha \int_ {\tau-t}^{\varrho}
      e^{2 \varepsilon s}\|u (s)\|^2 \|u(s)\|_p^p ds \right)+  8 \mathbb{E}\int_ {\tau-t}^{\tau}
      e^{2 \varepsilon s}\|\eta(s)\|_1 \|u(s)\|^2 ds
\\
 & + 8\int_ {\tau-t}^{\tau} e^{2 \varepsilon s}\|\eta(s)\| \left[\mathbb{E }\|u(s)\|^4+\mathbb{E }\|u(s)\|^2\|u(s-r)\|^2
 + (\mathbb{E }\|u(s)\|^2)^2 \right] ds
\\
\le &  \; \frac{16}{\lambda} \int_ {\tau-t}^{\tau}e^{2 \varepsilon s}\|\eta(s)\|_1^2ds+
2 \varepsilon^{-1} \|\eta\|_{L^{\infty}(\mathbb{R},\ell^2)} e^{2\varepsilon(\tau-t+r)} \mathbb{E}\|u_{\tau-t}\|_{\mathcal{C}_r}^4
\\
 & +\left[\lambda+4\|\eta\|_{L^{\infty}(\mathbb{R},\ell^2)}(5+e^{2\varepsilon r})\right]
\int_ {\tau-t}^{\tau} e^{2 \varepsilon s} \mathbb{E}\|u(s)\|^4 ds.
\end{split}
\end{equation}
For the third term on the right-hand side of \eqref{segment-est-3}, we have
\begin{equation}\label{segment-est-5}
\begin{split}
&  8 \mathbb{E} \sup_{\tau-t \le \varrho \le \tau} \int_ {\tau-t}^{\varrho} e^{2\varepsilon s}\|u (s)\|^2  (g(s), u(s)) ds
\\
\le & 8 \mathbb{E} \int_ {\tau-t}^{\tau} e^{\frac{3}{2}\varepsilon s}\|u (s)\|^3  \cdot e^{\frac{1}{2}\varepsilon s} \|g(s)\|ds
\\
\le &   \frac{5}{2}\lambda\int_ {\tau-t}^{\tau} e^{2 \varepsilon s} \mathbb{E }\|u(s)\|^4 ds
+\frac{3456}{125\lambda^3} \int_{\tau-t}^{\tau}  e^{2 \varepsilon s} \|g(s)\|^4 ds.
\end{split}
\end{equation}
For the fourth and fifth and terms on the right-hand side of \eqref{segment-est-3}, by \eqref{bound-2} we have
\begin{equation}\label{segment-est-6}
\begin{split}
&   4 \mathbb{E} \sup_{\tau-t \le \varrho \le \tau} \int_ {\tau-t}^{\varrho} e^{2 \varepsilon s}\|u (s)\|^2
\|\tilde{\sigma}(s, u(s), u(s-r), \mathcal{L}_{u  (s)} )\|_{L_2(\ell^2,\ell^2)}^2 ds
\\
 &+ 8 \mathbb{E} \sup_{\tau-t \le \varrho \le \tau} \int_ {\tau-t}^{\varrho} e^{2 \varepsilon s} {\rm tr} \{(u(s)\otimes u(s))(\tilde{\sigma}(s, u(s), u(s-r), \mathcal{L}_{u  (s)}))
 \\
 & \qquad\qquad\qquad\qquad\qquad\quad  \cdot(\tilde{\sigma}(s, u(s), u(s-r), \mathcal{L}_{u  (s)}))^*\} ds
\\
\le &  12 \mathbb{E} \int_ {\tau-t}^{\tau} e^{2 \varepsilon s}\|u (s)\|^2
\|\tilde{\sigma}(s, u(s), u(s-r), \mathcal{L}_{u  (s)} )\|_{L_2(\ell^2,\ell^2)}^2 ds
\\
\le &  72 \int_ {\tau-t}^{\tau} e^{2 \varepsilon s} \|\chi(s)\|^2 (\mathbb{E} \|u(s)\|^4+  \mathbb{E} \|u(s)\|^2 \|u(s-r)\|^2+
(\mathbb{ E}\|u(s)\|^2)^2) ds
\\
&+ 12 \int_ {\tau-t}^{\tau} e^{2 \varepsilon s} \|\kappa(s)\|^2 ds +
12 \int_{\tau-t}^{\tau} e^{2 \varepsilon s} \|\kappa(s)\|^2\; \mathbb{E} \|u (s)\|^2 ds
\\
\le &  \left[36\|\chi\|^2_{L^{\infty}(\mathbb{R},\ell^2)}(5+e^{2 \varepsilon r})
+12\|\kappa\|^2_{L^{\infty}(\mathbb{R},\ell^2)} \right] \int_ {\tau-t}^{\tau} e^{2 \varepsilon s} \|u(s)\|^4 ds
\\
&+  18 \varepsilon^{-1} e^{2 \varepsilon (\tau-t+r)}\|\chi\|^2_{L^{\infty}(\mathbb{R},\ell^2)} \mathbb{E}\|u_{\tau-t}\|_{\mathcal{C}_r}^4
+6 \varepsilon^{-1} \|\kappa\|_{L^{\infty}(\mathbb{R},\ell^2)}^2 e^{2 \varepsilon \tau}.
\end{split}
\end{equation}
For the last term on the right-hand side of \eqref{segment-est-3}, by \eqref{segment-est-6} we have
\begin{equation}\label{segment-est-7}
\begin{split}
&   8 \mathbb{E} \sup_{\tau-t \le \varrho \le \tau} \biggl|\int_ {\tau-t}^{\varrho} e^{2\varepsilon s} \|u (s)\|^2 \left( u (s), \tilde{\sigma}(s, u(s),u(s-r), \mathcal{L}_{u(s)}) dW(s)\right) \biggr|
\\
\le &\, 8c_1\mathbb{E}\{\sup_{\tau-t \le \varrho \le \tau} e^{2\varepsilon \varrho}\|u(\varrho)\|^4 \cdot
(\int_{\tau-t}^\tau e^{2\varepsilon s}\|u(s)\|^2 \|\tilde{\sigma}(s, u(s), u(s-r), \mathcal{L}_{u(s)})\|_{L_2(\ell^2, \ell^2)}^2 ds)\}^{\frac{1}{2}}
\\
\le &\, \frac{1}{2} \mathbb{E}(\sup_{\tau-t \le \varrho\le \tau}e^{2\varepsilon \varrho}\|u(\varrho)\|^4)
+ 32 c_1^2 \int_{\tau-t}^\tau e^{2\varepsilon s}\|u(s)\|^2 \|\tilde{\sigma}(s, u(s), u(s-r), \mathcal{L}_{u(s)})\|_{L_2(\ell^2, \ell^2)}^2 ds
\\
\le &\,  \frac{1}{2} \mathbb{E}(\sup_{\tau-t \le \varrho\le \tau}e^{2\varepsilon \varrho}\|u(\varrho)\|^4)
\\
&+ \left[96c_1^2\|\chi\|^2_{L^{\infty}(\mathbb{R},\ell^2)}(5+e^{2 \varepsilon r})
+32c_1^2 \|\kappa\|^2_{L^{\infty}(\mathbb{R},\ell^2)} \right] \int_ {\tau-t}^{\tau} e^{2 \varepsilon s} \mathbb{E}\|u(s)\|^4 ds
\\
&+  48c_1^2\varepsilon^{-1} e^{2 \varepsilon (\tau-t+r)}\|\chi\|^2_{L^{\infty}(\mathbb{R},\ell^2)} \mathbb{E}\|u_{\tau-t}\|_{\mathcal{C}_r}^4
+16 \varepsilon^{-1} c_1^2 \|\kappa\|_{L^{\infty}(\mathbb{R},\ell^2)}^2 e^{2\varepsilon \tau},
\end{split}
\end{equation}
Combining \eqref{segment-est-3}-\eqref{segment-est-7}, we obtain that
 \begin{equation}\label{segment-est-8}
\begin{split}
& \mathbb{E}(\sup_{\tau-t \le \varrho \le \tau} e^{2\varepsilon \varrho} \|u(\varrho)\|^4)
 + 2 (2 \lambda-\varepsilon) \int_ {\tau-t}^{\tau} e^{ 2 \varepsilon s} \mathbb{E} \|u( s )\|^4 ds
\\
\le &\, \frac{1}{2} \mathbb{E}(\sup_{\tau-t \le \varrho\le \tau}e^{2\varepsilon \varrho}\|u(\varrho)\|^4)
\\
&+\left[2+2 \varepsilon^{-1}\|\eta\|_{L^{\infty}(\mathbb{R},\ell^2)} e^{2\varepsilon r}
+6\varepsilon^{-1}e^{2 \varepsilon r}\|\chi\|^2_{L^{\infty}(\mathbb{R},\ell^2)}(3+8 c_1^2) \right] e^{2 \varepsilon (\tau-t)}\mathbb{E}\|u_{\tau-t}\|_{\mathcal{C}_r}^4
\end{split}
\end{equation}
\begin{equation*}
\begin{split}
& +\left[\frac{7}{2}\lambda+4\|\eta\|_{L^{\infty}(\mathbb{R},\ell^2)}(5+e^{2\varepsilon r})
+[12\|\chi\|^2_{L^{\infty}(\mathbb{R},\ell^2)}(5+e^{2 \varepsilon r})
+4\|\kappa\|^2_{L^{\infty}(\mathbb{R},\ell^2)}](3+8c_1^2)\right]
\\
&\quad \times \int_ {\tau-t}^{\tau} e^{2 \varepsilon s} \mathbb{E}\|u(s)\|^4 ds+\frac{3456}{125\lambda^3}  \int_{\tau-t}^{\tau}  e^{2 \varepsilon s} \|g(s)\|^4 ds+\frac{16}{\lambda} \int_ {\tau-t}^{\tau}e^{2 \varepsilon s}\|\eta(s)\|_1^2ds
\\
& + 2 \varepsilon^{-1}\|\kappa\|_{L^{\infty}(\mathbb{R}, \ell^2)}^2 e^{2\varepsilon \tau}(3+8c_1^2).
\end{split}
\end{equation*}
For $t>r$, \eqref{segment-est-8} and \eqref{lamda} implies that
\begin{equation} \label{segment-est-9}
\begin{split}
\quad &\; \mathbb{E}(\sup_{\tau-r \le \varrho \le \tau}\|u(\varrho)\|^4)  \le  e^{-2\varepsilon(\tau-r)}
\mathbb{E}(\sup_{\tau-t \le \varrho \le \tau} e^{-2\varepsilon \varrho}\|u(\varrho)\|^4)
\\
\le & \; 2\left[ 2+2 \varepsilon^{-1}\|\eta\|_{L^{\infty}(\mathbb{R},\ell^2)} e^{2\varepsilon r}
+6\varepsilon^{-1}e^{2 \varepsilon r}\|\chi\|^2_{L^{\infty}(\mathbb{R},\ell^2)}(3+8 c_1^2)\right]  e^{2 \varepsilon (r-t)} \mathbb{E}\|u_{\tau-t}\|_{\mathcal{C}_r}^4
\\
&\; +\frac{6912}{125\lambda^3} e^{2\varepsilon r} \int_{\tau-t}^{\tau}  e^{2 \varepsilon (s-\tau)} \|g(s)\|^4 ds+\frac{32}{\lambda}
e^{2\varepsilon r} \int_ {\tau-t}^{\tau}
e^{2 \varepsilon (s-\tau)}\|\eta(s)\|_1^2ds
\\
&\; + 4 \varepsilon^{-1}\|\kappa\|_{L^{\infty}(\mathbb{R}, \ell^2)}^2 e^{2\varepsilon r}(3+8c_1^2).
\end{split}
\end{equation}
Since $\mathcal{L}_{u_{\tau-t}} \in D(\tau-t)$ and $D\in \mathcal{D}$, we have
\begin{eqnarray*}
\lim_{t \rightarrow \infty} e^{2 \varepsilon (r-t)} \mathbb{E}\|u_{\tau-t}\|_{\mathcal{C}_r}^4\le
 e^{2 \varepsilon (r-\tau)} \lim_{t \rightarrow \infty} e^{2\varepsilon (\tau-t)}\|D(\tau-t)\|_{\mathcal{P}_4{(\mathcal{C}_r)}}^4=0,
\end{eqnarray*}
and hence, there exists $T=T(\tau, r, \varepsilon, D)>r$ such that for all $t \ge T$,
\begin{eqnarray*}
e^{2 \varepsilon (r-t)} \mathbb{E}\|u_{\tau-t}\|_{\mathcal{C}_r}^4 \le 1,
\end{eqnarray*}
which together with \eqref{segment-est-9} and \eqref{g^4} yields that
\begin{equation*}
\begin{split}
\mathbb{E}(\sup_{\tau-r \le \varrho \le \tau}\|u(\varrho)\|^4)
\le & \; 2\left[ 2+2 \varepsilon^{-1}\|\eta\|_{L^{\infty}(\mathbb{R},\ell^2)} e^{2\varepsilon r}
+6\varepsilon^{-1}e^{2 \varepsilon r}\|\chi\|^2_{L^{\infty}(\mathbb{R},\ell^2)}(3+8c_1^2)\right]
\\
&\; + 4 \varepsilon^{-1}\|\kappa\|_{L^{\infty}(\mathbb{R}, \ell^2)}^2 e^{2\varepsilon r}(3+8c_1^2)
\\
&\; +\frac{6912}{125\lambda^3} e^{2\varepsilon r} \int_{\tau-t}^{\tau}  e^{2 \varepsilon (s-\tau)} \|g(s)\|^4 ds
+\frac{32}{\lambda} e^{2\varepsilon r} \int_ {\tau-t}^{\tau}
e^{2 \varepsilon (s-\tau)}\|\eta(s)\|_1^2ds.
\end{split}
\end{equation*}
Thus, the proof is complete.
\end{proof}

\section{Existence of Pullback Measure Attractors}
In this section, we will establish the existence of pullback measure attractors.

\vskip0.1in
For any given $\tau \in \mathbb{R}$, let $\xi \in L^2(\Omega, \mathcal{C}_r)$ be $\mathcal{F}_\tau$-measurable.
Under Hypotheses {\bf(H1)}-{\bf (H3)}, Proposition \ref{Pro1} ensures that system \eqref{ob-3}-\eqref{ob-4}
has a unique strong solution $u$ in $L^2(\Omega, C([\tau-r,T],\ell^2))$ for any $T>\tau-r$, with the property that
$\mathcal{L}_{u_{t}} \in \mathcal{P}_2(\mathcal{C}_r)$. By Proposition \ref{Pro1}, we have that  for all $t\ge r \ge \tau, \tau \in \mathbb{R}$,
\begin{eqnarray*}
u(t,\tau,\xi)=u(t,r,u(r, \tau,\xi)).
\end{eqnarray*}
For any given $\mu \in \mathcal{P}_2(\mathcal{C}_r)$, we denote by  $u(t, \tau, \mu)$ the solution of system \eqref{ob-3}-\eqref{ob-4} with initial distribution $\mathcal{L}_{u_\tau}=\mu$. There is a bit of abuse of notation here, but it should not cause any confusion.
Due to the weak uniqueness of solutions, for any $t \ge \tau$, we can define a nonlinear operator $P_{\tau,t}^*$
from  $\mathcal{P}_2(\mathcal{C}_r)$ to $\mathcal{P}_2(\mathcal{C}_r)$ by
\begin{eqnarray}\label{df-P}
P_{\tau,t}^*\mu=\mathcal{L}_{u_{t}(\cdot, \tau,\mathcal{\mu})},\quad \forall t \ge  0,
\end{eqnarray}
where the segment $u_{t}(\cdot, \tau,\mathcal{\mu})$ of the solution $u(t, \tau, \mu)$ is given by
$$u_t(s, \tau, \mu)=u(t+s, \tau , \mu), \quad \forall s \in [-r, 0].$$
Clearly, $P_{\tau,t}^*$ is a semigroup with the property
\begin{eqnarray}\label{md1}
P_{\tau,t}^*=P_{r,t}^*P_{\tau, r}^*, \quad \forall t\ge r \ge \tau, \tau\in\mathbb{R}.
\end{eqnarray}

For every $t\in \mathbb{R}^+$
and $\tau \in \mathbb{R}$, we define
$\Phi (t, \tau): \mathcal{P}_2(\mathcal{C}_r)
\to \mathcal{P}_2(\mathcal{C}_r)$ by
\begin{align*}
 \Phi (t,\tau) \mu=
 P^*_{\tau,  \tau +t} \mu,\quad \forall \mu\in \mathcal{P}_2(\mathcal{C}_r).
\end{align*}
By \eqref{md1}, for all $t,s\in \mathbb{R}^+$, $\tau \in \mathbb{R}$
and $\mu \in  \mathcal{P}_2(\mathcal{C}_r)$, we have
$$
\Phi (t+s, \tau)\mu =
\Phi (t, s+\tau) ( \Phi(s, \tau) \mu).
$$
 {Clearly, $\Phi$ is a non-autonomous
dynamical system on
$(\mathcal{P}_2(\mathcal{C}_r), d_{\mathcal{P} (\mathcal{C}_r)})$.
\vskip0.1in

In the following, we prove that $\Phi$ is continuous on the bounded subset $D$ of $\mathcal{P}_2(\mathcal{C}_r)$. To see this, we
first prove the following lemma.
\begin{lm}\label{P-dis}
Assume that  {\bf (H1)}-{\bf (H3)} hold. Then, for any $\mu, \upsilon \in \mathcal{P}_2(\mathcal{C}_r)$ and $t \ge \tau$,
\begin{eqnarray*}
\mathbb{W}_2(P_{\tau,t}^*\mu,P_{\tau,t}^*\upsilon) \le  \tilde{c}_1 \mathbb{W}_2(\mu, \nu) e^{\tilde{c}_2(t-\tau)},
\end{eqnarray*}
where
$
\tilde{c}_1\; :=\; 3+8 r \|\psi\|^2_{L^{\infty}([\tau, t],\ell^2)}
+6 r \|\chi\|^2_{L^{\infty}(\mathbb{R},\ell^2)}(1+2c_1^2) $\\
and
$\;\;\; \tilde{c}_2\; :=\; 4\|\Theta\|_{L^{\infty}([\tau, t],\ell^2)}+18\|\psi\|_{L^{\infty}([\tau, t],\ell^2)}
+18\|\chi\|^2_{L^{\infty}(\mathbb{R},\ell^2)} (1+2c_1^2).
$
\end{lm}

\begin{proof}
Let $u(t):=u(t, \tau, \mu)$ and $v(t):=u(t, \tau, \upsilon)$ be two solutions to \eqref{ob-3}-\eqref{ob-4}, where $\mathcal{L}_{u_\tau}=\mu$ and $\mathcal{L}_{v_\tau}=\upsilon$, respectively, and
$$\mathbb{W}_2(\mu, \upsilon)^2= \mathbb{E}\|u_\tau-v_\tau\|_{\mathcal{C}_r}^2.$$
Following the It\^o formula, we obtain that for all $t \ge \tau$,
\begin{eqnarray} \label{W2-lemma-2}\small
\begin{split}
&\; \mathbb{E} (\sup_{\tau \le \varrho \le t}\|u(\varrho)-v(\varrho)\|^2)
\le \mathbb{E}\|u_\tau(0)-v_\tau(0)\|^2
\\
&\; +2\mathbb{E} \sup_{\tau \le \varrho \le t}\int_\tau^{\rho}(u(s)-v(s), [f(s, u(s), u(s-r), \mathcal{L}_{u(s)})-f(s, v(s), v(s-r), \mathcal{L}_{v(s)})]) ds
\\
&\; + \mathbb{E}  \sup_{\tau \le \varrho \le t} \int_\tau^{\varrho} \|\tilde{\sigma}(s, u(s), u(s-r), \mathcal{L}_{u(s)})
-\tilde{\sigma}(s, v(s), v(s-r), \mathcal{L}_{v(s)})\|_{L_2(\ell^2, \ell^2)}^2 ds
\\
&\; +2 \mathbb{E}  \sup_{\tau \le \varrho \le t} \bigl|\int_\tau^{\varrho} (u(s)-v(s), [\tilde{\sigma}(s, u(s), u(s-r), \mathcal{L}_{u(s)})-\tilde{\sigma}(s, v(s), v(s-r), \mathcal{L}_{v(s)})]dW(s)) \bigr|.
\end{split}
\end{eqnarray}

We estimate each term in the right-hand side of \eqref{W2-lemma-2} separately. For the second term, by using \eqref{fi-Lip}
and \eqref{dis-property}, we notice that
\begin{equation}\label{W2-lemma-3}
\begin{split}
&\, 2 \mathbb{E} \sup_{\tau \le \varrho \le t} \int_\tau^{\varrho}(u(s)-v(s), [f(s, u(s), u(s-r), \mathcal{L}_{u(s)})
-f(s, v(s), v(s-r), \mathcal{L}_{v(s)})]) ds
\\
\le &\, 2 \mathbb{E} \sup_{\tau \le \varrho \le t} \int_\tau^{\varrho}(u(s)-v(s), [f(s, u(s), u(s-r), \mathcal{L}_{u(s)})
-f(s, v(s), u(s-r), \mathcal{L}_{u(s)})]) ds
\\
&\, +2 \mathbb{E} \sup_{\tau \le \varrho \le t} \int_\tau^{\varrho}(u(s)-v(s), [f(s, v(s), u(s-r), \mathcal{L}_{u(s)})
-f(s, v(s), v(s-r), \mathcal{L}_{v(s)})]) ds
\\
\le &\,  2 \mathbb{E} \int_\tau^{t} \|\Theta(s)\|\|u(s)-v(s)\|^2  ds
\\
&\, + 2 \mathbb{E} \int_\tau^{t} \sum_{i \in \mathbb{Z}} \psi_i(s) \left[|u_i(s-r)-v_i(s-r)|+(\mathbb{E} |u_i(s)-v_i(s)|^2)^{\frac{1}{2}} \right] |u_i(s)-v_i(s)|ds
\end{split}
\end{equation}
\begin{equation*}
\begin{split}
\le &\,  2 \mathbb{E} \int_\tau^t \|\Theta(s)\| \|u(s)-v(s)\|^2 ds+\mathbb{E} \int_\tau^t \|\psi(s)\| \|u(s)-v(s)\|^2 ds
\\
&\,+4 \mathbb{E} \int_\tau^t \|\psi(s)\|(\|u(s-r)-v(s-r)\|^2+ \mathbb{E}\|u(s)-v(s)\|^2)ds
\\
\le &\, (2\|\Theta\|_{L^{\infty}([\tau, t],\ell^2)}+9\|\psi\|_{L^{\infty}([\tau, t],\ell^2)})\int_\tau^t \mathbb{E}\|u(s)-v(s)\|^2 ds
\\
&\,+ 4 r \|\psi\|^2_{L^{\infty}([\tau, t],\ell^2)}\mathbb{ E }\|u_\tau-v_\tau\|_{\mathcal{C}_r}^2.
\end{split}
\end{equation*}
For the third term, by using \eqref{segama-Lip} and \eqref{dis-property} we have that
\begin{equation} \label{W2-lemma-4}
\begin{split}
&\, \mathbb{E} \sup_{\tau \le \varrho \le t} \int_\tau^{\varrho}
\|\tilde{\sigma}(s, u(s), u(s-r), \mathcal{L}_{u(s)})-\tilde{\sigma}(s, v(s), v(s-r), \mathcal{L}_{v(s)})\|_{L_2(\ell^2, \ell^2)}^2 ds
\\
\le &\, 3\|\chi\|^2_{L^{\infty}(\mathbb{R},\ell^2)} \int_\tau^t
\mathbb{E}\left( \|u(s)-v(s)\|^2+\|u(s-r)-v(s-r)\|^2+ \mathbb{E}\|u(s)-v(s)\|^2 \right) ds
\\
\le &\, 9\|\chi\|^2_{L^{\infty}(\mathbb{R},\ell^2)}
\int_\tau^t \mathbb{E} \|u(s)-v(s)\|^2 ds+3 r \|\chi\|^2_{L^{\infty}(\mathbb{R},\ell^2)}
 \mathbb{E}\|u_\tau-v_\tau\|_{\mathcal{C}_r}^2.
\end{split}
\end{equation}
For the last term, by using the BDG inequality, \eqref{segama-Lip} and \eqref{dis-property} we obtain that
\begin{equation} \label{W2-lemma-5}\small
\begin{split}
&\,  2 \mathbb{E}\sup_{\tau \le \varrho \le t} \bigl |  \int_\tau^{\varrho} (u(s)-v(s), [\tilde{\sigma}(s, u(s), u(s-r), \mathcal{L}_{u(s)})
-\tilde{\sigma}(s, v(s), v(s-r), \mathcal{L}_{v(s)})]dW(s)) \bigr|
\\
\le &\, 2c_1\mathbb{E}\{\sup_{\tau \le \varrho \le t}\|u(\varrho)-v(\varrho)\| \cdot
(\int_\tau^t \|\tilde{\sigma}(s, u(s), u(s-r), \mathcal{L}_{u(s)})
-\tilde{\sigma}(s, v(s), v(s-r), \mathcal{L}_{v(s)})\|_{L_2(\ell^2, \ell^2)}^2 ds)^{\frac{1}{2}}\}
\\
\le &\, \frac{1}{2} \mathbb{E}(\sup_{\tau \le \varrho\le t}\|u(\varrho)-v(\varrho)\|^2)
\\
&\,+ 2c_1^2 \int_\tau^t \mathbb{E}\|\tilde{\sigma}(s, u(s), u(s-r), \mathcal{L}_{u(s)})
-\tilde{\sigma}(s, v(s), v(s-r), \mathcal{L}_{v(s)})\|_{L_2(\ell^2, \ell^2)}^2 ds
\\
\le &\, \frac{1}{2} \mathbb{E}(\sup_{\tau \le \varrho\le t}\|u(\varrho)-v(\varrho)\|^2)
\\
&\,+2c_1^2 \left( 9\|\chi\|^2_{L^{\infty}(\mathbb{R},\ell^2)}\int_\tau^t \mathbb{E} \|u(s)-v(s)\|^2 ds
+3 r\|\chi\|^2_{L^{\infty}(\mathbb{R},\ell^2)} \|u_\tau-v_\tau\|_{\mathcal{C}_r}^2\right).
\end{split}
\end{equation}
Combining \eqref{W2-lemma-2}-\eqref{W2-lemma-5}, we have that for all $t \ge \tau$,
\begin{equation}\label{W2-lemma-6}
\begin{split}
 &\,  \mathbb{E} (\sup_{\tau \le \varrho \le t}\|u(\varrho)-v(\varrho)\|^2)
\\
 \le & \frac{1}{2} \mathbb{E}(\sup_{\tau \le \varrho\le t}\|u(\varrho)-v(\varrho)\|^2)
\\
&\, + \left[2\|\Theta\|_{L^{\infty}([\tau, t],\ell^2)}+9\|\psi\|_{L^{\infty}([\tau, t],\ell^2)}
+9\|\chi\|^2_{L^{\infty}(\mathbb{R},\ell^2)} (1+2c_1^2)\right]
 \int_\tau^t \mathbb{E} \|u(s)-v(s)\|^2 ds
\\
&\, +\left[ 1+4 r \|\psi\|^2_{L^{\infty}([\tau, t],\ell^2)}+3 r \|\chi\|^2_{L^{\infty}(\mathbb{R},\ell^2)}(1+2c_1^2)\right] \mathbb{E}\|u_\tau-v_\tau\|_{\mathcal{C}_r}^2.
\end{split}
\end{equation}
Then, we obtain that for all $t \ge \tau$,
\begin{equation*}
\begin{split}
\mathbb{E} (\sup_{\tau \le \varrho \le t}\|u(\varrho)-v(\varrho)\|^2) \le &\, \tilde{c}_2
\int_\tau^t \mathbb{E} \|u(s)-v(s)\|^2 ds
\\
&+2\left[1+4 r \|\psi\|^2_{L^{\infty}([\tau, t],\ell^2)}+3 r\|\chi\|^2_{L^{\infty}(\mathbb{R},\ell^2)}(1+2c_1^2)\right] \mathbb{E}\|u_\tau-v_\tau\|_{\mathcal{C}_r}^2.
\end{split}
\end{equation*}
Observing that
\begin{eqnarray*}
\mathbb{E}(\sup_{\tau-r \le \varrho \le \tau}\|u(\varrho)-v(\varrho)\|^2)=\mathbb{E}\|u_\tau-v_\tau\|_{\mathcal{C}_r}^2,
\end{eqnarray*}
we get that
\begin{equation*}
\begin{split}
 &\,  \mathbb{E} (\sup_{\tau-r \le \varrho \le t}\|u(\varrho)-v(\varrho)\|^2) \le \tilde{c}_2
\int_\tau^t \mathbb{E} \|u(s)-v(s)\|^2 ds+\tilde{c}_1 \mathbb{E}\|u_\tau-v_\tau\|_{\mathcal{C}_r}^2.
\end{split}
\end{equation*}
Finally, Lemma \ref{P-dis} is proven using the Gronwall inequality.
\end{proof}

As a straightforward corollary, we immediately obtain that
\begin{cor}\label{cor}
For every bounded subset $S$
of $\mathcal {P}_4(\mathcal{C}_r)$, if {\bf (H1)}-{\bf (H3)} hold, then $\Phi(t,\tau)$
is continuous from  $(S, d_{\mathcal{P}(\mathcal{C}_r)})$
to $(\mathcal {P}_4(\mathcal{C}_r), d_{\mathcal{P}(\mathcal{C}_r)})$.
\end{cor}
\begin{proof}
Let $\{\mu_n\}_{n=1}^{\infty}$ be a sequence that converges to $\mu$ in $(S, d_{\mathcal{P}(\mathcal{C}_r)})$. By Lemma \ref{P-dis}, for any $n \in \mathbb{N}^+$,
$t\in \mathbb{R}^+$ and $\tau \in \mathbb{R}$,
\begin{eqnarray} \label{conti-1}
\mathbb{W}_2(P_{\tau,\tau+t}^*\mu_n,P_{\tau,t}^*\mu) \le  \tilde{c}_1 \mathbb{W}_2(\mu_n, \mu) e^{\tilde{c}_2t}.
\end{eqnarray}
On the other hand, since $\mu_n$ converges weakly to $\mu$, by Skorokhod's representation theorem, there exists a probability space $(\widetilde{\Omega}, \widetilde{\mathcal{F}}, \widetilde{\mathbb{P}})$ and a sequence of random variables $u_0$ and $u_{0,n}$ defined on $(\widetilde{\Omega}, \widetilde{\mathcal{F}}, \widetilde{\mathbb{P}})$ such that their distributions are $\mu$ and $\mu_n$, respectively, and
 \begin{eqnarray} \label{conti-2}
 \lim_{n\rightarrow \infty} \|u_{0,n}-u_0\|_{\mathcal{C}_r}=0, \,\tilde{ \mathbb{P}}-a.s.
 \end{eqnarray}

Since $\mu_n \in S$, there exists a constant $R > 0$ such that $\mathbb{E} \left( |u_{0,n}|^4 \right) \le R$. This implies that the sequence $\{u_{0,n} \}_{n=1}^\infty$ is uniformly integrable in $L^2(\tilde{\Omega}, \mathcal{C}_r)$ as $n \to \infty$.
Using the continuity result in \eqref{conti-2} (which has already been established), and applying the Vitali convergence theorem, we conclude that $u_{0,n} \to u_0$ in $L^2(\tilde{\Omega}, \mathcal{C}_r)$. Therefore, as $n \to \infty$,
\begin{eqnarray*}
\mathbb{W}_2(\mu_n, \mu) \le \{\mathbb{E}^{\tilde{\mathbb{P}}}\|u_{0,n}-u_0\|_{\mathcal{C}_r}^2\}^{\frac{1}{2}}\rightarrow 0,
\end{eqnarray*}
where $\mathbb{E}^{\tilde{\mathbb{P}}}$ is the expectation with respect to $\tilde{\mathbb{P}}$.
Using \eqref{conti-1}, we get that $\mathbb{W}_2(\Phi(\tau,t)\mu_n, \Phi(\tau,t)\mu)\rightarrow 0$ as $n \rightarrow \infty$.
Hence, we obtain $\Phi(\tau,t)\mu_n$ converges to $\Phi(\tau,t)\mu$ weakly. Thus, the proof is complete.
\end{proof}

Based on Lemma \ref{P-dis}, we conclude that the mapping $\Phi$ defines a non-autonomous dynamical system on $(\mathcal{P}_4 (\mathcal{C}r), d_{\mathcal{P} (\mathcal{C}_r)})$, which is continuous on bounded subsets of $\mathcal{P}_4 (\mathcal{C}_r)$. Next, we proceed to prove the existence of $\mathcal{D}$-pullback absorbing sets for $\Phi$.
\begin{lm}\label{absorbset}
Assume that  {\bf  (H1)}-{\bf (H3)}, \eqref{lamda} and \eqref{g^4} hold. Then $\Phi$ has a closed
$\mathcal{D}$-pullback absorbing set $K=\{K(\tau): \tau \in \mathbb{R}\} \in \mathcal{D}$
which is given by, for each $\tau\in\mathbb{R}$,
\begin{align} \label{absorbset-1}
K(\tau )=\left\{\mu \in {\mathcal{P}}_4 (\mathcal{C}_r): \int_{\mathcal{C}_r} \|\xi\|_{\mathcal{C}_r}^4 \mu (d \xi) \le M_4 (\tau)\right\},
\end{align}
where
 $$M_4(\tau)=M_3+ M_3\int_{-\infty }^\tau e^{2\varepsilon(s-\tau)}\left(\|g(s)\|^4+\|\eta(s)\|^2_1 \right)ds,$$
and $\varepsilon>0$ is the same number as in \eqref{lamda-1}, and $ M_3>0 $  is  the same constant as in Lemma \ref{est-3}
which is independent of $\tau$.
\end{lm}

\begin{proof}
First, note that $K(\tau)$ is a closed subset of $\mathcal{P}_4 (\mathcal{C}_r)$. On the other hand,
for every $\tau \in \mathbb{R}$ and $D \in \mathcal{D}$, by using Lemma \ref{est-3}, we have that there exists $T=T(\tau, D)>0$ such that for all $t\ge T$,
\begin{align}\label{absorbset-3}
\Phi(t,\tau-t) D(\tau-t) \subseteq K(\tau).
\end{align}
Finally, by \eqref{g^4} and \eqref{absorbset-1} we have, as $\tau \to -\infty$,
$$e^{2\varepsilon \tau }\| K(\tau)\|_{\mathcal{P}_4\mathcal{(C}_r)}^4 \le e^{2\varepsilon \tau} M_3 + M_3
\int_{-\infty }^\tau e^{2\eta s} \left(\|g(s)\|^4+\|\eta(s)\|_1^2\right) ds \to 0.$$
This implies that $K = {K(\tau) : \tau \in \mathbb{R}} \in \mathcal{D}$. Therefore, from \eqref{absorbset-3}, we conclude that $K$, as defined in \eqref{absorbset-1}, is a closed $\mathcal{D}$-pullback absorbing set for $\Phi$.
\end{proof}

Next, we establish the $\mathcal{D}$-pullback asymptotic compactness of $\Phi$.

\begin{lm}
\label{tight}
Assume that {\bf  (H1)}-{\bf (H3)}, \eqref{lamda} and \eqref{g^2}-\eqref{g^4} hold. Then $\Phi$
is $\mathcal{D}$-pullback asymptotically compact in $(\mathcal{P}_4(\mathcal{C}_r), d_{\mathcal{P}(\mathcal{C}_r)})$.
\end{lm}

\begin{proof}
Given $\tau \in \mathbb{R}$, $D  \in  \mathcal{D}$,  $t_n \to  \infty$ and $\mu_n \in D (\tau-t_n)$,
we will prove that the sequence  $\{\Phi (t_n, \tau -t_n)\mu_n \}_{n=1}^\infty$ has a convergent
subsequence in $(\mathcal{P}_4 (\mathcal{C}_r), d_{\mathcal{P} (\mathcal{C}_r)})$.

Without loss of generality, assume that $v_n  \in L^4  (\Omega, \mathcal{F}_{ \tau - t_n}, \mathcal{C}_r)$ with $ \mathcal{L}_{ v_n } =\mu_n$.
Consider the solution $u(\tau, \tau-t_n, v_n)$ of \eqref{ob-3} with initial data
$v_n$ at initial time $\tau -t_n$. We first show the distributions of $\{u_\tau(\cdot, \tau-t_n, v_n)\}_{n=1}^\infty$
are  tight in $\mathcal{C}_r$.

Let $\theta: \mathbb{R}^n \to [0,1]$ be the smooth cut-off function given by \eqref{tail-est-1},
and $\theta_m=\left(\theta\left ({\frac im}\right )\right)_{i\in \mathbb{Z}}$
for every $m\in \mathbb{N}^+$. Then the solution segment $u_\tau$ can be decomposed as:
$$u_{\tau}(\cdot, \tau-t_n, v_n)=\theta_m u_{\tau}(\cdot, \tau-t_n, v_n)+(1-\theta_m)u_\tau(\cdot, \tau-t_n, v_n).$$
By Lemma \ref{est-1+}, there exists $N_1=N_1(\tau, D)$ such that for all $m \in \mathbb{N}^+$ and  $n \geq N_1$, we have
\begin{equation}\label{1-theta-1}
\begin{split}
\mathbb{E}\|(1-\theta_m)u_\tau(\cdot, \tau-t_n, v_n)\|_{\mathcal{C}_r}^2
= &\; \mathbb{E }\sup_{\tau-r \le \varrho \le \tau} \sum_{i \in \mathbb{Z}} (1-\theta(\frac i m))^2 u(\varrho, \tau-t_n, v_n)^2
\\
\le &\; \mathbb{E}\|u_\tau(\cdot, \tau-t_n, v_n)\|_{\mathcal{C}_r}^2 \le C_1,
\end{split}
\end{equation}
where $C_1=C_1(\tau)$ is a constant depending only on $\tau$, but not on $n$ or $D$.
By \eqref{1-theta-1} and the Chebyshev inequality, we obtain that
  for all $m\in \mathbb{N}^+$, $n\ge N_1$ and $R>0$,
$$\mathbb{P}\left(\{\|(1-\theta_m )u_\tau(\cdot, \tau - t_n, v_n)\|_{\mathcal{C}_r}^2>R\}\right) \le {\frac {C_1}{R^2}}.$$
Therefore, for every $\delta>0$, there exists $R(\delta, \tau)>0$ such that for all $m\in \mathbb{N}^+$ and $n\ge N_1$,
\begin{align}\label{1-theta-2}
\mathbb{P}\left(\{\|(1-\theta_m)u_{\tau}(\cdot, \tau-t_n,v_n)\|_{\mathcal{C}_r}>R(\delta, \tau)\}\right)<\delta.
\end{align}
For any given $m \in \mathbb{N}^+$, we consider the closed subset of $\mathcal{C}_r$:
$$Z_\delta=\{w\in \mathcal{C}_r:\; \|w\|_{\mathcal{C}_r}\le R(\delta, \tau), \  \ w_i(t) =0 \ \text{for}\ |i|>2m\}.$$
Next, we claim that $Z_\delta$ is a compact subset of $\mathcal{C}_r$. Indeed, we may consider the auxiliary set
\begin{eqnarray*}
\tilde{Z}_\delta =\{w=(w_i)_{|i| \le 2m} \in C([-r,0], \mathbb{R}^{2m+1}):\;  |w|_{C_r}:=\sup_{t\in [-r,0]}\|w(t)\|
\le R(\delta, \tau)\},
\end{eqnarray*}
where $\|w(t)\|=(\sum_{|i|\le 2m} |w_i(t)|^2)^{\frac{1}{2}}$. There is some redundancy in the notation here, but it should not cause confusion.
For any given $w\in \tilde{Z}_\delta$,
we note that $w_i$ is uniformly continuous on the interval $[-r,0]$ for all $|i| \le 2m$. Based on this fact, for any $\epsilon>0$,
there exists $\tilde{\delta}>0$ such that for all $|i| \le 2m$,
\begin{eqnarray*}
|w_i(s_1)-w_i(s_2)|<\frac{\epsilon}{2^{\frac{2m+1}{2}}}
\end{eqnarray*}
whenever $s_1,s_2 \in [-r,0]$ and $|s_1-s_2|<\tilde{\delta}$. This yields that
\begin{eqnarray*}
\|w(s_1)-w(s_2)\|^2=\sum_{|i| \le 2m} |w_i(s_1)-w_i(s_2)|^2 <\epsilon^2
\end{eqnarray*}
whenever $s_1,s_2 \in [-r,0]$ and $|s_1-s_2|<\tilde{\delta}$. Thus, $\tilde{Z}_\delta$ is equicontinuous.
On the other hand, $\tilde{Z}_\delta$ is a uniformly bounded set by the definition. Therefore, by
the Arzela-Ascoli theorem,   $\tilde{Z}_\delta$ is a compact set.Our assertion follows easily from the relationship between $Z_\delta$ and $\tilde{Z}_\delta$.

\vskip0.1in

By \eqref{1-theta-2} we have that for all $m\in \mathbb{N}^+$ and $n\ge N_1$,
\begin{align}\label{1-theta-3}
 \mathbb{P} \left(\left\{(1-\theta_m)u_\tau(\cdot,\tau-t_n,v_n)\in Z_\delta\right\} \right)>1-\delta.
\end{align}
Since $\delta>0$ is arbitrary, by \eqref{1-theta-3} we obtain that for every $m\in \mathbb{N}^+$,
\begin{align}\label{1-theta-4}
\{\mathcal{L}_{(1-\theta_m)u_\tau(\cdot,\tau-t_n,v_n)}\}_{n=1}^\infty \ \text{is tight in} \ \mathcal{C}_r,
\end{align}
where $\mathcal{L}_{(1-\theta_m)u_\tau(\cdot, \tau-t_n, v_n)}$ is the distribution of $(1-\theta_m)u_\tau(\cdot, \tau-t_n, v_n)$ in $\mathcal{C}_r$.

Now, we prove that $\{\mathcal{L}_{u_\tau(\cdot, \tau - t_n, v_n)}\}_{n=1}^\infty$ is tight in $\mathcal{C}_r$ by using the uniform tail-ends estimates.
Indeed, by Lemma \ref{est-2} we infer that for every $\epsilon>0$, there exist $N_2=N_2(\epsilon, \tau, D)\ge N_1$ and
$m_0=m_0(\epsilon, \tau, D)\in \mathbb{N}^+$ such that for all $n\ge N_2$,
\begin{eqnarray}\label{1-theta-5}
\mathbb{E} (\sup_{\tau-r \le \varrho \le \tau} \sum_{|i| \ge 2 m_0} |u_i(\varrho, \tau - t_n, v_n)|^2) <\frac{1}{9}\epsilon^2.
\end{eqnarray}
By \eqref{1-theta-5} we get that for all $n\ge N_2$,
\begin{align}\label{1-theta-6}
\mathbb{E} \left(\|\theta_{m_0}u_\tau(\cdot, \tau -t_n, v_n)\|^2\right)<{\frac 19}\epsilon^2.
\end{align}
By \eqref{1-theta-4}, we obtain that
$\{\mathcal{L}_{(1-\theta_{m_0})u_\tau(\cdot, \tau-t_n, v_n)}\}_{n=N_2}^\infty$ is tight in $\mathcal{C}_r$. Hence,
there exist
$n_1,\cdots , n_l \ge N_2 $ such that
\begin{align}\label{1-theta-7}
\{\mathcal{L}_{(1-\theta_{m_0})u_\tau(\cdot, \tau-t_n, v_n)}\}_{n=N_2}^\infty
\subseteq
\bigcup_{j=1}^{l} B\left(\mathcal{L}_{(1-\theta_{m_0})u_\tau(\cdot,\tau-t_{n_j},v_{n_j})}, \ {\frac 13} \epsilon \right),
\end{align}
where
$B\left(\mathcal{L}_{(1-\theta_{m_0})u_\tau(\cdot, \tau-t_{n_j}, v_{n_j})}, \ {\frac 13} \epsilon \right)$ is the
${\frac 13} \epsilon$-neighborhood of $ \mathcal{L}_{(1-\theta_{m_0})u_\tau(\cdot, \tau-t_{n_j}, v_{n_j})}$ in the space
 $(\mathcal{P} (\mathcal{C}_r), d_{\mathcal{P} (\mathcal{C}_r)})$.
  We note that
  \begin{align}\label{1-theta-8}
 \{\mathcal{L}_{u_\tau(\cdot, \tau-t_n, v_n)}\}_{n=N_2}^\infty \subseteq \bigcup_{j=1}^{l}
 B\left (\mathcal{L}_{u_\tau(\cdot, \tau-t_{n_j}, v_{n_j})},  \  \epsilon \right).
 \end{align}
Given $n\ge N_2$, by \eqref{1-theta-7} we know that there exists $j_0\in \{1,\cdots, l\}$ such that
\begin{align}\label{1-theta-9}
\mathcal{L}_{
(1-\theta_{m_0})  u_\tau(\cdot, \tau-t_n, v_n)} \in B\left (\mathcal{L}_{(1-\theta_{m_0})u_\tau(\cdot,\tau-t_{n_{j_0}}, v_{n_{j_0}})},
\ {\frac 13} \epsilon \right ).
 \end{align}
By \eqref{1-theta-6} and \eqref{1-theta-9} we have
\begin{equation*}
\begin{split}
 & d_{\mathcal{P}(\mathcal{C}_r) }\left(\mathcal{L}_{u_\tau(\cdot, \tau-t_n, v_n)}, \ \mathcal{L}_{u_\tau(\cdot, \tau-t_{n_{j_0}}, v_{n_{j_0}})} \right )
 \\
 =& \, \sup_{\psi \in L_b(\mathcal{C}_r), \|\psi\|_{L_b({\mathcal{C}_r})} \le 1}
 \left| \int_{\mathcal{C}_r} \psi d \mathcal{L}_{u_\tau(\cdot, \tau-t_n, v_n)}
 -\int_{\mathcal{C}_r} \psi  d\mathcal{L}_{u_\tau(\cdot, \tau-t_{n_{j_0}}, v_{n_{j_0}})} \right|
 \\
 =& \, \sup_{\psi \in L_b(\mathcal{C}_r), \|\psi\|_{L_b(\mathcal{C}_r)} \le 1} \left|\mathbb{E} \left(\psi(u_\tau(\cdot, \tau-t_n, v_n))\right)
 -\mathbb{E}\left(\psi(u_\tau(\cdot, \tau-t_{n_{j_0}}, v_{n_{j_0}}))\right) \right|
\\
\le &  \sup_{\psi \in L_b(\mathcal{C}_r), \|\psi\|_{L_b(\mathcal{C}_r)} \le 1} \left|\mathbb{ E}\left(\psi(u_\tau(\cdot, \tau-t_n,  v_n))\right)
 -\mathbb{E} \left(\psi( (1-\theta_{m_0} )u_\tau(\cdot, \tau-t_{n}, v_{n})) \right) \right|
 \\
 &+ \sup_{\psi \in L_b(\mathcal{C}_r), \|\psi\|_{L_b(\mathcal{C}_r)} \le 1}\left|\mathbb{E} \left(\psi( (1-\theta_ {m_0})u_\tau(\cdot, \tau-t_n, v_n))\right)
 -\mathbb{E} \left(\psi((1-\theta_ {m_0})u_\tau(\cdot, \tau-t_{n_{j_0}}, v_{n_{j_0}}))\right) \right|
 \\
 &+ \sup_{\psi \in L_b(\mathcal{C}_r), \|\psi\|_{L_b(\mathcal{C}_r)}\le 1}
 \left|\mathbb{ E} \left(\psi( (1-\theta_ {m_0})u_\tau(\cdot, \tau-t_{n_{j_0}}, v_{n_{j_0}}))\right)
 -\mathbb{E}\left(\psi(u_\tau(\cdot, \tau-t_{n_{j_0}}, v_{n_{j_0}}))\right)\right|
 \\
 \le &  \mathbb{E} \left(\|\theta_ {m_0} u_\tau(\cdot, \tau-t_{n}, v_{n})\|_{\mathcal{C}_r}\right)
  +\mathbb{E}\left(\|\theta_ {m_0}u_\tau(\cdot, \tau-t_{n_{j_0}}, v_{n_{j_0}})\|_{\mathcal{C}_r} \right)
 \\
 & +d_{\mathcal{P} (\mathcal{C}_r)}\left (\mathcal{L}_{(1-\theta_{m_0} )u_\tau(\cdot, \tau - t_n,  v_n)}, \
 \mathcal{L}_{(1-\theta_{m_0})u_\tau(\cdot, \tau-t_{n_{j_0}}, v_{n_{j_0}})}\right)
 \\
 <&{\frac 13} \epsilon +{\frac 13} \epsilon + {\frac 13} \epsilon =\epsilon
 \end{split}
 \end{equation*}
which yields \eqref{1-theta-9}. Since $\epsilon>0$ is arbitrary, by \eqref{1-theta-8} we see that
the sequence $\{\mathcal{L}_{u_\tau(\cdot, \tau-t_n, v_n)}\}_{n=1}^\infty$ is tight in $ {\mathcal{P}(\mathcal{C}_r)}$,
which implies that there exists $\upsilon \in \mathcal{P}(\mathcal{C}_r)$ such that,
up to a subsequence,
 \begin{align}\label{1-theta-10}
 \mathcal{L}_{u_\tau(\cdot, \tau-t_n, v_n)}\to \upsilon\ \text{weakly}.
 \end{align}

 It remains to show that $\upsilon \in \mathcal{P}_4(\mathcal{C}_r)$.
 Let $K=\{K(\tau): \tau \in \mathbb{R}\}$  be the closed $\mathcal{D}$-pullback
 absorbing set of $\Phi$ given by \eqref{absorbset-1}. Then, there exists
 $N_3=N_3 (\tau, D)\in \mathbb{N}^+$ such that for all $n\ge N_3$,
 \begin{align}\label{1-theta-11}
 \mathcal{L}_{u_\tau(\cdot, \tau -t_n, v_n)} \in K(\tau).
\end{align}
 Since $K(\tau)$ is closed with respect to the weak topology of $(\mathcal{P}(\mathcal{C}_r), d_{\mathcal{P}(\mathcal{C}_r)})$,
 by \eqref{1-theta-10}-\eqref{1-theta-11} we obtain $\upsilon \in K(\tau)$. Thus $\upsilon \in \mathcal{P}_4 (H)$.
Thus, the proof of this lemma is complete.
\end{proof}

Summarizing the above, we have established the existence and uniqueness of $\mathcal{D}$-pullback measure attractors of $\Phi$.
\begin{thm}\label{main_e}
Assume that  {\bf (H1)}-{\bf (H3)}, \eqref{lamda} and \eqref{g^2}-\eqref{g^4} hold. Then, $\Phi$ has a unique $\mathcal{D}$-pullback
measure attractor $\mathcal{A}= \{\mathcal{A} (\tau): \tau \in \mathbb{R} \}  \in \mathcal{D}$.
\end{thm}

\begin{proof}
The existence and uniqueness of the measure attractor $\mathcal{A}$ follow from Proposition \ref{exatt} based on Lemmas \ref{P-dis},
\ref{absorbset} and \ref{tight}.
\end{proof}

\section{Singleton attractors and exponentially mixing of invariant measures}

In this section, under further conditions, we prove that the measure attractor of the system \eqref{ob-3}-\eqref{ob-4} is a singleton. Building on this result, we establish the existence, uniqueness, and exponential mixing of invariant measures for the autonomous case.

In the following,  we assume that $\lambda$ satisfies
\eqref{lamda} and additional  condition:
\begin{equation} \label{lamda-add}
\lambda>2\|\Theta\|_{L^{\infty}(\mathbb{R},\ell^2)}+5\|\psi\|_{L^{\infty}(\mathbb{R},\ell^2)}+9\|\chi\|^2_{L^{\infty}(\mathbb{R},\ell^2)}(1+4c_1^2).
\end{equation}
Using \eqref{lamda} and \eqref{lamda-add} we find that there exists a sufficiently small
$\eta\in (0,1)$ such that both \eqref{lamda-1} and the following inequality
are satisfied:
\begin{eqnarray} \label{epsilca}
2 \lambda -\varepsilon \!\!\!&>&\!\!\! 4\|\Theta\|_{L^{\infty}(\mathbb{R},\ell^2)}+10\|\psi\|_{L^{\infty}(\mathbb{R},\ell^2)}+18\|\chi\|^2_{L^{\infty}(\mathbb{R},\ell^2)}(1+4c_1^2).
\end{eqnarray}

Under the assumption \eqref{lamda-add}, we have the following lemma.
\begin{lm}\label{im1}
Suppose  that {\bf (H1)}-{\bf (H3)} and \eqref{lamda-add} hold, and $\psi(\cdot), \Theta(\cdot) \in L^{\infty}(\mathbb{R}, \ell^2)$.  Then, for every
$\tau \in \mathbb{R}$, $t\ge r$ and
$\mu, \tilde{\mu}\in \mathcal{P}_2(\mathcal{C}_r)$, we have
\begin{eqnarray*}
\mathbb{E}(\|u_\tau(\cdot, \tau-t, \mu)- u_\tau(\cdot, \tau-t, \tilde{\mu})\|_{\mathcal{C}_r}^2) \le  \tilde{c}_3 \mathbb{W}_2(\mu, \tilde{\mu})^2 e^{-\varepsilon(t-r)},
\end{eqnarray*}
where $\varepsilon>0$ is the same number as in \eqref{epsilca} \\
and
$\tilde{c}_3:=4\left[1+2 \varepsilon^{-1} e^{\varepsilon r} \|\psi\|^2_{L^{\infty}(\mathbb{R},\ell^2)}+3 \varepsilon^{-1} e^{\varepsilon r} \|\chi\|^2_{L^{\infty}(\mathbb{R},\ell^2)}(1+4c_1^2)\right].$
\end{lm}
\begin{proof}
Let $u(\varrho):=u(\varrho, \tau-t, \mu)$ and $v(\varrho):=u(\varrho, \tau-t, \tilde{\mu})$ be two solutions to system \eqref{ob-3}-\eqref{ob-4}, where $\mathcal{L}_{u_{\tau-t}}=\mu$ and $\mathcal{L}_{v_{\tau-t}}=\tilde{\mu}$, respectively, and
$\mathbb{W}_2(\mu, \tilde{\mu})^2= \mathbb{E}\|u_{\tau-t}-v_{\tau-t}\|_{\mathcal{C}_r}^2.$
Using the It\^o formula, we obtain that for all $\varrho \ge \tau-t$,
\begin{equation} \label{W2-lemma-2+}\small
\begin{split}
 &\; \mathbb{E} (\sup_{\tau-t \le \varrho \le \tau} e^{\varepsilon \varrho}
\|u(\varrho)-v(\varrho)\|^2)+(2\lambda-\varepsilon) \int_{\tau-t}^\tau e^{\varepsilon s}\|u(s)-v(s)\|^2 ds
\\
\le & \; 2 e^{\varepsilon(\tau-t)}\mathbb{E}\|u_{\tau-t}(0)-v_{\tau-t}(0)\|^2
\\
&\; +4\mathbb{E} \sup_{\tau-t \le \varrho \le \tau}\int_{\tau-t}^{\rho}e^{\varepsilon s}(u(s)-v(s), [f(s, u(s), u(s-r), \mathcal{L}_{u(s)})-f(s, v(s), v(s-r), \mathcal{L}_{v(s)})]) ds
\\
&\; + 2\mathbb{E}  \sup_{\tau-t \le \varrho \le \tau} \int_{\tau-t}^{\varrho} e^{\varepsilon s}\|\tilde{\sigma}(s, u(s), u(s-r), \mathcal{L}_{u(s)})
-\tilde{\sigma}(s, v(s), v(s-r), \mathcal{L}_{v(s)})\|_{L_2(\ell^2, \ell^2)}^2 ds
\\
&\; +4 \mathbb{E}  \sup_{\tau-t \le \varrho \le \tau} \bigl|\int_{\tau-t}^{\varrho} e^{\varepsilon s} (u(s)-v(s), [\tilde{\sigma}(s, u(s), u(s-r), \mathcal{L}_{u(s)})-\tilde{\sigma}(s, v(s), v(s-r), \mathcal{L}_{v(s)})]dW(s)) \bigr|.
\end{split}
\end{equation}

We estimate each term in the right-hand side of \eqref{W2-lemma-2+} separately. For the second term, by using \eqref{fi-Lip}
and \eqref{dis-property}, we obtain that
\begin{equation}\label{W2-lemma-3}
\begin{split}
&\, 4 \mathbb{E} \sup_{\tau-t \le \varrho \le \tau} \int_{\tau-t}^{\varrho} e^{\varepsilon s}(u(s)-v(s), [f(s, u(s), u(s-r), \mathcal{L}_{u(s)})
-f(s, v(s), v(s-r), \mathcal{L}_{v(s)})]) ds
\\
\le &\, 4 \mathbb{E} \sup_{\tau-t \le \varrho \le \tau} \int_{\tau-t}^{\varrho} e^{\varepsilon s}(u(s)-v(s), [f(s, u(s), u(s-r), \mathcal{L}_{u(s)})
-f(s, v(s), u(s-r), \mathcal{L}_{u(s)})]) ds
\\
&\, +4 \mathbb{E} \sup_{\tau-t \le \varrho \le \tau} \int_{\tau-t}^{\varrho} e^{\varepsilon s}(u(s)-v(s), [f(s, v(s), u(s-r), \mathcal{L}_{u(s)})
-f(s, v(s), v(s-r), \mathcal{L}_{v(s)})]) ds
\\
\le &\,  4 \mathbb{E} \int_{\tau-t}^{\tau}  e^{\varepsilon s}\|\Theta(s)\|\|u(s)-v(s)\|^2  ds
\\
&\, + 4 \mathbb{E} \int_{\tau-t}^{\tau} \sum_{i \in \mathbb{Z}} \psi_i(s)  e^{\varepsilon s} (|u_i(s-r)-v_i(s-r)|+(\mathbb{E} |u_i(s)-v_i(s)|^2)^{\frac{1}{2}}) |u_i(s)-v_i(s)|ds
\\
\le &\,  4 \mathbb{E} \int_{\tau-t}^\tau  e^{\varepsilon s}\|\Theta(s)\| \|u(s)-v(s)\|^2 ds
+2\mathbb{E} \int_{\tau-t}^\tau  e^{\varepsilon s}\|\psi(s)\| \|u(s)-v(s)\|^2 ds
\\
&\,+4 \mathbb{E} \int_{\tau-t}^\tau  e^{\varepsilon s} \|\psi(s)\|(\|u(s-r)-v(s-r)\|^2+ \mathbb{E}\|u(s)-v(s)\|^2)ds
\\
\le &\, (4\|\Theta\|_{L^{\infty}(\mathbb{R},\ell^2)}+10\|\psi\|_{L^{\infty}(\mathbb{R},\ell^2)})\int_{\tau-t}^\tau \mathbb{E}
\left(e^{\varepsilon s}\|u(s)-v(s)\|^2 \right) ds
\\
&\,+ 4 \varepsilon^{-1} e^{\varepsilon(\tau-t+r)} \|\psi\|^2_{L^{\infty}(\mathbb{R},\ell^2)}\mathbb{ E }\|u_{\tau-t}-v_{\tau-t}\|_{\mathcal{C}_r}^2.
\end{split}
\end{equation}
For the third term, by using \eqref{dis-property} and \eqref{segama-Lip}, we obtain that
\begin{equation} \label{W2-lemma-4+}
\begin{split}
&\, 2\mathbb{E} \sup_{\tau-t \le \varrho \le \tau} \int_{\tau-t}^{\varrho} e^{\varepsilon s}
\|\tilde{\sigma}(s, u(s), u(s-r), \mathcal{L}_{u(s)})-\tilde{\sigma}(s, v(s), v(s-r), \mathcal{L}_{v(s)})\|_{L^2(\ell^2, \ell^2)}^2 ds
\\
\le &\, 6\|\chi\|^2_{L^{\infty}(\mathbb{R},\ell^2)} \int_{\tau-t}^\tau
 e^{\varepsilon s}\mathbb{E}[\|u(s)-v(s)\|^2+\|u(s-r)-v(s-r)\|^2+ \mathbb{E}\|u(s)-v(s)\|^2] ds
\\
\le &\, 18\|\chi\|^2_{L^{\infty}(\mathbb{R},\ell^2)}
\int_{\tau-t}^\tau \mathbb{E} ( e^{\varepsilon s}\|u(s)-v(s)\|^2) ds+6 \varepsilon^{-1} e^{\varepsilon(\tau-t+r)} \|\chi\|^2_{L^{\infty}(\mathbb{R},\ell^2)}
 \mathbb{E}\|u_{\tau-t}-v_{\tau-t}\|_{\mathcal{C}_r}^2.
\end{split}
\end{equation}
For the last term, by using the BDG inequality, \eqref{dis-property} and \eqref{segama-Lip}, we have that
\begin{equation} \label{W2-lemma-5+}\small
\begin{split}
&\,  4 \mathbb{E}\sup_{\tau-t \le \varrho \le \tau} \bigl |  \int_{\tau-t}^{\varrho}  e^{\varepsilon s}(u(s)-v(s), [\tilde{\sigma}(s, u(s), u(s-r), \mathcal{L}_{u(s)})
-\tilde{\sigma}(s, v(s), v(s-r), \mathcal{L}_{v(s)})]dW(s)) \bigr|
\\
\le &\, 4c_1\mathbb{E}\{\sup_{\tau-t \le \varrho \le \tau}  e^{\varepsilon \varrho}\|u(\varrho)-v(\varrho)\|
\\
&\, \qquad\qquad\qquad \cdot
(\int_{\tau-t}^\tau  e^{\varepsilon s} \|\tilde{\sigma}(s, u(s), u(s-r), \mathcal{L}_{u(s)})
-\tilde{\sigma}(s, v(s), v(s-r), \mathcal{L}_{v(s)})\|_{L_2(\ell^2, \ell^2)}^2 ds)^{\frac{1}{2}}\}
\\
\le &\, \frac{1}{2} \mathbb{E}(\sup_{\tau-t \le \varrho\le \tau} e^{\varepsilon \varrho} \|u(\varrho)-v(\varrho)\|^2)
\\
&\,+ 8c_1^2 \int_{\tau-t}^\tau \mathbb{E}  e^{\varepsilon s}\|\tilde{\sigma}(s, u(s), u(s-r), \mathcal{L}_{u(s)})
-\tilde{\sigma}(s, v(s), v(s-r), \mathcal{L}_{v(s)})\|_{L_2(\ell^2, \ell^2)}^2 ds
\\
\le &\, \frac{1}{2} \mathbb{E}(\sup_{\tau-t \le \varrho\le \tau} e^{\varepsilon \varrho}\|u(\varrho)-v(\varrho)\|^2)
+72c_1^2 \|\chi\|^2_{L^{\infty}(\mathbb{R},\ell^2)}\int_{\tau-t}^\tau \mathbb{E}  (e^{\varepsilon s}\|u(s)-v(s)\|^2) ds
\\
&\,+24c_1^2  \varepsilon^{-1} e^{\varepsilon(t-\tau+r)}\|\chi\|^2_{L^{\infty}(\mathbb{R},\ell^2)} \|u_{\tau-t}-v_{\tau-t}\|_{\mathcal{C}_r}^2.
\end{split}
\end{equation}
By combining \eqref{W2-lemma-2+}-\eqref{W2-lemma-5+}, we obtain that
\begin{equation}\label{W2-lemma-6+}
\begin{split}
 &\,  \mathbb{E} (\sup_{\tau-t \le \varrho \le \tau} e^{\varepsilon \varrho} \|u(\varrho)-v(\varrho)\|^2)+(2\lambda-\varepsilon)
 \int_{\tau-t}^\tau \mathbb{E} (e^{\varepsilon s}\|u(s)-v(s)\|^2) ds
 \\
  \le &\, \frac{1}{2} \mathbb{E}(\sup_{\tau-t \le \varrho\le \tau}  e^{\varepsilon \varrho}\|u(\varrho)-v(\varrho)\|^2)
\\
&\, + \left[ 4\|\Theta\|_{L^{\infty}(\mathbb{R},\ell^2)}+10\|\psi\|_{L^{\infty}(\mathbb{R},\ell^2)}+18\|\chi\|^2_{L^{\infty}(\mathbb{R},\ell^2)}(1+4c_1^2)\right]
\int_{\tau-t}^\tau \mathbb{E} (e^{\varepsilon s}\|u(s)-v(s)\|^2) ds
\\
&\,+\left[ 2+4 \varepsilon^{-1} e^{\varepsilon r} \|\psi\|^2_{L^{\infty}(\mathbb{R},\ell^2)}+6 \varepsilon^{-1} e^{\varepsilon r} \|\chi\|^2_{L^{\infty}(\mathbb{R},\ell^2)}(1+4c_1^2)\right] e^{\varepsilon(\tau-t)} \mathbb{E}\|u_{\tau-t}-v_{\tau-t}\|_{\mathcal{C}_r}^2.
\end{split}
\end{equation}
By virtue of \eqref{lamda-add}, we have that
\begin{equation} \label{add-1}
\begin{split}
 &\, \mathbb{E} (\sup_{\tau-t \le \varrho \le \tau} e^{\varepsilon \varrho}\|u(\varrho)-v(\varrho)\|^2)
\\
\le&\, 4\left[ 1+2 \varepsilon^{-1} e^{\varepsilon r} \|\psi\|^2_{L^{\infty}(\mathbb{R},\ell^2)}+3 \varepsilon^{-1} e^{\varepsilon r} \|\chi\|^2_{L^{\infty}(\mathbb{R},\ell^2)}(1+4c_1^2) \right] e^{\varepsilon(\tau-t)} \mathbb{E}\|u_{\tau-t}-v_{\tau-t}\|_{\mathcal{C}_r}^2.
\end{split}
\end{equation}
Finally, noticing that for all $t \ge r$
\begin{eqnarray*}
\mathbb{E}(\sup_{\tau-r \le \varrho \le \tau}\|u(\varrho)-v(\varrho)\|^2)
\!\!\!&\le&\!\!\!
e^{-\varepsilon(\tau-r)}  \mathbb{E}(\sup_{\tau-t \le \varrho \le \tau}e^{\varepsilon \varrho}\|u(\varrho)-v(\varrho)\|^2),
\end{eqnarray*}
along with \eqref{add-1}, we complete this proof.
\end{proof}

As an immediate consequence of Lemma \ref{im1}, we obtain that
\begin{cor}\label{cor1}
Suppose  that {\bf (H1)}-{\bf (H3)} and \eqref{lamda-add} hold, and $\psi(\cdot), \Theta(\cdot) \in L^{\infty}(\mathbb{R}, \ell^2)$.  Then, for every
$\tau \in \mathbb{R}$, $t\ge r$ and
$\mu, \tilde{\mu}\in \mathcal{P}_2(\mathcal{C}_r)$, we have
\begin{eqnarray*}
\mathbb{W}_2(P_{\tau-t, \tau}\mu, P_{\tau-t, \tau}\tilde{\mu}) \le  \tilde{c}_3^{\frac{1}{2}} e^{-\frac{1}{2}\varepsilon(t-r)} \mathbb{W}_2(\mu, \tilde{\mu}),
\end{eqnarray*}
where $\varepsilon>0$ is the same number as in \eqref{epsilca}
\end{cor}

With the preparations above, we now apply Lemma \ref{im1} to prove the main results of this section.
\begin{thm}\label{main-s}
Assume that  {\bf (H1)}-{\bf (H3)}, \eqref{lamda}, \eqref{g^2}-\eqref{g^4} and \eqref{lamda-add} hold, and $\psi(\cdot), \Theta(\cdot) \in L^{\infty}(\mathbb{R}, \ell^2)$. Then, the following results hold:
\begin{enumerate}
\item [(i)]
The  $\mathcal{D}$-pullback
    measure attractor
$\mathcal{A}$ of
$\Phi$  is a singleton;  that is,
$\mathcal{A}= \{   \mu (\tau)
: \tau \in \mathbb{R}   \}$.

\item[(ii)]
If $f, g$ and $\sigma$ are all time independent, then  the stochastic lattice equation with delays \eqref{ob-3} has a unique invariant measure
$\mu$ in $\mathcal{P}_4(\mathcal{C}_r)$. Moreover, $\mu$ is exponentially mixing in the sense that for any
$\upsilon \in \mathcal{P}_2 (\mathcal{C}_r)$,
\begin{align}\label{main_s 2}
\mathbb{W}_2
\left (
P_{0, t}^* \upsilon, \ \mu
\right )
\le \tilde{c}_3^{\frac{1}{2}} e^{-{\frac 12}\varepsilon (t-r)}\mathbb{W}_2\left (\upsilon, \ \mu \right) , \quad \forall t\ge r,
\end{align}
where $\varepsilon>0$ is the same number as in \eqref{epsilca}.
 \end{enumerate}
 \end{thm}

\begin{proof} (i) According to Theorem \ref{main_e}, let $\mathcal{A}= \{\mathcal{A}(\tau): \tau \in \mathbb{R}\}$ denote
the unique $\mathcal{D}$-pullback measure attractor of $\Phi$.  We will prove that $\mathcal{A}(\tau)$ is a set with a single point  for each
$\tau \in \mathbb{R}$. To this end, for any $\mu,\tilde{ \mu} \in \mathcal{A}(\tau)$, we aim  to show $\mu =\tilde{\mu}$.

Indeed, let $\{t_n\}_{n=1}^\infty$ be a sequence of positive numbers such that $t_n \to \infty$.
By the invariance of $\mathcal{A}$, for each $n\in \mathbb{N}^+$, there exist $\mu_{\tau -t_n}, \tilde{\mu}_{\tau-t_n} \in \mathcal{A} (\tau -t_n)$
such that
\begin{eqnarray*}
\mu \!\!\!&=&\!\!\! \Phi(t_n, \tau -t_n) \mu_{\tau -t_n} = \mathcal{L}_{u_\tau(\cdot,\, \tau, \,\mu_{\tau -t_n})}
\end{eqnarray*}
and
\begin{eqnarray*}
\tilde{\mu}\!\!\!&=&\!\!\!\Phi(t_n, \tau -t_n) \tilde{\mu}_{\tau -t_n}= \mathcal{L}_{u_\tau(\cdot,\, \tau,\,\tilde{ \mu}_{\tau -t_n})}.
\end{eqnarray*}
By using Lemma \ref{im1}, we get
\begin{equation}\label{im2 p1}
\begin{split}
   d_{\mathcal{P} (\mathcal{C}_r)}(\mu, \tilde{\mu}) = & \, \sup_{\psi\in L_b (\mathcal{C}_r), \|\psi\|_{L_b(\mathcal{C}_r)} \le 1}
   \left |\int_{\mathcal{C}_r} \psi d\mathcal{L}_{u_\tau(\cdot, \, \tau-t_n, \, \mu_{\tau-t_n})}-\int_{\mathcal{C}_r} \psi
   \mathcal{L}_{u_\tau(\cdot,\, \tau-t_n, \,\tilde{\mu}_{\tau-t_n})}\right|
   \\
    = &\, \sup_{\psi\in L_b(\mathcal{C}_r), \|\psi\|_{L_b}(\mathcal{C}_r) \le 1} \left|\mathbb{E} \left(\psi(u_\tau( \cdot, \tau -t_n, \mu_{\tau-t_n}))\right)
    -\mathbb{E}\left(\psi(u_\tau(\cdot, \tau-t_n, \tilde{\mu}_{\tau-t_n}))\right)\right|
   \\
   \le &\,
   \mathbb{E}\left(\|u_\tau(\cdot, \tau -t_n, \mu_{\tau-t_n})-u_\tau(\cdot, \tau -t_n, \tilde{\mu}_{\tau-t_n})\|_{\mathcal{C}_r}\right)
    \\
    \le &\,
  \tilde{c}_3^{\frac{1}{2}} e^{-{\frac 12}\varepsilon (t_n-r)}\mathbb{W}_2(\mu_{\tau-t_n}, \tilde{\mu}_{\tau-t_n})
    \\
    \le &\,
  2 \tilde{c}_3^{\frac{1}{2}} e^{-{\frac 12}\varepsilon (t_n-r)} \|\mathcal{A}(\tau-t_n)\|_{\mathcal{P}_2(\mathcal{C}_r)}
  \\
  \le &\, 2 \tilde{c}_3^{\frac{1}{2}} e^{-\frac{1}{2}\varepsilon(\tau-r)}
  (e^{2 \varepsilon (\tau-t_n)}\|\mathcal{A}(\tau-t_n)\|_{\mathcal{P}_4(\mathcal{C}_r)}^4)^{\frac{1}{4}}.
 \end{split}
\end{equation}
    Since $\mathcal{A} \in \mathcal{D}$, we
    have
   $$
   \lim_{t\to \infty}
    e^{ 2 \varepsilon (\tau - t )} \|\mathcal{A} (\tau -t )\|^4
  _{\mathcal{P}_4(\mathcal{C}_r)} =0,
  $$
  which implies that the right-side of \eqref{im2 p1}
  converges to zero as $t_n\to \infty$.
  Hence,
  $\mu=\tilde{\mu}$.

  (ii) If   $f, g$ and $\sigma$ are all time independent,
then $\mu(\tau)=\mu$ for any $\tau \in \mathbb{R}$. Hence, $\mu$ is
an invariant measure.
Applying Corollary \ref{cor1}, we obtain \eqref{main_s 2},
  which  completes the proof.
  \end{proof}

\section{Upper semicontinuity of pullback measure attractors}

In this section, we study the limiting behavior of pullback measure attractors for the distribution-dependent non-autonomous stochastic lattice system with delays.
\vskip0.1in
To this end, we consider two systems \eqref{ob-1'} and \eqref{ob-1''}
on the integer set $\mathbb{Z}$.
Unless otherwise specified, we assume that all the functions $f^\epsilon=(f^\epsilon_i)_{i \in \mathbb{Z}}$ and $\sigma^\epsilon=(\sigma_i^\epsilon)_{i \in \mathbb{Z}}$ in \eqref{ob-1'} satisfy the conditions {\bf  (H1)}-{\bf (H3)} uniformly with respect to $\epsilon \in (0,1)$. The functions $f$ and $\sigma$ in \eqref{ob-1''}
satisfy the conditions corresponding to {\bf  (H1)}-{\bf (H3)} without the terms involving distributions $\mu, \mu_1$ or $\mu_2$. We also assume that
\eqref{lamda} and \eqref{g^2}-\eqref{g^4}  hold.

\vskip0.1in
Furthermore, we assume that for all $i \in \mathbb{Z}$, $\epsilon \in (0,1)$,  $t, u, v \in \mathbb{R}$ and $\mu \in \mathcal {P}_2(\mathbb{R})$,
\begin{align}
\left |f_i^\epsilon(t, u, v, \mu)-f_i(t, u, v)  \right | \le & \, \epsilon \tilde{\varphi}_i (t)(|u|+|v|+\mathbb{W}_2^{\mathbb{R}}(\mu, \delta_0)) \label{limc1}
\end{align}
and
\begin{align}
\left|\sigma_i^\epsilon(t, u, v, \mu)-\sigma_i(t,u, v)\right| \le & \, \epsilon \tilde{\psi}_i(t)(|u|+|v|
+\mathbb{W}_2^{\mathbb{R}}(\mu, \delta_0)), \label{limc3}
\end{align}
where $\tilde{\varphi}=(\tilde{\varphi}_i)_{i \in \mathbb{Z}}$, $\tilde{\phi}=(\tilde{\phi}_i)_{i \in \mathbb{Z}}$ and
$\tilde{\psi}=(\tilde{\psi}_i)_{i \in \mathbb{Z}} \in L^{\infty}(\mathbb{R},\ell^2)$.

\vskip0.1in
In exactly the same manner as before, we rewrite \eqref{ob-1'} and \eqref{ob-1''} in the following abstract form:
\begin{equation}\label{ob-3'}
\begin{split}
& du^\epsilon(t)+\nu A u^\epsilon(t) dt+\lambda u^\epsilon(t) dt
\\
& =(f^\epsilon(t, u^\epsilon(t),u^\epsilon(t-r),\mathcal{L}_{u^\epsilon(t)})+g(t)) dt
+\tilde{\sigma}^\epsilon(t, u^\epsilon(t),u^\epsilon(t-r), \mathcal{L}_{u^\epsilon(t)})dW(t), \quad t >\tau
\end{split}
\end{equation}
and
\begin{equation}\label{ob-3''}
\begin{split}
& du(t)+\nu A u(t) dt+\lambda u(t) dt
\\
&=(f(t, u(t),u(t-r))+g(t)) dt+\tilde{\sigma}(t, u(t),u(t-r))dW(t), t >\tau,
\end{split}
\end{equation}
where the operator $\tilde{\sigma}^\epsilon: \mathbb{R} \times \ell^2 \times \ell^2 \times \mathcal{L}^2 \rightarrow L(\ell^2, \ell^2)$ is defined
by \begin{eqnarray*}
\tilde{\sigma}^\epsilon(t, u,v, \mu)w=(\sigma_i^\epsilon(t, u_i, v_i, \mu_i)w_i)_{i \in \mathbb{Z}}
\end{eqnarray*}
and the operator $\tilde{\sigma}: \mathbb{R} \times \ell^2 \times \ell^2 \rightarrow L(\ell^2, \ell^2)$ is defined by
\begin{eqnarray*}
\tilde{\sigma}^\epsilon(t, u,v)w=(\sigma_i^\epsilon(t, u_i, v_i)w_i)_{i \in \mathbb{Z}}
\end{eqnarray*}
for every $u=(u_i)_{i \in \mathbb{Z}}$, $v=(v_i)_{i \in \mathbb{Z}}$, $w=(w_i)_{i \in \mathbb{Z}} \in \ell^2$ and $\mu=(\mu_i)_{i \in \mathbb{Z}} \in \mathcal{L}^2$.
\vskip0.1in
Clearly, by using \eqref{limc3}, for all $t \in \mathbb{R}$, $u,v \in \ell^2$ and $\mu \in \mathcal{L}^2$, we have
\begin{align}
\left\|\tilde{\sigma}^\epsilon(t, u, v, \mu)-\tilde{\sigma}(t,u, v)\right\|_{L_2(\ell^2, \ell^2)}^2 \le & \, 3 \epsilon^2 \|\tilde{\psi}(t)\|^2
(\|u\|^2+\|v\|^2+\rho(\mu,\hat{\delta}_{0})^2), \label{limc3'}
\end{align}

We denote by $u^\epsilon(t, \tau, u_\tau)$ the solution of \eqref{ob-3'} with the initial condition $u^\epsilon(\tau)=u_\tau$.
The next lemma is about the convergence of solutions of \eqref{ob-3'} as $\epsilon \to 0$.

\begin{lm}\label{up1}
Assume that \eqref{limc1}-\eqref{limc3} hold. Then
for every $v \in \mathbb{R}^+$, $\tau \in \mathbb{R}$ and $R>0$, the solutions of \eqref{ob-3'} and \eqref{ob-3''} satisfy:
 $$\lim\limits_{\epsilon  \to 0^+ }\ \sup\limits_{ \|u_\tau\|_{L^2(\Omega, \mathcal{C}_r)}\le R} \  \
 \mathbb{E} \left(\|u_{\tau+v}^\epsilon (\cdot, \tau, u_\tau)-u_{\tau+v}(\cdot, \tau, u_\tau)\|_{\mathcal{C}_r}^2 \right)=0.$$
\end{lm}

\begin{proof}
Let $ u^\epsilon(t)=u^\epsilon(t, \tau, u_\tau )$, $u(t)=u(t, \tau, u_\tau)$ and $w(t)=u^\epsilon(t)-u(t)$.
By \eqref{ob-3'} and \eqref{ob-3''}, we obtain that for all $t>\tau$,
\begin{equation*}
\begin{split}
& dw(t)+\nu A w(t) dt + \lambda w(t)dt
\\
& = \big(f^\epsilon \left(t,u^\epsilon(t), u^\epsilon(t-r), \mathcal{L}_{u^\epsilon(t)}\right)-f\left(t,u(t),u(t-r)\right)\big)dt
+\big(g^\epsilon(t)-g(t)\big)dt
\\
&\quad +\big(\sigma^\epsilon(t, u^\epsilon(t), u^\epsilon(t-r), \mathcal{L}_{u^\epsilon(t)})-\sigma(t, u(t), u(t-r))\big)dW(t).
\end{split}
\end{equation*}
By the It\^{o} formula, taking the supremum and expectation, we get all $t \ge \tau$,
\begin{equation} \label{up1-3}
 \begin{split}
 & \mathbb{E} (\sup_{\tau \le \varrho \le \tau+v} \|w(\varrho)\|^2)
 \\
  \le & 2\mathbb{E} \sup_{\tau \le \varrho \le \tau+v} \int_ {\tau}^{\varrho} ([f^\epsilon \left(s,u^\epsilon(s), u^\epsilon(s-r), \mathcal{L}_{u^\varepsilon(s)}\right)-f\left(s,u(s),u(s-r)\right)], w(s)) ds
 \\
 & + \mathbb{E} \sup_{\tau \le \varrho \le \tau+v}  \int_{\tau}^{\varrho} \|\sigma^\epsilon(s, u^\epsilon(s), u^\epsilon(s-r), \mathcal{L}_{u^\epsilon(s)})-\sigma(s, u(s), u(s-r))\|_{L_2(\ell^2,\ell^2)}^2 ds
 \\
 &+ 2 \mathbb{E} \sup_{\tau \le \varrho \le \tau+v} \biggl|\int_ {\tau}^{\varrho} (w(s), \left[ \sigma^\epsilon(s, u^\epsilon(s), u^\epsilon(s-r), \mathcal{L}_{u^\epsilon(s)})-\sigma(s, u(s), u(s-r)) \right] dW(s))\biggr|.
 \end{split}
 \end{equation}
Next, we estimate each term on the right-hand side of \eqref{up1-3} individually. For the first term on the right-hand side of
\eqref{up1-3}, we notice that
\begin{equation*}
 \begin{split}
& \int_ {\tau}^{\varrho} ([f^\epsilon \left(s,u^\epsilon(s), u^\epsilon(s-r), \mathcal{L}_{u^\epsilon(s)}\right)-f\left(s,u(s),u(s-r)\right)], w(s)) ds
\\
=& \int_ {\tau}^{\varrho} ([f^\epsilon \left(s,u^\epsilon(s), u^\epsilon(s-r), \mathcal{L}_{u^\epsilon(s)}\right)-f^\epsilon\left(s,u(s),u^\epsilon(s-r), \mathcal{L}_{u^\epsilon(s)}\right)], w(s)) ds
\\
& +\int_ {\tau}^{\varrho} ([f^\epsilon\left(s,u(s),u^\epsilon(s-r), \mathcal{L}_{u^\epsilon(s)}\right)
-f^\epsilon \left(s,u(s), u(s-r), \mathcal{L}_{u(s)}\right)], w(s)) ds
\\
& + \int_ {\tau}^{\varrho} ([f^\epsilon \left(s,u(s), u(s-r), \mathcal{L}_{u(s)}\right)
-f\left(s,u(s),u(s-r)\right)], w(s)) ds.
 \end{split}
 \end{equation*}
Along with \eqref{fi-2}, \eqref{f-Lip} and \eqref{limc3'}, we get that
\begin{equation}\label{up1-4}
\begin{split}
& 2\mathbb{E} \sup_{\tau \le \varrho \le \tau+v} \int_ {\tau}^{\varrho} ([f^\epsilon \left(s,u^\epsilon(s), u^\epsilon(s-r), \mathcal{L}_{u^\epsilon(s)}\right)-f\left(s,u(s),u(s-r)\right)], w(s)) ds
\\
\le & \, 2 \mathbb{E} \int_{\tau}^{\tau+v} \|\Theta(s)\| \|w(s)\|^2 ds
+2 \mathbb{ E} \int_{\tau}^{\tau+v} \sum_{i \in \mathbb{Z}} \psi_i(s)
(|w_i(s-r)|+ (\mathbb{E}|w_i(s)|^2)^{\frac12}) |w_i(s)| ds
\\
&\, +2\epsilon \mathbb{E} \int_{\tau}^{\tau+v}   \sum_{i \in \mathbb{Z}}
\tilde{\varphi}_i (t)(|u_i(s)|+|u_i(s-r)|+(\mathbb{E}|u_i(s)|^2)^{\frac{1}{2}})|w_i(s)|ds
\end{split}
\end{equation}
\begin{equation*}
 \begin{split}
\le & \,  2\|\Theta\|_{L^{\infty}([\tau, \tau+v],\ell^2 )} \int_{\tau}^{\tau+v} \mathbb{E}\|w(s)\|^2 ds
+ \mathbb{E}\int_{\tau}^{\tau+v} \|\psi(s) \|\|w(s)\|^2 ds
\\
& \, +\mathbb{E} \int_{\tau}^{\tau+v}  \|\psi(s)\|
(\|w(s-r)\|^2 + \mathbb{E}\|w(s)\|^2) ds
\\
&\,+3\epsilon \mathbb{E} \int_{\tau}^{\tau+v}
\|\tilde{\varphi}(t)\|(\|u(s)\|^2+\|u(s-r)\|^2+\mathbb{E}\|u(s)\|^2)ds
+\epsilon \mathbb{E} \int_{\tau}^{\tau+v} \|\tilde{\varphi}(t)\| \|w(s)\|^2ds
\\
\le & \,  (2\|\Theta\|_{L^{\infty}([\tau, \tau+v],\ell^2 )}+3\|\psi\|_{L^{\infty}([\tau, \tau+v],\ell^2)}+\epsilon\|\tilde{\varphi}\|_{L^{\infty}([\tau, \tau+v],\ell^2 )})
\mathbb{E} \int_{\tau}^{\tau+v} \|w(s)\|^2 ds
\\
&\,+9\epsilon\|\tilde{\varphi}\|_{L^{\infty}([\tau, \tau+v],\ell^2 )} \int_{\tau}^{\tau+v}
\mathbb{E}\|u(s)\|^2ds+3\epsilon r \|\tilde{\varphi}\|_{L^{\infty}([\tau, \tau+v],\ell^2 )}
\mathbb{E}\|u_\tau\|_{\mathcal{C}_r}^2.
\end{split}
\end{equation*}
For the second term on the right-hand side of \eqref{up1-3}, we notice that
\begin{eqnarray*}
\!\!\!&&\!\!\! \int_{\tau}^{\varrho} \|\sigma^\epsilon(s, u^\epsilon(s), u^\epsilon(s-r), \mathcal{L}_{u^\epsilon(s)})-\sigma(s, u(s), u(s-r))\|_{L_2(\ell^2,\ell^2)}^2 ds
\\
\!\!\!&\le&\!\!\! \int_{\tau}^{\varrho} \|\sigma^\epsilon(s, u^\epsilon(s), u^\epsilon(s-r), \mathcal{L}_{u^\epsilon(s)})-\sigma^\epsilon(s, u(s), u(s-r),\mathcal{L}_{u(s)})\|_{L_2(\ell^2,\ell^2)}^2 ds
\\
\!\!\!&&\!\!\! +\int_{\tau}^{\varrho} \|\sigma^\epsilon(s, u(s), u(s-r), \mathcal{L}_{u(s)})-\sigma(s, u(s), u(s-r))\|_{L_2(\ell^2,\ell^2)}^2 ds,
\end{eqnarray*}
By \eqref{segama-Lip} and \eqref{limc3'}, we get that
\begin{equation}\label{up1-6}
\begin{split}
& \mathbb{E} \sup_{\tau \le \varrho \le \tau+v}  \int_{\tau}^{\varrho} \|(\sigma^\epsilon(s, u^\epsilon(s), u^\epsilon(s-r), \mathcal{L}_{u^\epsilon(s)})-\sigma(s, u(s), u(s-r)))\|_{L_2(\ell^2,\ell^2)}^2 ds
\\
\le & \,  3\mathbb{ E} \int_{\tau}^{\tau+v}  \|\chi(t)\|^2(\|w(s)\|^2+\|w(s-r)\|^2+\mathbb{E}\|w(s)\|^2)ds
\\
&\, +3 \epsilon^2 \mathbb{ E} \int_{\tau}^{\tau+v} \|\tilde{\psi}(t)\|^2
(\|u(s)\|^2+\|u(s-r)\|^2+\mathbb{E}\|u(s)\|^2) ds
\end{split}
\end{equation}
\begin{equation*}
\begin{split}
\le & \,  9\|\chi\|^2_{L^{\infty}(\mathbb{R}, \ell^2)}
\mathbb{E}\int_ {\tau}^{\tau+v}\|w(s)\|^2 ds +9\epsilon^2 \|\tilde{\psi}\|^2_{L^{\infty}([\tau, \tau+v], \ell^2)}
\int_ {\tau}^{\tau+v}\mathbb{E}\|u(s)\|^2 ds
\\
& \, +3\epsilon^2 r \|\tilde{\psi}\|^2_{L^{\infty}([\tau, \tau+v], \ell^2)}
\mathbb{E}\|u_{\tau}\|_{\mathcal{C}_r}^2.
\end{split}
\end{equation*}
For the last term on the right-hand side of \eqref{up1-3}, by \eqref{up1-6} we have
\begin{equation}\label{up1-7}
\begin{split}
&    2 \mathbb{E} \sup_{\tau \le \varrho \le \tau+v} \biggl|\int_ {\tau}^{\varrho} (w(s),
\left[ \sigma^\epsilon(s, u^\epsilon(s), u^\epsilon(s-r), \mathcal{L}_{u^\epsilon(s)})-\sigma(s, u(s), u(s-r))\right] dW(s))\biggr|
\\
\le &\, 2c_1\mathbb{E}\{\sup_{\tau \le \varrho\le \tau+v} \|w(\varrho)\|^2
\\
&\,\qquad\qquad\qquad \cdot
(\int_{\tau}^{ \tau+v }
\|\sigma^\epsilon(s, u^\epsilon(s), u^\epsilon(s-r), \mathcal{L}_{u^\epsilon(s)})-\sigma(s, u(s), u(s-r)) \|_{L_2(\ell^2, \ell^2)}^2 ds)\}^{\frac{1}{2}}
\\
\le &\, \frac{1}{2} \mathbb{E}(\sup_{\tau \le \varrho\le \tau+v}\|w(\varrho)\|^2)
\\
&\,+ 2 c_1^2 \mathbb{E}\int_{\tau}^{\tau+v}
\|\sigma^\epsilon(s, u^\epsilon(s), u^\epsilon(s-r), \mathcal{L}_{u^\epsilon(s)})-\sigma(s, u(s), u(s-r))\|_{L_2(\ell^2, \ell^2)}^2 ds
\\
\le &\,  \frac{1}{2} \mathbb{E}(\sup_{\tau \le \varrho\le \tau+v}\|w(\varrho)\|^2)
+18c_1^2 \|\chi\|^2_{L^{\infty}(\mathbb{R}, \ell^2)}
\mathbb{E}\int_ {\tau}^{\tau+v}\|w(s)\|^2 ds
\\
&\, +18c_1^2\epsilon^2 \|\tilde{\psi}\|^2_{L^{\infty}([\tau, \tau+v], \ell^2)}
\int_ {\tau}^{\tau+v}\mathbb{E}\|u(s)\|^2 ds
+6 c_1^2\epsilon^2 r \|\tilde{\psi}\|^2_{L^{\infty}([\tau, \tau+v], \ell^2)}
\mathbb{E}\|u_{\tau}\|_{\mathcal{C}_r}^2.
\end{split}
\end{equation}
Combining \eqref{up1-3}-\eqref{up1-7}, we get that
 \begin{equation*}
\begin{split}
&  \mathbb{E} (\sup_{\tau \le \varrho\le \tau+v} \|w(\varrho)\|^2)
\\
\le &\, \frac{1}{2} \mathbb{E}(\sup_{\tau \le \varrho\le \tau+v}\|w(\varrho)\|^2)+\{(2\|\Theta\|_{L^{\infty}([\tau, \tau+v],\ell^2 )}+3\|\psi\|_{L^{\infty}([\tau, \tau+v],\ell^2)}+\epsilon\|\tilde{\varphi}\|_{L^{\infty}([\tau, \tau+v],\ell^2 )})
\\
& \, +9 \|\chi\|^2_{L^{\infty}(\mathbb{R}, \ell^2)}(1+2c_1^2)\}
\mathbb{E}\int_{\tau}^{\tau+v} \|w(s)\|^2 ds
\\
& +\left[9\epsilon\|\tilde{\varphi}\|_{L^{\infty}([\tau, \tau+v],\ell^2 )}+9\epsilon^2 \|\tilde{\psi}\|^2_{L^{\infty}([\tau, \tau+v], \ell^2)}
(1+2c_1^2)\right]\int_{\tau}^{\tau+v}
\mathbb{E}\|u(s)\|^2ds
\\
&+\epsilon r\left[3 \|\tilde{\varphi}\|_{L^{\infty}([\tau, \tau+v],\ell^2 )}+3\epsilon\|\tilde{\psi}\|^2_{L^{\infty}([\tau, \tau+v], \ell^2)}
(1+2c_1^2)\right]
\mathbb{E}\|u_\tau\|_{\mathcal{C}_r}^2.
\end{split}
\end{equation*}
Thus, there exist constants $C_2$ and $C_3$ such that
 \begin{equation*}
\begin{split}
&  \mathbb{E} (\sup_{\tau \le \varrho\le \tau+v} \|w(\varrho)\|^2)
\\
\le &\, C_2 \mathbb{E} \int_{\tau}^{\tau+v}\|w(s)\|^2 ds+\epsilon C_3(1+\mathbb{E}\|u_\tau\|_{\mathcal{C}_r}^2+
\int_{\tau}^{\tau+v}\mathbb{E}\|u(s)\|^2 ds)
\\
\le &\, C_2 \mathbb{E}\int_{\tau}^{\tau+v} \sup_{\tau \le \varrho\le s}\|w(\varrho)\|^2 ds+\epsilon C_3(1+\mathbb{E}\|u_\tau\|_{\mathcal{C}_r}^2+
\int_{\tau}^{\tau+v}\mathbb{E}\|u(s)\|^2 ds).
\end{split}
\end{equation*}
Using the Gronwall inquality, we obtain that
\begin{eqnarray*}
\mathbb{E} (\sup_{\tau \le \varrho\le \tau+v} \|w(\varrho)\|^2) \le
\epsilon C_3(1+\mathbb{E}\|u_\tau\|_{\mathcal{C}_r}^2+
\int_{\tau}^{\tau+v}\mathbb{E}\|u(s)\|^2 ds) e^{C_2 \tau}.
\end{eqnarray*}
Finally, using the condition that $\mathbb{E}(\|u_\tau\|_{\mathcal{C}_r}^2) \le R^2$ and the facts \eqref{cha} and
\begin{eqnarray*}
\mathbb{E} (\sup_{\tau-r \le \varrho\le \tau} \|w(\varrho)\|^2)=0.
\end{eqnarray*}
Thus, the proof of Lemma \ref{up1} is complete.
\end{proof}

We now write
 the
non-autonomous dynamical system associated
with  \eqref{ob-1'}
as $\Phi^\epsilon$
to highlight its dependence
on $\epsilon$,  and  use $\Phi$ for  the
 dynamical system associated
with  \eqref{ob-1''}.
The $\mathcal{D}$-pullback measure
attractors of $\Phi^\epsilon$ and
$\Phi$ are denoted by
$\mathcal{A}^\epsilon$  and
$\mathcal{A}$ respectively.

\vskip0.1in

 As an immediate consequence of Lemma \ref{up1}, we obtain
the following convergence result
for  $\Phi^\epsilon$.

\begin{cor}\label{up2}
If \eqref{limc1}-\eqref{limc3} hold, then
for every  $t\in \mathbb{R}^+$,
$\tau \in \mathbb{R}$ and $R>0$, $\Phi^\epsilon$ satisfies:
   $$ \lim_{\epsilon \to 0^+}\ \sup_{  \mu (\|\cdot\|_{\mathcal{C}_r}^2) \le R}\ \mathbb{W}_2 \left (\Phi^\epsilon (t,\tau)\mu,\
   \Phi (t,\tau)\mu \right)=0,$$
and hence
$$\lim_{\epsilon \to 0^+}\ \sup_{  \mu (\|\cdot\|_{\mathcal{C}_r}^2) \le R}\ d_{{\mathcal{P}(\mathcal{C}_r)}}
  \left (\Phi^\epsilon (t,\tau)\mu, \ \Phi(t,\tau)\mu \right)=0.$$
\end{cor}

At the end of this section,  we prove the upper semi-continuity of $\mathcal{D}$-pullback measure attractors
of \eqref{ob-3'} as $\epsilon \rightarrow 0^+$.

\begin{thm}\label{main_up}
Assume that \eqref{limc1}-\eqref{limc3} hold. Then, for every $\tau \in \mathbb{R}$,
$$\lim\limits_{ \epsilon \rightarrow 0^+} d_{ \mathcal{P}_4(\mathcal{C}_r)}\Big(\mathcal{A}^\epsilon(\tau),
  \mathcal{A}(\tau)\Big)=0,$$
where $d_{ \mathcal{P}_4(\mathcal{C}_r)}(\cdot, \cdot)$ is the Hausdorff semi-distance of sets in terms of $\mathcal{P}_4(\mathcal{C}_r)$.
\end{thm}

\begin{proof}
Since $f^\epsilon$, $g^\epsilon$, $\sigma^\epsilon$  and $f$, $g$, $\sigma$ satisfy {(\bf H1)-(\bf H3)} uniformly with respect to
$\epsilon \in (0,1)$,
we have that the family $K=\{K(\tau): \tau \in \mathbb{R}\}$ given by \eqref{absorbset} is a
closed $\mathcal{D}$-pullback absorbing set of $\Phi$ and $\Phi^\epsilon$ for all $\epsilon\in (0,1)$.
In particular, for every $\tau \in \mathbb{R}$, $\mathcal{A}(\tau)$ and $\mathcal{A}^\epsilon(\tau)$ are subsets of $K(\tau)$
for all $\epsilon \in (0,1)$.

Since $\mathcal{A} $ is the $\mathcal{D}$-pullback measure attractor of $\Phi$ in $(\mathcal{P}_4(\mathcal{C}_r), d_{\mathcal{P}(\mathcal{C}_r)})$,
we obtain that for every $\delta>0$, there exists $T=T(\delta, \tau, K)>0$ such that
\begin{align} \label{main_up p1}
d_{\mathcal{P}_4(\mathcal{C}_r)}
\big(\Phi(T, \tau-T)K(\tau-T), \mathcal{A}(\tau)\big)<\frac{1}{2}\delta.
\end{align}
By Corollary \ref{up2} we have
 $$\lim\limits_{\epsilon \to 0}\sup\limits_{\mu \in K(\tau-T)}d_{ \mathcal{P}_4(\mathcal{C}_r)}
\big(\Phi^\epsilon (T, \tau-T)\mu, \Phi(T, \tau-T)\mu \big)=0.$$
Hence, there exists $\epsilon_1 \in (0,1)$ such that for all $\epsilon \in (0,\epsilon_1)$,
\begin{align} \label{main_up p2}
\sup\limits_{\mu \in K(\tau-T)}d_{ \mathcal{P}_4(\mathcal{C}_r)}
       \big(\Phi^\epsilon (T, \tau-T) \mu, \Phi(T, \tau-T) \mu \big)< \frac{1}{2}\delta.
\end{align}
Since $\mathcal{A}^\epsilon ( \tau-T ) \subseteq K(\tau-T)$,
by \eqref{main_up p1} and \eqref{main_up p2}, we obtain that
\begin{align} \label{main_up p3}
d_{ \mathcal{P}_4(\mathcal{C}_r)}\big(\Phi(T, \tau-T)\mathcal{A}^\epsilon(\tau-T), \mathcal{A}(\tau)\big)
<\frac{1}{2}\delta.
\end{align}
and for all $\epsilon \in (0,\epsilon_1)$,
 \begin{align} \label{main_up p4}
\sup\limits_{\mu \in   {\mathcal{A}}^\epsilon (\tau-T)}d_{ \mathcal{P}_4(\mathcal{C}_r)}
\big(\Phi^\epsilon (T, \tau-T) \mu, \Phi(T, \tau-T)\mu \big)
<\frac{1}{2} \delta.
\end{align}
By \eqref{main_up p3}-\eqref{main_up p4}, we have that for all $\epsilon \in (0,\epsilon_1)$,
\begin{align} \label{main_up p5}
\sup\limits_{ \mu \in \mathcal{A}^\epsilon(\tau-T)} d_{ \mathcal{P}_4 (\mathcal{C}_r) }
       \big(\Phi^\epsilon (T,\tau-T) \mu, \mathcal{A}(\tau)\big)
<\delta.
\end{align}
Using \eqref{main_up p5} and the invariance of $\mathcal{A}^\epsilon$,
we get for all $\epsilon \in (0,\epsilon_1)$,
$$\sup\limits_{ \mu \in \mathcal{A}^\epsilon(\tau)}d_{\mathcal{P}_4(\mathcal{C}_r)}
\big(\mu, \mathcal{A}(\tau)\Big)<\delta.$$
Thus, the proof of Theorem \ref{main_up} is complete.
\end{proof}

Based on Theorems \ref{main-s} and \ref{main_up}, we immediately obtain the convergence relation of invariant measures between \eqref{ob-1'}
 and \eqref{ob-1} as $\epsilon \to 0^+$
if the measure attractors of both of them are singletons.

\begin{cor}\label{ip_con}
Under the hypotheses of Theorems \ref{main-s} and \ref{main_up}, if $f^\epsilon, \sigma^\epsilon$ and $f, \sigma$ are all time independent,
then the unique invariant measure
of \eqref{ob-1'} converges to that of \eqref{ob-1} as $\epsilon \to 0^+$.
 \end{cor}

\bibliographystyle{plain}


\end{document}